\documentclass[a4paper]{amsart}

\pdfoutput=1

\usepackage{amsmath,amssymb,amsthm}
\usepackage[utf8]{inputenc}
\usepackage[T1]{fontenc}
\usepackage{libertine}
\usepackage[libertine,cmintegrals,cmbraces]{newtxmath}
\usepackage{scalerel}
\usepackage{multicol}
\usepackage{enumitem}
\usepackage{comment} 
\usepackage{microtype}
\usepackage{xcolor}
\definecolor{darkred}{RGB}{139,0,0}
\definecolor{darkblue}{RGB}{0,0,139}
\definecolor{darkgreen}{RGB}{0,100,0}
\usepackage[linktocpage]{hyperref}
\hypersetup{
  colorlinks   = true, 
  urlcolor     = darkred, 
  linkcolor    = darkred, 
  citecolor   = darkgreen}
\usepackage{tikz}
\usepackage{tikz-cd}
\usepackage[capitalise]{cleveref}

\usepackage{stmaryrd}

\setenumerate[1]{label=(\roman*),}

\setitemize[1]{label=\raisebox{0.25ex}{\tiny$\bullet$}}

\tikzset{
  symbol/.style={
    draw=none,
    every to/.append style={
      edge node={node [sloped, allow upside down, auto=false]{$#1$}}}}}

\AtBeginDocument{%
   \def\MR#1{}}

\newcommand{\circled}[1]{\raisebox{.5pt}{\textcircled{\raisebox{-.9pt} {#1}}}}

\newcommand\smallsquare{\mathbin{\text{\raise0.17ex\hbox{\scalebox{.7}{$\blacksquare$}}}}}

\newcommand{\id}{\ensuremath{\mathrm{id}}}
\newcommand{\im}{\ensuremath{\mathrm{im}}}

\newcommand{\SO}{\ensuremath{\mathrm{SO}}}

\newcommand{\BO}{\ensuremath{\mathrm{BO}}}
\newcommand{\EO}{\ensuremath{\mathrm{EO}}}
\newcommand{\BSO}{\ensuremath{\mathrm{BSO}}}

\newcommand{\Diff}{\ensuremath{\mathrm{Diff}}}
\newcommand{\Diffor}{\ensuremath{\mathrm{Diff}^{\scaleobj{0.8}{+}}}}

\newcommand{\BlockDiff}{\ensuremath{\widetilde{\mathrm{Diff}}}}
\newcommand{\BlockMaps}{\ensuremath{\widetilde{\mathrm{Map}}}}
\newcommand{\BC}{\ensuremath{\mathrm{BC}}}
\newcommand{\BDiff}{\ensuremath{\mathrm{BDiff}}}
\newcommand{\BlockBDiff}{\ensuremath{\mathrm{B\widetilde{Diff}}}}
\newcommand{\Bun}{\ensuremath{\mathrm{Bun}}}
\newcommand{\BlockBun}{\ensuremath{\mathrm{\widetilde{Bun}}}}
\newcommand{\hAut}{\ensuremath{\mathrm{hAut}}}
\newcommand{\BhAut}{\ensuremath{\mathrm{BhAut}}}
\newcommand{\BlockhAut}{\ensuremath{\mathrm{\widetilde{hAut}}}}
\newcommand{\BlockBhAut}{\ensuremath{\mathrm{B\widetilde{hAut}}}}

\newcommand{\interior}[1]{\ensuremath{\mathrm{int}(#1)}}

\newcommand{\Iso}{\ensuremath{\mathrm{Iso}}}

\newcommand{\catsingle}[1]{\ensuremath{\mathcal{#1}}}

\newcommand{\oC}{\ensuremath{\mathrm{C}}}

\newcommand{\oH}{\ensuremath{\mathrm{H}}}
\newcommand{\oK}{\ensuremath{\mathrm{K}}}
\newcommand{\oO}{\ensuremath{\mathrm{O}}}
\newcommand{\oB}{\ensuremath{\mathrm{B}}}

\newcommand{\bfF}{\ensuremath{\mathbf{F}}}

\newcommand{\bfL}{\ensuremath{\mathbf{L}}}

\newcommand{\bfN}{\ensuremath{\mathbf{N}}}

\newcommand{\bfR}{\ensuremath{\mathbf{R}}}
\newcommand{\bfZ}{\ensuremath{\mathbf{Z}}}
\newcommand{\bfQ}{\ensuremath{\mathbf{Q}}}
\newcommand{\bfS}{\ensuremath{\mathbf{S}}}

\newcommand{\bfHZ}{\ensuremath{\mathbf{HZ}}}
\newcommand{\bfHF}{\ensuremath{\mathbf{HF}}}

\newcommand{\cC}{\ensuremath{\catsingle{C}}}

\newcommand{\cL}{\ensuremath{\catsingle{L}}}
\newcommand{\cM}{\ensuremath{\catsingle{M}}}
\newcommand{\cN}{\ensuremath{\catsingle{N}}}

\newcommand{\cS}{\ensuremath{\catsingle{S}}}

\newcommand{\Maps}{\ensuremath{\mathrm{Maps}}}

\newcommand{\hofib}{\ensuremath{\mathrm{hofib}}}
\newcommand{\colim}{\ensuremath{\mathrm{colim}}}

\newcommand{\hocolim}{\ensuremath{\mathrm{hocolim}}}

\newcommand{\ra}{\rightarrow}
\newcommand{\la}{\leftarrow}
\newcommand{\lra}{\longrightarrow}
\newcommand{\lla}{\longleftarrow}
\newcommand{\xra}[1]{\xrightarrow{#1}}
\newcommand{\xlra}[1]{\overset{#1}{\longrightarrow}}
\newcommand{\xlrahook}[1]{\overset{#1}{\lhook\joinrel\longrightarrow}}
\newcommand{\xlla}[1]{\overset{#1}{\longleftarrow}}
\newcommand{\longhookrightarrow}{\lhook\joinrel\longrightarrow}

\newcommand{\BGL}{\mathrm{BGL}}

\newcommand{\GL}{\mathrm{GL}}

\newcommand{\Aut}{\mathrm{Aut}}

\newcommand{\Der}{\mathrm{Der}}
\newcommand{\Hom}{\mathrm{Hom}}
\newcommand{\Sym}{\mathrm{Sym}}

\newcommand{\ext}{\mathrm{ext}}

\newcommand{\sfr}{\mathrm{sfr}}
\newcommand{\inc}{\mathrm{inc}}
\newcommand{\pr}{\mathrm{pr}}
\newcommand{\ev}{\mathrm{ev}}

\newcommand{\Wh}{\mathrm{Wh}}

\newcommand{\BG}{\mathrm{BG}}

\newcommand{\G}{\mathrm{G}}

\newcommand{\TOP}{\mathrm{Top}}
\newcommand{\BTOP}{\mathrm{BTop}}

\newcommand{\Mod}[1]{\ (\mathrm{mod}\ #1)}

\newtheorem{bigthm}{Theorem}
\newtheorem{bigcor}[bigthm]{Corollary}

\newtheorem{thm}{Theorem}[section]

\newtheorem{lem}[thm]{Lemma}
\newtheorem{prop}[thm]{Proposition}
\newtheorem{cor}[thm]{Corollary}

\theoremstyle{definition}

\newtheorem*{notation}{Notation}
\theoremstyle{remark}
\newtheorem{ex}[thm]{Example}

\newtheorem{rem}[thm]{Remark}

\newtheorem*{nrem}{Remark}

\newcommand{\mapnoname}[4]{\ensuremath{\begin{array}{rcl} 
      #1 & \longrightarrow & #2 \\[0.3em] 
      #3 & \longmapsto & #4
    \end{array}}}
\newcommand{\map}[5]{\ensuremath{#1\colon\begin{array}{rcl} 
      #2 & \longrightarrow & #3 \\[0.3em] 
      #4 & \longmapsto & #5
    \end{array}}}   

\begin{document}

\title{A homological approach to pseudoisotopy theory. I}
\author{Manuel Krannich}
\email{krannich@dpmms.cam.ac.uk}
\address{Centre for Mathematical Sciences, Wilberforce Road, Cambridge CB3 0WB, UK}

\begin{abstract}
We construct a zig-zag from the once delooped space of pseudoisotopies of a closed $2n$-disc to the once looped algebraic $K$-theory space of the integers and show that the maps involved are $p$-locally $(2n-4)$-connected for $n>3$ and large primes $p$. The proof uses the computation of the stable homology of the moduli space of high-dimensional handlebodies due to Botvinnik--Perlmutter and is independent of the classical approach to pseudoisotopy theory based on Igusa's stability theorem and work of Waldhausen. Combined with a result of Randal-Williams, one consequence of this identification is a calculation of the rational homotopy groups of $\mathrm{BDiff}_\partial(D^{2n+1})$ in degrees up to $2n-5$.
\end{abstract}

\maketitle
The homotopy type of the group of \emph{pseudoisotopies}, or \emph{concordance diffeomorphisms},
\[\oC(M)=\{\phi\colon M\times [0,1]\xlra{\cong} M\times [0,1]\mid \phi|_{M\times\{0\}\cup\partial M\times [0,1]}=\id\},\] of a smooth compact $d$-dimensional manifold $M$ in the smooth topology has been an object of interest to geometric topologists for many years, not least because of its intimate connection to algebraic $K$-theory already visible on the level of path components. 
Building on Cerf's proof that $\oC(M)$ is connected if $M$ is simply connected and $d\ge5$ \cite{Cerf}, Hatcher and Wagoner \cite{HatcherWagoner} computed the group $\pi_0\oC(M)$ of isotopy classes of concordances in high dimensions by relating it to the lower algebraic $K$-groups of the integral group ring $\bfZ[\pi_1M]$ of the fundamental group of $M$.\footnote{See also Igusa's corrections in \cite{IgusaWhatHappens}.} Beyond its components, the homotopy type of the space of concordances $\oC(M)$ and its relation to $K$-theory has so far been studied in two steps: deep work of Igusa \cite{IgusaStability} shows that the \emph{stabilisation map} 
\[\oC(M)\lra \oC(M\times [0,1])\] induced by crossing with an interval is $\min(\frac{d-4}{3},\frac{d-7}{2})$-connected, so in this range up to about a third of the dimension, one may consider the \emph{stable concordance space} $\colim_k\oC(M\times [0,1]^k)$ instead, which in turn admits a complete description in terms of Waldhausen's generalised algebraic $K$-theory for spaces by Waldhausen, Jahren, and Rognes' foundational \emph{stable parametrised $h$-cobordism theorem} \cite{WJR}.

In this work, we focus on the case $M=D^{2n}$ of a closed disc of even dimension and study its space of concordances via a new route---independent of the classical approach---which does not involve stabilising the dimension and is, vaguely speaking, \emph{homological} instead of \emph{homotopical}; we shall elaborate on this at a later point. Our main result relates the delooped concordance space $\BC(D^{2n})$ to the once looped algebraic $K$-theory space of the integers $\Omega^{\infty+1}\oK(\bfZ)$ in a range up to approximately the dimension, $p$-locally for primes $p$ that are large with respect to the dimension and the degree.

\begin{bigthm}\label{thm:mainthm}For $n>3$, there exists a zig-zag \[\BC(D^{2n})\lra\cdot\lla \Omega^{\infty+1}_0\oK(\bfZ)\] whose maps are $p$-locally $\min(2n-4,2p-4-n)$-connected for primes $p$.\end{bigthm}

\begin{nrem}The result we prove is slightly stronger than stated here (see \cref{thm:mainthmtechnical}) and implies for instance that $\pi_{2n-4}\BC(D^{2n})\otimes\bfQ$ surjects onto $K_{2n-3}(\bfZ)\otimes\bfQ$ as long as $n>3$.
\end{nrem}

When combined with Borel's work on the stable rational cohomology of arithmetic groups \cite{BorelI}, \cref{thm:mainthm} provides an isomorphism
\[\pi_*\BC(D^{2n})\otimes \bfQ\cong K_{*+1}(\bfZ)\otimes\bfQ\cong\begin{cases}\bfQ&\text{if }*\equiv0\Mod{4}\\0&\text{otherwise}\end{cases}\quad\text{ for }0<*<2n-4,\] and an epimorphism $\pi_*\BC(D^{2n})\otimes \bfQ\ra K_{*+1}(\bfZ)\otimes\bfQ$ in degree $2n-4$, which goes significantly beyond the range that was previously accessible by relying on Igusa's stability result and shows for instance that $\BC(D^{8})$ is nontrivial, even rationally. Given that the $K$-groups $K_{*}(\bfZ)$ are known to contain $p$-torsion for comparatively large primes with respect to the degree due to contributions from Bernoulli numbers, \cref{thm:mainthm} also exhibits many new torsion elements in $\pi_*\BC(D^{2n})$, such as one of order $691$ in $\pi_{21}\BC(D^{2n})$ as long as $n>12$ resulting from the fact that $K_{22}(\bfZ)$ is cyclic of that order.

\bigskip

In the remainder of this introduction, we explain more direct applications of \cref{thm:mainthm} and conclude by indicating some of the ideas that go into its proof.

\subsection*{Diffeomorphisms and concordances of odd discs}
Restricting a concordance to the moving part of its boundary induces a homotopy fibre sequence \[\Diff_\partial(D^{d+1})\lra \oC(D^{d})\lra \Diff_\partial(D^{d})\] that compares the group of concordances $\oC(D^{d})$ of a $d$-disc to its group of its diffeomorphisms $\Diff_\partial(D^d)$ fixing the boundary pointwise. By a result of Randal-Williams \cite[Thm\,4.1]{RWUpperBound} based on Morlet's lemma of disjunction and work of Berglund and Madsen \cite{BerglundMadsen} (a combination which incidentally inspired parts of our strategy to prove \cref{thm:mainthm}), the space $\BDiff_\partial(D^{2n})$ is rationally $(2n-5)$-connected, so the delooped maps
\[\BDiff_\partial(D^{2n+1})\lra\BC(D^{2n})\quad\text{and}\quad\BC(D^{2n+1})\lra\BDiff_\partial(D^{2n+1})\] are rationally $(2n-5)$-connected as well, resulting in the following corollary of \cref{thm:mainthm}.

\begin{bigcor}\label{cor:odddiscandconc}There exist isomorphisms
\[\pi_*\BDiff_\partial(D^{2n+1})\otimes\bfQ\cong\pi_*\BC(D^{2n+1})\otimes\bfQ\cong K_{*+1}(\bfZ)\otimes\bfQ\] in degrees $*<2n-5$ and epimorphisms \[\pi_*\BDiff_\partial(D^{2n+1})\otimes\bfQ\lra K_{*+1}(\bfZ)\otimes\bfQ\quad\text{ and }\quad\pi_*\BC(D^{2n+1})\otimes\bfQ\lra K_{*+1}(\bfZ)\otimes\bfQ\] in degree $2n-5$.
\end{bigcor}

\begin{nrem}\ \begin{enumerate}\item The range in \cref{cor:odddiscandconc} is nearly optimal: by work of Watanabe \cite{WatanabeI} or as consequence of Weiss' results on topological Pontryagin classes \cite{WeissDalian}, the group $\pi_{2n-2}\BDiff_\partial(D^{2n+1})\otimes\bfQ$ is known to be nontrivial for many values of $n$ for which $K_{2n-1}(\bfZ)\otimes\bfQ$ vanishes by Borel's work mentioned above. 
\item \label{item:improve-with-KRW}A combination of the strengthening of \cref{thm:mainthm} mentioned earlier with recent work of Kupers and Randal-Williams \cite{KRWdiscs} improves the range of \cref{cor:odddiscandconc} by one degree from $2n-5$ to $2n-4$.
\item In the range captured by Igusa's stability result, i.e.\,up to approximately degree $2n/3$, \cref{cor:odddiscandconc} was previously known as a result of a classical computation due to Farrell and Hsiang \cite{FarrellHsiang} based on Waldhausen's approach to pseudoisotopy theory (of which the proof of \cref{cor:odddiscandconc} is independent).
\end{enumerate}
\end{nrem}

\subsection*{Homeomorphisms of Euclidean spaces} By an enhancement of a result due to Morlet (see e.g.\ \cite[Thm 4.4]{BurgheleaLashofSmoothing}), there are homotopy equivalences \[\begin{gathered}\BDiff_\partial(D^d)\simeq\Omega^d_0\TOP(d)/\oO(d)\quad\text{and}\\\BC(D^d)\simeq \Omega^d\hofib\Big(\TOP(d)/\oO(d)\ra \TOP(d+1)/\oO(d+1)\Big)\end{gathered}\] for $d\ge5$ that relate the groups of diffeomorphisms and concordances of a $d$-disc to the homotopy fibre $\TOP(d)/\oO(d)$ of the map $\BO(d)\ra\BTOP(d)$ that classifies the inclusion of the orthogonal group $\oO(d)$ into the topological group $\TOP(d)$ of homeomorphisms of $\bfR^d$, and to its stabilisation map $\TOP(d)/\oO(d)\ra \TOP(d+1)/\oO(d+1)$ induced by taking products with the real line. \cref{thm:mainthm} and \cref{cor:odddiscandconc} can thus be reformulated in terms of these equivalent spaces and result in particular in the following corollary, using the fact that the groups $\pi_*\TOP(d)/\oO(d)$ are finite for $*<d+2$ and $d\neq 4$ \cite[Essay V, 5.0]{KirbySiebenmann}.

\begin{bigcor}\label{thm:TopO}There exists an isomorphism
\[\pi_*\TOP(2n+1)/\oO(2n+1)\otimes \bfQ\cong \begin{cases}0&\mbox{for }*\le 2n+1\\K_{*+2n+2}(\bfZ)\otimes\bfQ&\mbox{for }2n+1<*< 4n-4\end{cases}\]
and an epimorphism $\pi_*\TOP(2n+1)/\oO(2n+1)\otimes\bfQ\ra K_{*+2n+2}(\bfZ)\otimes\bfQ$ in degree $4n-4$.
\end{bigcor} 

\begin{nrem}As per item \ref{item:improve-with-KRW} of the previous remark, this range can be improved by one degree.
\end{nrem}

\subsection*{Idea of proof}Instead of sketching the proof of \cref{thm:mainthm}, we outline a strategy to achieve a seemingly different task: relating the $p$-local \emph{homology} of $\BC(D^{2n})$ to $K$-theory in a range of degrees. This should, however, convey the main ideas; the actual proof of \cref{thm:mainthm} uses a similar strategy to construct a zig-zag between $\BC(D^{2n})$ and $\Omega^{\infty+1}_0\oK(\bfZ)$ that consists of maps that are $p$-local homology isomorphisms in a range and then argues that the maps are actually $p$-locally highly connected. In the sketch that follows, we allow ourselves to be somewhat vague; full details shall be given in the body of this work.

The root of the proof of \cref{thm:mainthm} is to consider the odd-dimensional disc $D^{2n+1}$ as the $0$th member of a whole family of manifolds---the high-dimensional handlebodies
\[V_g\coloneq \natural^gD^{n+1}\times S^n\]
given as iterated boundary connected sums of $D^{n+1}\times S^n$.
Comparing the groups of diffeomorphisms $\Diff_{D^{2n}}(V_g)$ that pointwise fix a chosen disc $D^{2n}\subset \partial V_g$ in the boundary to the corresponding block diffeomorphism groups yields homotopy fibre sequences
\[\BlockDiff_{D^{2n}}(V_g)/\Diff_{D^{2n}}(V_g)\lra \BDiff_{D^{2n}}(V_g)\lra \BlockBDiff_{D^{2n}}(V_g),\] one for each $g$. Varying $g$, these fibre sequences are connected by stabilisation maps induced by extending (block) diffeomorphisms along the inclusion $V_g\subset V_{g+1}$ by the identity, and Morlet's lemma of disjunction ensures that the map between homotopy fibres \[\BlockDiff_{D^{2n}}(V_g)/\Diff_{D^{2n}}(V_g)\lra \BlockDiff_{D^{2n}}(V_{g+1})/\Diff_{D^{2n}}(V_{g+1})\] is highly connected. It is not hard to see that the space $\BC(D^{2n})$ of interest is equivalent to this fibre for $g=0$, so to access $\BC(D^{2n})$ in a range, we may as well study the homotopy fibre of the sequence obtained from the previous one by taking homotopy colimits,
\begin{equation}\label{equ:stablefibreseqeuenceVg}\BlockDiff_{D^{2n}}(V_\infty)/\Diff_{D^{2n}}(V_\infty)\lra \BDiff_{D^{2n}}(V_\infty)\lra \BlockBDiff_{D^{2n}}(V_\infty).\end{equation} By work of Botvinnik and Perlmutter \cite{BotvinnikPerlmutter}, the homology of $\BDiff_{D^{2n}}(V_\infty)$ has a surprisingly simple description in homotopy theoretical terms, so to compute the homology of the fibre of \eqref{equ:stablefibreseqeuenceVg} and hence that of $\BC(D^{2n})$ in a range, one might try to compute the homology of $\BlockBDiff_{D^{2n}}(V_\infty)$ and analyse the Serre spectral sequence of \eqref{equ:stablefibreseqeuenceVg}. This is essentially what we do, and it involves several steps of which some might be of independent interest:
\begin{enumerate}
\item In \cref{sect:surgerytheory}, we use surgery theory to express the space of block diffeomorphisms of a general manifold triad satisfying a $\pi$-$\pi$-condition $p$-locally for large primes in terms of its homotopy automorphisms covered by certain bundle data. A similar result in the rational setting which inspired ours but applies to another class of triads was obtained by Berglund and Madsen \cite{BerglundMadsen} (see also \cref{rem:BerglundMadsenComparison}).
\item \cref{section:mcg} serves to compute variants of the mapping class group $\pi_0\Diff_{D^{2n}}(V_g)$ up to extensions in terms of automorphisms groups of the integral homology of $V_g$.
\item In \cref{sect:hAut}, we calculate the $p$-local homotopy and homology groups of the delooped space of homotopy automorphisms $\BhAut_{D^{2n}}(V_g,W_{g,1})$ of $V_g$ that fix $D^{2n}$ and restrict to a homotopy automorphism of the complement of the boundary as a module over the group $\pi_0\hAut_{D^{2n}}(V_g,W_{g,1})$ in a range of degrees. This uses some pieces of the apparatus of rational homotopy theory, as well as an ad-hoc $p$-local generalisation we provide along the way.
\item The action on the $n$th homology group $\oH_n(V_g;\bfZ)\cong\bfZ^g$ induces a map of the form
\[\BlockBDiff_{D^{2n}}(V_\infty)\lra \BGL_\infty(\bfZ)^+\simeq\Omega^{\infty}_0 \oK(\bfZ),\] a variant of which we show in \cref{sect:proofmainthm} with the help of all previous steps to be a $p$-local homology isomorphism in a range of degrees.\end{enumerate}

\subsection*{Outlook}
In \cite{conc2}, we will take a different approach and study concordance spaces $\oC(M)$ without restriction on the $p$-torsion. As a byproduct, the setup of \cite{conc2} will also make apparent that the zig-zag of \cref{thm:mainthm} is compatible with the iterated stabilisation map $\oC(D^{2n})\ra\oC(D^{2n}\times [0,1]^2)\simeq \oC(D^{2n+2})$ and agrees up to equivalence with the zig-zag 
\[\BC(D^{2n})\ra \Omega^{\infty+1}_0\Wh^{\Diff}(*)=\Omega^{\infty}_0\hofib\big(\bfS\ra K(\bfS)\big)\ra\Omega_0^{\infty}\hofib\big(\bfS\ra K(\bfZ)\big)\la \Omega_0^{\infty+1}K(\bfZ)\]
known from Waldhausen's work on pseudoisotopy theory \cite{WaldhausenManifold} (see also \cref{sec:reformulation}). 

\subsection*{Acknowledgements}
My thanks go to Oscar Randal-Williams for several valuable discussions and to a referee for their constructive feedback. I was partially supported by O.\,Randal-Williams' Philip Leverhulme Prize from the Leverhulme Trust and the European Research Council (ERC) under the European Union’s Horizon 2020 research and innovation programme (grant agreement No. 756444).

\tableofcontents

\section{Preliminaries}\label{section:prelim}
We start off with a lemma on semi-simplicial actions and a short recollection on nilpotent spaces for later reference, followed by foundational material on various types automorphisms of manifolds with bundle data. Primarily, this serves us to set up a convenient theory of block automorphism spaces with tangential structures.

\subsection{Semi-simplicial monoids and their actions}
\label{sect:semisimplicial}The \emph{homotopy quotient} of a semi-simplicial set $X_\bullet$ semi-simplicially acted upon by a semi-simplicial monoid $M_\bullet$ from the right is the semi-simplicial space $X_\bullet\sslash M_\bullet$ whose space of $p$-simplices is defined as the bar-construction $\oB(X_p,M_p,*)$, with face maps induced by the face maps of $M_\bullet$ and $X_\bullet$. For $X_\bullet=*_\bullet$ the semi-simplicial point, i.e.\,$*_p$ a singleton for all $p$, we abbreviate $X_\bullet\sslash M_\bullet$ by $\oB M_\bullet$. The unique semi-simplicial map $X_\bullet\ra *_\bullet$ induces a natural map $X_\bullet\sslash M_\bullet\ra \oB M_\bullet$ which is well-known to geometrically realise to a quasi-fibration with fibre the realisation of $X_\bullet$ if $M_\bullet$ is a group-like simplicial monoid acting simplicially on a simplicial set $X_\bullet$. To explain a generalisation of this fact for \emph{semi}-simplicial $M_\bullet$ and $X_\bullet$, we denote the geometric realisation by $\lvert-\rvert$ and consider the natural zig-zag
\[ \lvert X_\bullet\rvert\times \lvert M_\bullet\rvert\lla\lvert X_\bullet\times M_\bullet\rvert\lra\lvert X_\bullet\rvert\] whose left map is induced by the projections and the right map by the action. If the underlying semi-simplicial sets of $M_\bullet$ and $X_\bullet$ \emph{admit degeneracies}, i.e.\,if they agree (as semi-simplicial sets) with the underlying semi-simplicial sets of simplicial sets, then the left arrow is an equivalence (see e.g.\,\cite[Thm\,7.2]{EbertRWsimplicial}), so a contractible choice of a homotopy inverse yields an action map $\mu\colon \lvert X_\bullet\rvert\times \lvert M_\bullet\rvert\lra \lvert X_\bullet\lvert$. In this situation, we say that \emph{$M_\bullet$ acts on $X_\bullet$ by equivalences} if $\mu(-,m)\colon \lvert X_\bullet\lvert\ra \lvert X_\bullet\lvert$ is an equivalence for all $m\in \lvert M_\bullet\rvert$. The following lemma shows that this is sufficient to conclude that the natural map $X_\bullet\sslash M_\bullet\ra \oB M_\bullet$ realises to a quasi-fibration.

\begin{lem}\label{lem:simplicialaction}For a semi-simplicial monoid $M_\bullet$ acting on a semi-simplicial set $X_\bullet$ such that $M_\bullet$ and $X_\bullet$ admit degeneracies and the action of $M_\bullet$ on $X_\bullet$ is by equivalences, the sequence
\[X_\bullet\lra X_\bullet \sslash M_\bullet\lra \oB M_\bullet\]
induces a quasi-fibration on geometric realisations.
\end{lem}  

\begin{rem}\label{rem:simplicial-remark}\ 
\begin{enumerate}
\item In all situations we encounter, the condition that $M_\bullet$ and $X_\bullet$ admit degeneracies is ensured by them being Kan (every semi-simplicial Kan complex admits degeneracies  \cite{KanSemisimplicial}) and the condition that $M_\bullet$ acts on $X_\bullet$ by equivalences by $M_\bullet$ being group-like, i.e.\,the monoid of components $\pi_0\lvert M_\bullet\rvert$ having inverses. Note that \cref{lem:simplicialaction} neither requires the degeneracies of $M_\bullet$ to be compatible with the monoid structure nor those of $X_\bullet$ with the action. 
\item The condition that $M_\bullet$ and $X_\bullet$ admit degeneracies is crucial, even if $M_\bullet=G_\bullet$ is a simplicial group. For an instructive example, consider the semi-simplicial set $X_\bullet=G_\bullet^{\le 0}$ which agrees with $G_0$ in degree $0$ and is empty otherwise, semi-simplicially acted upon by $G_\bullet$ via right translations. In this case, the realisation $\lvert X_\bullet \sslash G_\bullet\rvert$ is contractible, but $\rvert G_\bullet^{\le 0}\rvert\simeq G_0$ is rarely equivalent to $\Omega \lvert\oB G_\bullet\rvert\simeq  \lvert G_\bullet\rvert$.
\item\label{item:simplicial-remark-final}Under the assumption of \cref{lem:simplicialaction}, the long exact sequence induced by the quasi-fibration in particular yields a canonical bijection $\pi_0\lvert X_\bullet\rvert / \pi_0\lvert M_\bullet\rvert\cong \pi_0\lvert X_\bullet \sslash M_\bullet\rvert$.
\end{enumerate}
\end{rem}
\begin{proof}[Proof of \cref{lem:simplicialaction}]
The sequence in question is the geometric realisation of a sequence
\begin{equation}\label{equ:semisimplicialsequence}X_{\bullet}\lra \oB_{\smallsquare}(X_\bullet,M_\bullet,*)\lra \oB_{\smallsquare}M_\bullet\end{equation} of simplicial semi-simplicial sets, where the simplicial (bar-)direction is indicated by the square $\smallsquare$ and the semi-simplicial direction by the bullet $\bullet$; the fibre $X_\bullet$ is constant in the $\smallsquare$-direction. Since the realisation of a simplicial set is canonically equivalent to the realisation of its underlying semi-simplicial set \cite[Lem.\,1.7]{EbertRWsimplicial} and the realisation of a bi-semi-simplicial sets is independent of which direction one realises first \cite[p.\,2106]{EbertRWsimplicial}, we may rely on a result of Segal \cite[Thm\,2.12]{EbertRWsimplicial} and the second part of its simplification from \cite[Lem.\,2.11]{EbertRWsimplicial} to reduce the claim to showing that the commutative squares 
\[
\begin{tikzcd}
\lvert\oB_{p}(X_\bullet,M_\bullet,*)\rvert\arrow[r,"d_p"]\dar&\lvert\oB_{p-1}(X_\bullet,M_\bullet,*)\rvert\dar\\
\lvert\oB_{p}M_\bullet\rvert\arrow[r,"d_p"]& \lvert\oB_{p-1}M_\bullet\rvert
\end{tikzcd}
\ \ and\ \ 
\begin{tikzcd}
\lvert\oB_{1}(X_\bullet,M_\bullet,*)\rvert\arrow[r,"d_0"]\dar&\lvert\oB_{0}(X_\bullet,M_\bullet,*)\rvert\dar\\
\lvert\oB_{1}M_\bullet\rvert\arrow[r,"d_0"]& \lvert\oB_{0}M_\bullet\rvert
\end{tikzcd}
\]
obtained from the simplicial structure of $\oB_{\smallsquare}(X_\bullet,M_\bullet,*)\ra \oB_{\smallsquare}(M_\bullet)$ by realising the $\bullet$-direction are homotopy cartesian for $p\ge0$. For the left square, this follows directly from the definition of the bar-construction together with the above mentioned fact that the realisation of the product of two semi-simplicial sets that admit degeneracies is canonically equivalent to the product of their realisations. Spelling out the definition of the bar-construction, the right hand square is given as 
\[
\begin{tikzcd}
\lvert X_\bullet\times M_\bullet\rvert\arrow[d,swap,"\lvert (\pr_2)_\bullet\rvert"]\arrow[r,"\lvert \mu_\bullet\rvert"]\dar&\lvert X_\bullet \rvert\dar\\
\lvert M_\bullet\rvert\arrow[r]& *,
\end{tikzcd}
\]
involving the action $\mu_\bullet$ and the projection $(\pr_2)_\bullet$ on the second coordinate, so we are left to show that the \emph{shear map} $(\lvert \mu_\bullet\rvert, \lvert (\pr_2)_\bullet\rvert)\colon \lvert X_\bullet \times M_\bullet\rvert\ra \lvert X_\bullet\rvert\times \lvert M_\bullet\rvert$
is an equivalence. This map fits into a commutative triangle
\[
\begin{tikzcd}
\lvert X_\bullet \times M_\bullet\rvert\arrow[rr,"{(\lvert \mu_\bullet\rvert, \lvert (\pr_2)_\bullet\rvert)}"]\arrow[dr,bend right = 15, "{\lvert (\pr_2)_\bullet\rvert}",swap]&&\lvert X_\bullet\rvert\times \lvert M_\bullet\rvert\arrow[dl, bend left = 15, "{\lvert (\pr_2)_\bullet\rvert}"]\\
&\lvert M_\bullet\rvert&
\end{tikzcd}
\]
whose map on diagonal homotopy fibres at a point $m\in \lvert M_\bullet\rvert$ agrees up to equivalence with the action map $\mu(-,m)\colon\lvert X_\bullet\rvert\ra \lvert X_\bullet\rvert$. The latter is an equivalence by assumption, so the shear map is an equivalence and the claim follows.
\end{proof}

\subsection{Nilpotent spaces}\label{section:nilpotency}A space is \emph{nilpotent} if it is path connected and its fundamental group is nilpotent and acts nilpotently on all higher homotopy groups. Such spaces have an unambiguous $p$-localisation at a prime $p$, which on homology and homotopy groups (including the fundamental group) has the expected effect of $p$-localisation in the algebraic sense \cite[Thm 6.1.2]{MayPonto}. Localisations are defined in terms of a universal property \cite[Def.\,5.2.3]{MayPonto}, which ensures that they are unique and functorial up to homotopy. A map between nilpotent spaces is \emph{$p$-locally $k$-connected} for a prime $p$ and $k\ge1$ if the induced map on $p$-localisations is $k$-connected in the usual sense.

\begin{lem}\label{lem:plocalconnectivity}For a map $X\ra Y$ between nilpotent spaces, a prime $p\ge2$, and $k\ge1$, the following statements are equivalent.
\begin{enumerate}
\item The map $X\ra Y$ is $p$-locally $k$-connected.
\item The induced map on $p$-localised homotopy groups $(\pi_*X)_{(p)}\to(\pi_*Y)_{(p)}$ is an isomorphism for $*<k$ and surjective for $*=k$.
\item The induced map on $p$-local homology groups $\oH_*(X;\bfZ_{(p)})\to \oH_*(Y;\bfZ_{(p)})$ is an isomorphism for $*< k$ and surjective for $*=k$.
\end{enumerate}
\end{lem}
\begin{proof}The above mentioned fact that $p$-localisation of nilpotent spaces commutes with taking homotopy groups shows that the first two items are equivalent. The equivalence between the second two follows for $k=1$ from the well-known fact that a morphism $G\to H$ between nilpotent groups is surjective if and only if it is surjective on abelianisations which one sees as follows: writing $\Gamma^1(G)=G$ and $\Gamma^n(G)=[\Gamma^{n-1}(G),G]$ for the lower central series, there is a natural epimorphism $\otimes^n_{\bfZ}G^{\mathrm{ab}}\ra \Gamma^{n}(G)/\Gamma^{n+1}(G)$ induced by taking left-normed commutators (see e.g.\,\cite[Thm 3.1]{Warfield}), so one may prove the claim by an induction on the nilpotence degree. For $k\ge2$, the equivalence between the second two items is a consequence of the relative Hurewicz theorem for nilpotent spaces (see e.g.\,\cite[Cor.\,3.4]{HiltonRoitberg}) applied to the $p$-localisation of the map in question.\end{proof}

At several points in this work, we will make use of the fact that nilpotent spaces behave well with respect to taking homotopy fibres (see e.g.\,\cite[Prop.\,4.4.1]{MayPonto} for a proof).

\begin{lem}\label{lem:nilpotentfibre}Let $\pi\colon E\ra B$ be a fibration, $e\in E$ a point, and $E_e\subset E$ and $F_e\subset \pi ^{-1}(\pi(e))$ the respective path components containing $e$. If $E_e$ is nilpotent, then $F_e$ is nilpotent as well.
\end{lem}

\subsection{Block spaces}\label{sect:blockspaces}
As is customary, we denote the standard $p$-simplex by $\Delta^p\subset\bfR^{p+1}$ and identify its faces $\sigma\colon \Delta^q\ra\Delta^p$ with their images. A \emph{$p$-block space} is a space $X$ together with a map $\pi\colon X\ra\Delta^p$. For a $q$-face $\sigma\subset \Delta^p$, the preimage $X_\sigma\coloneq\pi^{-1}(\sigma)$ becomes a $q$-block space by pulling back $\pi$ along $\sigma$. A \emph{block map} between block $p$-spaces is a map $f\colon X\ra Y$ between underlying spaces such that $f(X_\sigma)\subset Y_{\sigma}$ for all faces $\sigma\subset \Delta^p$ and a \emph{block homotopy equivalence} is a block map  $f\colon X\ra Y$ such that the induced map $f_\sigma\colon X_\sigma\ra Y_\sigma$ is a homotopy equivalence for all faces $\sigma\subset\Delta^p$. Spaces of the form $\Delta^p\times M$ are implicitly considered as $p$-block spaces via the projection. For $i=0,\ldots,p$, the map
\[\map{c_{i}}{[0,1]\times \Delta^{p-1}}{\Delta^p}{(s,t_0,\ldots,t_{p-1})}{((1-s)t_0,\ldots,(1-s)t_{i-1},s,(1-s)t_{i},\ldots,(1-s)t_{p-1} )}\]
induces by restriction for $0<\epsilon\le 1$ a diffeomorphism of manifolds with corners
\[c_{i,\epsilon}\colon [0,\varepsilon)\times \Delta^{p-1}\lra  c_i([0,\epsilon)\times \Delta^{p-1})\eqcolon\Delta^p_{i,\epsilon}\] onto an open neighborhood $\Delta^p_{i,\epsilon}$ of the $i$th codimension $1$ face $\Delta^p_i\subset\Delta^p$. A block map $f\colon \Delta^p\times M\ra \Delta^p\times N$ for spaces $M$ and $N$ is \emph{collared} if there exists an $\varepsilon>0$ such that\begin{enumerate}
\item $f(\Delta^p_{i,\epsilon}\times M)\subset \Delta^p_{i,\varepsilon}\times N$ and 
\item $(c_{i,\epsilon}^{-1}\times \id_N)\circ f|_{\Delta^p_{i,\epsilon}\times M}\circ (c_{i,\epsilon}\times \id_M)=f_{\Delta_i^p}\times \id_{[0,\epsilon)}$
\end{enumerate}
are satisfied for $i=0,\ldots,p$.

\subsection{Diffeomorphisms}\label{section:diffeomorphisms}
For a compact smooth $d$-manifold $W$ and two compact submanifolds $M,N\subset W$, we denote by $\BlockDiff_M(W,N)_\bullet$ the semi-simplicial group of \emph{block diffeomorphisms} whose $p$-simplices consists of all diffeomorphisms of $\Delta^p\times W$ that are collared block maps, fix $\Delta^p\times N$ setwise, and fix a neighborhood of $\Delta^p\times M$ pointwise. The semi-simplicial structure is induced by restricting diffeomorphisms of $\Delta^p\times M$ to $\sigma\times M$ for faces $\sigma\subset \Delta^p$. Whenever one of the submanifolds is empty, we omit it from the notation and in the case $M=\partial W$, we write $\BlockDiff_\partial(W,N)$ instead of $\BlockDiff_M(W,N)$. Making use of the collaring condition, one shows that $\BlockDiff_M(W,N)_\bullet$ satisfies the Kan property (see \cite[p.\,58-59]{HLLRW} for a proof in the case $M=N=\varnothing$; the general case follows in the same way\footnote{In \cite[App.\,A §3 a]{BLR}, the authors attempt to construct degeneracy maps for certain semi-simplicial sets of \emph{block embeddings}, which would in particular enhance $\BlockDiff_\partial(W)_\bullet$ to a simplicial group and thus imply that it satisfies the Kan property as every simplicial group does. However, the argument in \cite{BLR} is flawed: on two of the faces, the proposed degeneracy maps violate the collaring condition described on p.\,116 loc.\,cit. The more laborious argument in \cite{HLLRW} appears to be correct.}). The semi-simplicial subgroup of \emph{diffeomorphisms} \[\Diff_M(W,N)_\bullet\subset \BlockDiff_M(W,N)_\bullet\] is defined by requiring the diffeomorphisms of $\Delta^p\times M$ to commute with the projection to the simplex $\Delta^p$ instead of just preserving its faces. This semi-simplicial subgroup agrees with the (collared and smooth) singular set of the topological group $\Diff_M(W,N)$ of diffeomorphisms by which we mean the set of $0$-simplices $\Diff_M(W,N)_0$ equipped with the smooth Whitney topology, so there is a canonical weak equivalence $\lvert\Diff_M(W,N)_\bullet\rvert\ra \Diff_M(W,N)$ and we shall not distinguish between these spaces.

\subsection{Homotopy automorphisms}
\label{sect:hAutPrelim}The $p$-simplices of the semi-simplicial monoid of \emph{block homotopy automorphisms} $\BlockhAut_M(W,N)_\bullet$ are the block homotopy automorphisms of $\Delta^p\times W$ that fix $\Delta^p\times M$ pointwise and restrict to homotopy automorphisms of $\Delta^p\times N$. Unlike for $\BlockDiff_M(W,N)_\bullet$, it is straight-forward to see that $\BlockhAut_M(W,N)_\bullet$ is Kan: a map from the semi-simplicial horn $(\Lambda_i^p)_\bullet$ is represented by a homotopy equivalence $\psi\colon \Lambda_i^p\times W\ra \Lambda_j^p\times W$ and a lift to a $p$-simplex $\Delta^p\times W\ra \Delta^p\times W$ is given by $(\varphi_i\times \id_W)\circ (\id_{[0,1]}\times\psi)\circ (\varphi_i^{-1}\times \id_W)$, where $\varphi_i\colon [0,1]\times \Lambda_i^p\ra \Delta^p$ is any homeomorphism that extends the inclusion on $ \{0\}\times \Lambda_i^p\subset \Delta^p$ and restricts to a homeomorphism from $[0,1]\times \partial\Lambda_i^p\cup  \{1\}\times \Lambda_i^p$ onto the $i$th face $\Delta^p_i\subset \Delta^p$. As for diffeomorphisms, insisting that the homotopy equivalences of $\Delta^p\times M$ be over $\Delta^p$ defines a sub semi-simplicial monoid of \emph{homotopy automorphisms} 
\begin{equation}\label{equ:inclusionblockhaut}
\hAut_M(W,N)_\bullet\subset \BlockhAut_M(W,N)_\bullet,
\end{equation}
which agrees with the singular set of the space $\hAut_M(W,N)$ obtained by equipping the set of homotopy equivalences $\hAut_M(W,N)_0$ with the compact open topology, so also $\lvert \hAut_M(W,N)_\bullet\rvert$ and $\hAut_M(W,N)$ are canonically equivalent. An aspect which distinguishes the situation for homotopy automorphisms from that for diffeomorphisms is that the inclusion \eqref{equ:inclusionblockhaut} of Kan complexes induces an equivalence on geometric realisation, which one can see from the combinatorial description of their homotopy groups together with the contractibility of $\hAut_{\partial \Delta^p}(\Delta^p)$.

\subsection{Bundle maps, unstably}
\label{sect:unstablebundlemaps}
A \emph{bundle map} between two vector bundles $\xi\ra X$ and $\nu\ra Y$ over CW complexes $X$ and $Y$ is a commutative square of the form
\[
\begin{tikzcd}
\xi\rar{\phi}\dar&\nu\dar\\
X\rar{\bar{\phi}} &Y
\end{tikzcd}
\]
whose induced maps on vertical fibres are linear isomorphisms. Of course the underlying map of spaces $\bar{\phi}$ can be recovered from $\phi$, so we often omit it. Given a subcomplex $A\subset X$ and a bundle map $\ell_0\colon \xi|_{A}\ra \nu$ defined on the restriction of $\xi$ to $A$, the semi-simplicial set of \emph{block bundle maps} $\BlockBun_A(\xi,\nu;\ell_0)_\bullet$ has as its $p$-simplices the bundle maps $\Delta^p\times \xi\ra \Delta^p\times \nu$ that agree with $\id_{\Delta^p}\times \ell_0$ on $\Delta^p\times \xi|_{A}$ and whose underlying map $\Delta^p\times X\ra \Delta^p\times Y$ is a block map. As before, the semi-simplicial structure is induced by restriction to subspaces $\sigma\times X$ for faces $\sigma\subset\Delta^p$. Insisting that the underlying map between base spaces be over $\Delta^p$ defines the sub semi-simplicial set \[\Bun_A(\xi,\nu;\ell_0)_\bullet\subset \BlockBun_A(\xi,\nu;\ell_0)_\bullet\] of \emph{bundle maps}, which agrees with the singular set of the space $\Bun_A(\xi,\nu;\ell_0)$ obtained by equipping the set $\Bun_A(\xi,\nu;\ell_0)_0$ of bundle maps $\xi\ra\nu$ relative to $\ell_0$ with the compact-open topology. 
If $\xi=\nu$ and $\ell_0=\inc$ is the inclusion, then the semi-simplicial sets of (block) bundle maps $\BlockBun_A(\xi,\xi;\inc)_\bullet$ and $\Bun_A(\xi,\xi;\inc)_\bullet$ are semi-simplicial monoids under composition, and they act by precomposition on $\BlockBun_A(\xi,\nu;\ell_0)_\bullet$ respectively $\Bun_A(\xi,\nu;\ell_0)_\bullet$ for any bundle $\nu$ and bundle map $\ell_0\colon\xi|_{A}\ra\nu$. For a subcomplex $C\subset X$, we denote by \[\BlockhAut_A(\xi,C)_\bullet\subset \BlockBun_A(\xi,\xi;\inc)_\bullet\] the submonoid of those block bundle maps whose underlying selfmap of $\Delta^p\times X$ is a homotopy equivalence that restricts to an equivalence of $\Delta^p\times C$. The submonoid \[\hAut_A(\xi,C)_\bullet\subset \Bun_A(\xi;\xi,\inc)_\bullet\] is defined analogously. 

\subsubsection{Tangential bundle maps}
Introducing yet another variant of bundle maps, we define the semi-simplicial set of \emph{tangential block bundle maps} $\BlockBun_A(\xi,\nu;\ell_0)^\tau_\bullet$ as follows: writing $\tau_{M}$ for the tangent bundle of a manifold $M$ and $\varepsilon$ for the trivial line bundle, the $p$-simplices of $\BlockBun_A(\xi,\nu;\ell_0)^\tau_\bullet$ are the bundle maps $\varphi\colon\tau_{\Delta^p}\times \xi\ra \tau_{\Delta^p}\times \nu$ that agree with $\id_{\tau_{\Delta^p}}\times \ell_0$ on $\tau_{\Delta^p}\times \xi|_A$, satisfy $\varphi(\tau_{\Delta_i^p}\times \xi)\subset\tau_{\Delta_i^p}\times\nu$ for $0\le i\le p$, and make the diagram
\begin{equation}\label{equ:collarsquare}
\begin{tikzcd}[column sep=1.5cm]
\tau_{\Delta^p}|_{\Delta^p_i}\times \xi\arrow[r,"\varphi"]\arrow[d]& {\tau_{\Delta^p}}|_{\Delta^p_i}\times \nu\arrow[d]\\
(\tau_{\Delta_i^p}\oplus \varepsilon)\times \xi \arrow[r]&(\tau_{\Delta_i^p}\oplus \varepsilon)\times \nu
\end{tikzcd}
\end{equation}
commute, where the bottom horizontal map is given by the restriction of $\varphi$ on $\tau_{\Delta^p_i}\times \xi$ and the identity on $\varepsilon$, and the vertical maps are induced by the canonical trivialisation of $\tau_{[0,1]}$ and the derivative of the map $c_i$ defined in \cref{sect:blockspaces}. Requiring the underlying map $\Delta^p\times X\ra\Delta^p\times Y$ to be over $\Delta^p$ defines the sub semi-simplicial set \begin{equation}\label{equ:tangentialbundlemaps}\Bun_A(\xi,\nu;\ell_0)^\tau_\bullet\subset \BlockBun_A(\xi,\nu;\ell_0)^\tau_\bullet\end{equation} of \emph{tangential bundle maps}. As before, we have a chain of semi-simplicial monoids \[\hAut_A(\xi,C)^\tau_\bullet\subset \BlockhAut_A(\xi,C)^\tau_\bullet\subset \BlockBun_A(\xi,\xi;\inc)^\tau_\bullet\] which are defined in the same way as for non-tangential (block) bundle maps. Note that there is a canonical map 
\begin{equation}\label{equ:extedingtanegntialbundlemaps}\Bun_A(\xi,\nu;\ell_0)_\bullet\lra \Bun_A(\xi,\nu;\ell_0)_\bullet^\tau\end{equation}
 given by extending a bundle map $\Delta^p\times \xi\ra \Delta^p\times\nu$ over $\Delta^p$ to $\tau_{\Delta^p}\times \xi\ra \tau_{\Delta^p}\times\nu$ by the identity. This map is not an equivalence, but we shall see in the next paragraph that it becomes one after \emph{stabilisation}.
 
 \begin{rem}Note that the map \eqref{equ:extedingtanegntialbundlemaps} does \emph{not} extend to a map
$ \BlockBun_A(\xi,\nu;\ell_0)_\bullet\ra \BlockBun_A(\xi,\nu;\ell_0)_\bullet^\tau$
  in an obvious way. This is because, for a bundle map $\phi\colon \Delta^p\times \xi\ra \Delta^p\times\nu$ covering a block map, the map $\tau_{\Delta^p}\times \xi\ra \tau_{\Delta^p}\times\nu$ that maps $(x,y)$ to $(x,\phi(\overline{x},y))$ is \emph{not} a bundle map unless $\phi$ commutes with the projection on $\Delta^p$. Here $\overline{x}\in \Delta^p$ is the underlying point of $x\in\tau_{\Delta^p}$.
 \end{rem}
 
\subsection{Bundle maps, stably}\label{sect:stablebundlemaps}
A \emph{stable vector bundle} is a sequence of vector bundles $\psi=\{\psi_k\ra B_k\}_{k\ge l}$ for some $l\ge0$, where $\psi_k$ is $k$-dimensional, together with \emph{structure maps} $\psi_k\oplus \varepsilon \ra \psi_{k+1}$ for $k\ge l$, covering maps of the form $B_{k}\ra B_{k+1}$. Given a $d$-dimensional vector bundle $\xi$, its \emph{stabilisation} is the stable vector bundle $\xi^s$ with $\xi^s_{d+k}=\xi\oplus \epsilon^k$ and the identity as structure maps. For a $d$-dimensional vector bundle $\xi\ra X$, a stable vector bundle $\psi$, and a bundle map $\ell_0\colon \xi|_{A}\oplus \varepsilon^k\ra \psi_{d+k}$ for some $k$, the semi-simplicial set of \emph{stable bundle maps} is the colimit \[\Bun_A(\xi^s,\psi;\ell_0)_\bullet\coloneq \colim_{m\ge k}\Bun_A(\xi\oplus \varepsilon^m,\psi_{d+m};\ell_0)_\bullet\] over the stabilisation maps \[\Bun_A(\xi\oplus \varepsilon^m,\psi_{d+m};\ell_0)_\bullet\lra \Bun_A(\xi\oplus \varepsilon^{m+1},\psi_{d+m+1};\ell_0)_\bullet\] given by adding a trivial line bundle followed by postcomposition with the structure map $\psi_{d+m}\oplus\varepsilon\ra\psi_{d+m+1}$. Analogously, we define \emph{stable tangential bundle maps} as the colimit
\[\Bun_A(\xi^s,\psi;\ell_0)^\tau_\bullet\coloneq \colim_{m\ge k}\Bun_A(\xi\oplus \varepsilon^m,\psi_{d+m};\ell_0)^\tau_\bullet.\] As in \cref{sect:unstablebundlemaps}, there are semi-simplicial sub-monoids \[\hAut_A(\xi^s;C)_\bullet\subset \Bun_A(\xi^s,\xi^s;\inc)_\bullet\quad\text{and}\quad\hAut_A(\xi^s;C)_\bullet^\tau\subset \Bun_A(\xi^s,\xi^s;\inc)^\tau_\bullet,\] and also block variants of these semi-simplicial sets, defined by adding appropriate tildes. As the extension map \eqref{equ:extedingtanegntialbundlemaps} is compatible with the stabilisation maps, it gives rise to maps 
\begin{equation}\label{equ:euivtangential}\Bun_A(\xi^s,\psi;\ell_0)_\bullet\lra \Bun_A(\xi^s,\psi;\ell_0)_\bullet^\tau\quad\text{and}\quad \hAut_A(\xi^s;C)_\bullet\lra \hAut_A(\xi^s;C)_\bullet^\tau\end{equation} which we show in \cref{lem:tangentialisequ} to be equivalences as long as $X$ is a finite CW complex.

\subsection{Tangential structures}
\label{sect:tangentialstructures}
A $d$-dimensional \emph{tangential structure} is a fibration $\theta\colon B_d\ra \BO(d)$ whose target is the base of the universal $d$-dimensional vector bundle $\gamma_d\ra\BO(d)$. The semi-simplicial set of \emph{$\theta$-structures} on a $d$-dimensional vector bundle $\xi$ over a CW complex $X$ is the semi-simplicial set of bundle maps $\Bun_A(\xi,\theta^*\gamma_d;\ell_0)$ from $\xi$ to the pullback $\theta^*\gamma_d$ of the universal bundle along $\theta$, relative to a fixed bundle map $\ell_0\colon \xi|_A\ra \theta_d^*\gamma_d$. We denote the homotopy quotient (in the sense of \cref{sect:semisimplicial}) of the action of $\hAut_A(\xi,C)_\bullet\subset \Bun_A(\xi;\xi,\inc)_\bullet$ on the semi-simplicial set of $\theta$-structures by
\[\BhAut^{\theta}_A(\xi,C;\ell_0)_\bullet\coloneq \Bun_A(\xi,\theta^*\gamma_d;\ell_0)_\bullet\sslash \hAut_A(\xi,C)_\bullet.\]  
\begin{rem}Note that $\BhAut^{\theta}_A(\xi,C;\ell_0)_\bullet$ is in many cases empty or disconnected, so despite the suggestive notation, it is in general not the classifying space of any kind of group or monoid, (semi-)simplicial or topological. 
\end{rem}
\subsubsection{Stable tangential structures}\label{sec:stable-tangential-structures}
A \emph{stable tangential structure} is a fibration of the form $\Xi\colon B\ra \BO$, which induces a $d$-dimensional tangential structure $\Xi_d\colon B_d\ra\BO(d)$ for any $d\ge0$ by pulling back $\Xi$ along the stabilisation map $\BO(d)\ra\BO$. Note, however, that not all $d$-dimensional tangential structures arise this way, for instance the tangential structure $\EO(d)\ra\BO(d)$ encoding unstable framings does not. A stable tangential structure defines a stable vector bundle $\Xi^*\gamma=\{\Xi_{d}^*\gamma_d\}_{d\ge0}$ whose structure maps are induced by the canonical bundle map $\gamma_d\oplus\varepsilon\ra\gamma_{d+1}$ covering the usual stabilisation map $\BO(d)\ra\BO(d+1)$. Given a stable bundle map $\ell_0\colon \xi^s|_{A}\ra\Xi^*\gamma$, we call $\Bun_A(\xi^s,\Xi^*\gamma;\ell_0)$ the 
semi-simplicial set of \emph{stable $\theta$-structures} on $\xi$. As in the unstable case, we abbreviate
\begin{equation}\label{equ:tangentialhAut}\BhAut^{\Xi}_A(\xi^s,C;\ell_0)_\bullet\coloneq\Bun_A(\xi^s,\Xi^*\gamma;\ell_0)_\bullet\sslash \hAut_A(\xi^s,C)_\bullet.\end{equation} Tangential and or block variants of the previous definitions are defined by using the respective variants of bundle maps and adding appropriate tildes and or $\tau$-superscripts.

\subsubsection{Forgetting tangential structures}It follows from \cref{lem:simplicialaction} that the canonical sequence of semi-simplicial spaces
\begin{equation}\label{equ:bundfibreseqeunce}\Bun_A(\xi,\theta^*\gamma_d;\ell_0)_\bullet\lra\BhAut^{\theta}_A(\xi,C;\ell_0)_\bullet\lra \BhAut_A(\xi,C)_\bullet\end{equation} realises to a quasi-fibration, because $\Bun_A(\xi,\theta^*\gamma_d;\ell_0)_\bullet$ and $\hAut_A(\xi,C)_\bullet$ admit degeneracies since they satisfy the Kan property (they agree with the singular complex of a space, see \cref{sect:unstablebundlemaps}) and the action is by equivalences as $\hAut_A(\xi,C)_\bullet$ is group-like. The same argument applies to the stable analogue of this sequence involving \eqref{equ:tangentialhAut} and, using \cref{cor:tangentialiskan} (assuming that $X$ is a finite complex), also to its variants involving tangential and or block bundle maps. By definition of the universal bundle, the semi-simplicial set $\Bun_A(\xi,\theta^*\gamma_d;\ell_0)_\bullet$ is contractible in the case $\theta=\id$, so the second map in \eqref{equ:bundfibreseqeunce} is an equivalence for this particular choice of $\theta$. The analogous statement holds in the stable case as well and, as a result of Lemmas~\ref{lem:blockisnonblock} and~\ref{lem:tangentialisequ}, also for the tangential and or block variants.

\subsection{The derivative maps}
\label{sect:derivativemap}
Taking fibrewise derivatives of diffeomorphisms $\varphi\colon \Delta^p\times W\ra \Delta^p\times W$ over $\Delta^p$ induces a canonical semi-simplicial map
\begin{equation}\label{equ:derivativemap}\Diff_M(W,N)_\bullet\lra\hAut_M(\tau_W,N)_\bullet,\end{equation}
which we call the \emph{derivative map}. Furthermore, the notion of a tangential bundle map is designed exactly so that there is a \emph{block derivative map}
\begin{equation}\label{equ:blockderivativemap}\BlockDiff_M(W,N)_\bullet\lra\BlockhAut_M(\tau^s_W,N)^\tau_\bullet.\end{equation} given by assigning a block diffeomorphism $\Delta^p\times W\ra \Delta^p\times W$ its derivative $\tau_{\Delta^p}\times \tau_W\ra\tau_{\Delta^p}\times\tau_W,$ which indeed makes the square \eqref{equ:collarsquare} commute as $\varphi$ is assumed to be collared in the sense of \cref{sect:blockspaces}. 

\begin{rem}\ 
\begin{enumerate}
\item A different model of the block derivative map \eqref{equ:blockderivativemap} was considered by Berglund and Madsen in their prominent study of the rational homotopy type of spaces of block diffeomorphisms of manifolds with certain boundary conditions \cite{BerglundMadsen} (see Section 4.3 loc.\,cit. and \cref{rem:BerglundMadsenComparison} below).

\item Lemmas~\ref{lem:blockisnonblock} and~\ref{lem:tangentialisequ} provide a canonical equivalence \[\BlockhAut_M(\tau^s_W,N)^\tau_\bullet\simeq \hAut_M(\tau^s_W,N)_\bullet,\] so the delooped space $\BlockBhAut_M(\tau^s_W,N)^\tau_\bullet$ classifies fibrations $\pi_W\colon E\ra B$ with fibre $W$ together with the following data:\begin{enumerate}[label=(\arabic*)]
\item maps of fibrations $\pi_M\ra \pi_W\leftarrow \pi_N$ over the identity whose induced maps on fibres is equivalent to the system of inclusions $M\subset W\supset N$, 
\item a trivialisation of $\pi_M$, and 
\item a stable vector bundle $E\ra\BO$ over the total space of $\pi_W$ whose restriction to each fibre agrees with the stable tangent bundle of $W$.\end{enumerate} From this point of view, the block derivative \eqref{equ:blockderivativemap} comes as no surprise: it is reminiscent of the fact that a block bundle has an underlying fibration and a stable vertical tangent bundle by \cite{EbertRWBlock}. However, somewhat curiously, the block derivative map \eqref{equ:blockderivativemap} obviously factors over the variant $\BlockhAut_M(\tau_W,N)^\tau_\bullet$ involving the \emph{unstable} tangent bundle of $W$, giving rise to an \emph{unstable block derivative map}
\[\BlockBDiff_M(W,N)_\bullet\lra \BlockBhAut_M(\tau_W,N)^\tau_\bullet.\] The target of this map is neither equivalent to $\BhAut_M(\tau^s_W,N)_\bullet$ nor $\BhAut_M(\tau_W,N)_\bullet$, and it would be interesting to have a good description of what it classifies.\end{enumerate}

\end{rem}
We denote the submonoids of the path components hit by the derivative maps by
\begin{equation}\label{equ:components-hit-by-derivative}\hAut^\cong_M(\tau_W,N)_\bullet\subset\hAut_M(\tau_W,N)_\bullet\quad\text{and}\quad \BlockhAut^\cong_M(\tau^s_W,N)^\tau_\bullet\subset \BlockhAut_M(\tau^s_W,N)^\tau_\bullet\end{equation} and add the same $\cong$-superscript to \eqref{equ:tangentialhAut} and its tangential block variant to indicate when we take homotopy quotients by the submonoids \eqref{equ:components-hit-by-derivative} instead of the full monoids. Defining \begin{equation}\label{equ:moduli-space-stable-tangential-block}\BlockBDiff_M^{\Xi}(W,N;\ell_0)_\bullet\coloneq\BlockBun_M(\tau^s_W,{\Xi}^*\gamma;\ell_0)^\tau_\bullet\sslash \BlockDiff_M(W,N)_\bullet\end{equation} for a stable tangential structure $\Xi$ and a $\Xi$-structure $\ell_0$ on $\tau^s_W|_M$ and, more commonly, \[\BDiff_M^{\theta}(W,N;\ell_0)_\bullet\coloneq\Bun_M(\tau_W,\theta^*\gamma_d;\ell_0)_\bullet\sslash \Diff_M(W,N)_\bullet\] in the unstable case, the two derivative maps fit into a a commutative square
\begin{equation}\label{equ:stable-to-unstable-diff-to-aut}
\begin{tikzcd}
\BDiff^{\Xi_d}_M(W,N;\ell_0)_\bullet\arrow[r]\arrow[d]&\BhAut^{\Xi_d,\cong}_M(\tau_W,N;\ell_0)_\bullet\arrow[d]\\
\BlockBDiff^{\Xi}_M(W,N;\ell_0)_\bullet\arrow[r]&\BlockBhAut^{\Xi,\cong}_M(\tau^s_W,N;\ell_0)^\tau_\bullet 
\end{tikzcd}
\end{equation}
whose vertical maps are induced by the canonical composition \[\Bun_M(\tau_W,\Xi_d^*\gamma_d;\ell_0)_\bullet\ra\Bun_M(\tau^s_W,\Xi^*\gamma;\ell_0)_\bullet\ra\Bun_M(\tau^s_W,\Xi^*\gamma;\ell_0)^\tau_\bullet\subset \BlockBun_M(\tau^s_W,{\Xi}^*\gamma;\ell_0)^\tau_\bullet.\] 
All maps in this composition are equivalences, the first one by an exercise in obstruction theory and the second map as well as the final inclusion as a result of Lemmas~\ref{lem:blockisnonblock} and~\ref{lem:tangentialisequ}.  The composition of equivalences just discussed also induces the vertical maps in the commutative diagram 
\begin{equation}\label{equ:blockandnonblockbundlequotient}
\begin{tikzcd}
\BDiff^{\Xi_d}_M(W,N;\ell_0)_\bullet\arrow[r]\arrow[d]&\BDiff_M(W,N)_\bullet\arrow[d]\\
\BlockBDiff^{\Xi}_M(W,N;\ell_0)_\bullet\arrow[r]&\BlockBDiff_M(W,N)_\bullet
\end{tikzcd}
\end{equation}
whose horizontal maps are induced by forgetting tangential structures.
\begin{lem}\label{lem:forgetting-is-cartesian}The geometric realisation of the square \eqref{equ:blockandnonblockbundlequotient} is homotopy cartesian.
\end{lem}
\begin{proof}By construction the induced map on horizontal strict  fibres agrees with the composition of equivalences discussed above \eqref{equ:blockandnonblockbundlequotient}, so it suffices to show that the horizontal maps of the square realise to quasi-fibrations. As $\Diff_M(W,N)_\bullet$ and $\BlockDiff_M(W,N)_\bullet$ are Kan and group-like (see \cref{section:diffeomorphisms}), this follows from \cref{lem:simplicialaction} and \cref{cor:tangentialiskan}.
\end{proof}

\section{Surgery theory and spaces of block diffeomorphisms}
\label{sect:surgerytheory}
We use surgery theory to give a partial $p$-local description of the space $\BlockBDiff^\Xi_{\partial_0W}(W,\partial_1W)$ of block diffeomorphisms with tangential structures in terms of spaces of homotopy automorphisms with bundle data for manifold triads $(W;\partial_0W,\partial_1W)$ of dimension $d\ge6$ satisfying a $\pi$-$\pi$-condition. 

\begin{notation} As a point of notation, we refer to the geometric realisation of any of the semi-simplicial sets or spaces of the previous sections by omitting their $\bullet$-subscripts.
\end{notation}

\subsection{A reminder of surgery theory}\label{section:surgerytheory}
A $d$-dimensional \emph{manifold triad} is a triple $W=(W;\partial_0W,\partial_1W)$ consisting of a compact smooth $d$-manifold $W$ (possibly with corners) and (possibly empty) submanifolds $\partial_0W\subset W$ and $\partial_1W\subset W$ such that
\[\partial W=\partial_0 W\cup \partial_1 W\quad\text{and}\quad \partial(\partial_0W)=\partial_0W\cap\partial_1W=\partial(\partial_1W).\] A diffeomorphism or (simple) homotopy equivalence $(W;\partial_0W,\partial_1W)\ra(W';\partial_0W',\partial_1W')$ between two manifold triads is a diffeomorphism or (simple) homotopy equivalence $W\ra W'$ which restricts to a map of this kind between $\partial_0W$ and $\partial_0W'$, between $\partial_1W$ and $\partial_1W'$, and between their intersections. At times, we omit $\partial_0W$ and $\partial_1W$ from the notation and abbreviate a triad $(W;\partial_0W,\partial_1W)$ simply by $W$. The \emph{smooth structure set} $\cS(W)$ of a triad $W$ (see e.g.\,\cite[Ch.\,10]{WallBook}) is the collection of equivalence classes of simple homotopy equivalences of triads $N\ra W$ that restrict to a diffeomorphism $\partial_0N\ra\partial_0W$, where two such equivalences $N\ra W$ and $N'\ra W$ are considered equivalent if there exists a diffeomorphism of triads $N\ra N'$ that makes the triangle of triads
\[
\begin{tikzcd}[row sep=0.1cm,column sep=1.5cm]
N\arrow[dr,"\simeq",bend left=10]\arrow[dd,"\cong",swap]&\\
&W\\
N'\arrow[ur,"\simeq",swap,bend right=10]&
\end{tikzcd}
\]
homotopy commute relative to $\partial_0N$. The structure set $\cS(W)$ is canonically based; the identity serves as basepoint. The main tool to access $\cS(W)$ is the \emph{surgery exact sequence}
\begin{equation}\label{equ:ses}
\begin{tikzcd}[column sep=0.3cm]
\ldots\rar&\cS(W\times D^k)\rar&\cN(W\times D^k)\arrow[d, phantom, ""{coordinate, name=Z}]\rar&\cL(W\times D^k) \rar&\cS(W\times D^{k-1})\arrow[dllll,               
rounded corners,
to path={ -- ([xshift=2ex]\tikztostart.east)
|- (Z) [near end]\tikztonodes
-| ([xshift=-2ex]\tikztotarget.west) -- (\tikztotarget)}]
\\ \ldots\rar&\cL(W\times D^1)\rar&\cS(W)\rar&\cN(W)\rar&\cL(W)
\end{tikzcd}
\end{equation}
which relates the structure sets $\cS(W\times D^k)$ of the triads
\[W\times D^k=(W\times D^k;\partial_0W\times D^k\cup W\times \partial D^k;\partial_1W\times D^k)\]
 to the \emph{sets of normal invariants} $\cN(W\times D^k)$ and the \emph{$L$-groups} $\cL(W\times D^k)$. Assuming $d\ge 6$, this sequence is an exact sequence of abelian groups until $\cL(W\times D^1)$ where it continues as an exact sequence of based sets (see e.g.\,\cite[Ch.\,10]{WallBook}). The similarity with the long exact sequence of homotopy groups induced by a fibration is no coincidence: Quinn's \emph{surgery fibration} \cite{Quinn,QuinnThesis} is a homotopy fibration of based spaces
\begin{equation}\label{equ:Quinnfibration}\tilde{\bfS}(W)\lra\bfN(W)\lra \bfL(W)\end{equation} that induces \eqref{equ:ses} on homotopy groups (see also \cite[Ch.\,17A]{WallBook}, or \cite{Nicas} for a detailed account in the topological category). We refrain from describing  \eqref{equ:ses} or \eqref{equ:Quinnfibration} in detail; all we shall need to know are a few basic properties, which we explain in the following.

\subsubsection{The block structure space}\label{sect:structurespace}Assuming $d\ge6$, an application of the $s$-cobordism theorem results in a preferred homotopy equivalence \begin{equation}\label{equ:structurespaceequ}\BlockhAut^{\cong}_{\partial_0W}(W,\partial_1W)/\BlockDiff_{\partial_0W}(W)\simeq \tilde{\bfS}(W)_{\id}\end{equation} between the homotopy fibre of the canonical map $\BlockBDiff_{\partial_0W}(W)\ra \BlockBhAut^{\cong}_{\partial_0W}(W,\partial_1W)$ induced by inclusion and the basepoint component $\tilde{\bfS}(W)_{\id}\subset \tilde{\bfS}(W)$ of the \emph{block structure space} (cf.\,\cite[Ch.\,17A]{WallBook} or \cite[p.\,33-34]{BerglundMadsenI}). Here \begin{equation}\label{equ:components-aut-hit-by-diff}\BlockhAut^{\cong}_{\partial_0W}(W,\partial_1W)\subset \BlockhAut_{\partial_0W}(W,\partial_1W)\end{equation} are the components in the image of the canonical map $\BlockDiff_{\partial_0W}(W)\ra \BlockhAut_{\partial_0W}(W,\partial_1W)$. Note that a diffeomorphism of $W$ that fixes $\partial_0W$ pointwise automatically preserves $\partial_1W$ setwise, since $\partial_1W$ is the complement of the interior of $\partial_0W\subset\partial W$. On homotopy groups, the equivalence \eqref{equ:structurespaceequ} can be described as follows: using the combinatorial description of the relative homotopy groups of a semi-simplicial Kan pair, a class in \[\pi_k(\BlockhAut^{\cong}_{\partial_0W}(W,\partial_1W)/\BlockDiff_{\partial_0W}(W);\id)\cong\pi_k(\BlockhAut^{\cong}_{\partial_0W}(W,\partial_1W),\BlockDiff_{\partial_0W}(W);\id)\] is represented by a simple homotopy equivalence of triads $W\times D^k\ra W\times D^k$ which is the identity on $\partial_0W\times D^k$ and restricts to a diffeomorphism on $W\times \partial D^k$, so it defines a class in the structure set $\cS(W\times D^k)\cong\pi_k(\tilde{\bfS}(W); \id)$.

\subsubsection{The space of normal invariants}\label{sect:normalinvariants}The space of normal invariants $\bfN(W)$ admits a preferred homotopy equivalence to the pointed mapping space $\Maps_*(W/\partial_0W,\G/\oO)$ based at the constant map, where $\G/\oO$ is the homotopy fibre of the canonical map $\BO\ra\BG$ witnessing the fact that a stable vector bundle has an underlying stable spherical fibration (see e.g.\,\cite{Quinn} or \cite[Ch.\,10, 17A]{WallBook}). This map is one of infinite loop spaces, so its homotopy fibre $\G/\oO$ is an infinite loop space and hence so is the mapping space $\Maps_*(W/\partial_0W,\G/\oO)$. On homotopy groups, the composition
\[\tilde{\bfS}(W)\ra\bfN(W)\simeq\Maps_*(W/\partial_0W,\G/\oO)\ra\Maps_*(W/\partial_0W,\BO)\] has the following geometric description (see e.g.\,\cite[p.\,113-114]{WallBook}): given a class in the structure set $\pi_k(\tilde{\bfS}(W);*)\cong\cS(W\times D^k)$ represented by a simple homotopy equivalence $\varphi\colon N\ra W\times D^k$, choose a homotopy inverse $\tilde{\varphi}\colon W\times D^k\ra N$ of triads that agrees with $(\varphi|_{\partial_0N})^{-1}$ on $\partial_0(W\times D^k)$. Writing $\nu^s$ and $\tau^s$ for the stable normal respectively tangent bundle of a manifold, the stable vector bundle $(\tilde{\varphi}^*\nu^s_N)\oplus \tau^s_{W\times D^k}$ on $W\times D^k$ comes with a trivialisation on the subspace \[\partial_0(W\times D^k)=\partial_0W\times D^k\cup W\times \partial D^{k}\] by making use of the diffeomorphism $\tilde{\varphi}|_{\partial_0(W\times D^k)}$, and hence gives rise to a class \[[(\tilde{\varphi}^*\nu^s_N)\oplus \tau^s_{W\times D^k}]\in \pi_k(\Maps_*(W/\partial_0W,\BO);*).\] 

\subsubsection{The $L$-theory space}\label{sec:l-theory-space}The $L$-theory space $\bfL(W)$ is an infinite loop space as well (see e.g.\,\cite[Prop.\,2.2.2]{Nicas}), and its homotopy groups are canonically isomorphic to Wall's \emph{quadratic $L$-groups} (see e.g.\,\cite[Prop.\,2.2.4]{Nicas}). We shall not need to know much about these groups, except that $\pi_k(\bfL(W);*)\cong\cL(W\times D^k)$ vanishes if $W$ satisfies the \emph{$\pi$-$\pi$-condition}, i.e.\,if the inclusion $\partial_1W\subset W$ induces an equivalence on fundamental groupoids. This is a consequence of the exact sequence of $L$-groups of a triad (or more generally, $n$-ad) described for instance in \cite[Thm 3.1]{WallBook}. In other words, under these assumptions, the $L$-theory space $\bfL(W)$ is weakly contractible, so \eqref{equ:Quinnfibration} induces a preferred equivalence $\tilde{\bfS}(W)\simeq \bfN(W)$---an instance of the so-called \emph{$\pi$-$\pi$-theorem}.

\subsection{The block derivative map of a triad}With these basics of space-level surgery theory in mind, we now turn towards studying connectivity properties of the block derivative map resulting from \eqref{equ:blockderivativemap} and \eqref{equ:components-hit-by-derivative}
\[\BlockBDiff_{\partial_0W}(W)\lra \BlockBhAut^{\cong}_{\partial_0W}(\tau_W^s,\partial_1W)^\tau\] and its enhancement involving tangential structures (i.e.\ the realisation of the bottom row of \eqref{equ:stable-to-unstable-diff-to-aut}), beginning with a technical but useful lemma on the homotopy fibre featuring in \eqref{equ:structurespaceequ}. We refer \cref{section:nilpotency} for a recollection on nilpotent spaces.

\begin{lem}\label{lem:nilpotent}For a manifold triad $(W;\partial_0W,\partial_1W)$ of dimension $d\ge6$, the homotopy fibre \[\BlockhAut^\cong_{\partial_0W}(W,\partial_1W)/\BlockDiff_{\partial_0W}(W)=\hofib(\BlockBDiff_{\partial_0W}(W)\ra\BlockBhAut^\cong_{\partial_0W}(W,\partial_1W))\] is a nilpotent space.
\end{lem}
\begin{proof}
By definition of the subspace $\BlockhAut^\cong_{\partial_0W}(W,\partial_1W)\subset \BlockhAut_{\partial_0W}(W,\partial_1W)$ in \cref{sect:structurespace}, the map $\BlockBDiff_{\partial_0W}(W)\ra \BhAut^\cong_{\partial_0W}(W,\partial_1W)$ is surjective on fundamental groups, so its homotopy fibre is connected. To see that it is nilpotent, we use the equivalence \eqref{equ:structurespaceequ} to the identity component $\tilde{\bfS}(W)_{\id}$ of the block structure space, which is itself equivalent to a component of the homotopy fibre of the surgery obstruction map $\bfN(W)\ra\bfL(W)$ of \eqref{equ:Quinnfibration}. Given this description of the space in consideration, the claim follows from an application of \cref{lem:nilpotentfibre}, using the fact that, being an infinite loop space, the space of normal invariants $\bfN(W)$ is nilpotent (see \cref{sect:normalinvariants}).\end{proof}

To state the main result of this section, some notation is in order. For a stable tangential structure $\Xi\colon B\ra\BO$ and a $\Xi$-structure $\ell\colon \tau^s_W\ra\Xi^*\gamma$ (see \cref{sec:stable-tangential-structures}), we write $\ell_0$ for the restriction of $\ell$ to $\tau^s_{W}|_{\partial_0W}$ and denote by 
\[\BlockBDiff^\Xi_{\partial_0W}(W;\ell_0)_\ell\quad\text{and}\quad \BlockBhAut^{\Xi,\cong}_{\partial_0W}(\tau_W^s,\partial_1W;\ell_0)^\tau_\ell\]
the components of the spaces (see \cref{sec:stable-tangential-structures} and \cref{sect:derivativemap} for the notation)
\[\BlockBDiff^\Xi_{\partial_0W}(W;\ell_0)=\BlockBun_{\partial_0W}(\tau^s_W,{\Xi}^*\gamma;\ell_0)^\tau\sslash \BlockDiff_{\partial_0W}(W)
\]
and 
\[\BlockBhAut^{\Xi,\cong}_{\partial_0W}(\tau_W^s,\partial_1W;\ell_0)^\tau=\BlockBun_{\partial_0W}(\tau^s_W,{\Xi}^*\gamma;\ell_0)^\tau\sslash\BlockhAut^{\cong}_{\partial_0W}(\tau_W^s,\partial_1W)^\tau\]
 that correspond to $\ell$ under the canonical bijection \[\pi_0\BlockBhAut^{\Xi,\cong}_{\partial_0W}(\tau_W^s,\partial_1W;\ell_0)^\tau\cong \pi_0\BlockBDiff^\Xi_{\partial_0W}(W;\ell_0)\cong  \pi_0\BlockBun_{\partial_0W}(\tau^s_W,{\Xi}^*\gamma;\ell_0)^\tau/\pi_0\BlockDiff_{\partial_0W}(W)\] resulting from \cref{rem:simplicial-remark} \ref{item:simplicial-remark-final} and the discussion in Sections~\ref{sec:stable-tangential-structures} and ~\ref{sect:derivativemap}, using that the map induced by the block derivative \[\pi_0\BlockDiff_{\partial_0W}(W)\lra \pi_0\BlockhAut^{\cong}_{\partial_0W}(\tau_W^s,\partial_1W)^\tau\] is surjective by the definition of the target in \eqref{equ:components-hit-by-derivative}. Recall from \cref{sec:l-theory-space} that a manifold triad $W=(W;\partial_0W,\partial_1W)$ satisfies the \emph{$\pi$-$\pi$-condition} if the inclusion $\partial_1W\subset W$ induces an equivalence on fundamental groupoids. 

\begin{thm}\label{thm:rationalmodel}
Let $d\ge6$ and $W$ be a $d$-dimensional triad satisfying the $\pi$-$\pi$-condition. For a stable tangential structure $\Xi$ and a $\Xi$-structure $\ell$ on $\tau_W^s$, the homotopy fibre of the map
\[\BlockBDiff^\Xi_{\partial_0W}(W;\ell_0)_\ell\lra \BlockBhAut^{\Xi,\cong}_{\partial_0W}(\tau_W^s,\partial_1W;\ell_0)^\tau_\ell\] is nilpotent and has finite homotopy groups. Moreover, this  fibre is $p$-locally $(2p-4-k)$-connected for primes $p$, where $k$ is the relative handle dimension of the inclusion $\partial_0W\subset W$.\end{thm}

\begin{rem}\label{rem:BerglundMadsenComparison}\cref{thm:rationalmodel} is inspired by a similar result of Berglund and Madsen \cite[Thm\,1.1]{BerglundMadsen}, which applies to a different class of triads, namely those satisfying $\partial_0W=\partial W\cong S^{d-1}$. Another point in which their result differs from ours is that it is purely rational, and does in fact not seem to admit a $p$-local refinement analogous to \cref{thm:rationalmodel}. This is because $p$-torsion occurs for primes $p$ that can be rather large with respect to the degree, originating from contributions of numerators of divided Bernoulli numbers to the homotopy groups of the homotopy fibre $\TOP/\oO$ of the canonical map $\BO\ra\BTOP$.
\end{rem}

\begin{proof}[Proof of \cref{thm:rationalmodel}]Using \cref{cor:tangentialiskan}, an application of \cref{lem:simplicialaction} to the horizontal arrows of the canonical square
\begin{equation} \label{equ:squarexistructure}
\begin{tikzcd}
\BlockBDiff^{\Xi}_{\partial_0W}(W;\ell_0)^\tau_\ell\rar\dar&\BlockBDiff_{\partial_0W}(W)\arrow[d]\\
\BlockBhAut^{\Xi,\cong}_{\partial_0W}(\tau_W^s,\partial_1W;\ell_0)^\tau_\ell\rar&\BlockBhAut^{\cong}_{\partial_0W}(\tau_W^s,\partial_1W)^\tau,
\end{tikzcd}
\end{equation}
identifies the horizontal homotopy fibres with the union of components of the space $\BlockBun_{\partial_0W}(\tau^s_W,{\Xi}^*\gamma;\ell_0)^\tau$ given by the $\pi_0\BlockDiff_{\partial_0W}(W)$-orbit of the $\Xi$-structure $\ell$, so \eqref{equ:squarexistructure} is homotopy cartesian, which reduces the proof to the case $\Xi=\id$ in which both rows of the square are equivalences (see \cref{sec:stable-tangential-structures}). To settle the case $\Xi=\id$, we consult the map of fibre sequences induced by the block derivative map
\begin{equation}\label{equ:added-map-between-fibration}
\begin{tikzcd}[column sep=0.4cm,row sep=0.5cm]
\frac{\BlockhAut^{\cong}_{\partial_0W}(W,\partial_1W)}{\BlockDiff_{\partial_0W}(W)}
\rar\arrow[d,swap]&\BlockBDiff_{\partial_0W}(W)\rar\dar&\BlockBhAut^{\cong}_{\partial_0W}(W,\partial_1W)\arrow[d,equal]\\
\frac{\BlockhAut^{\cong}_{\partial_0W}(W,\partial_1W)}{\BlockhAut^{\cong}_{\partial_0W}(\tau_W^s,\partial_1W)^\tau}\rar&\BlockBhAut^{\cong}_{\partial_0W}(\tau_W^s,\partial_1W)^\tau\rar&\BlockBhAut^{\cong}_{\partial_0W}(W,\partial_1W)
\end{tikzcd}
\end{equation}
in order to see that the homotopy fibre in question agrees with the homotopy fibre of the left vertical map and is therefore nilpotent by Lemmas~\ref{lem:nilpotentfibre} and~\ref{lem:nilpotent}. The bottom right horizontal map of \eqref{equ:added-map-between-fibration} forms the rightmost column of a commutative diagram \begin{equation}\label{equ:comparisondiagram1}
\begin{tikzcd}
\BhAut^{\cong}_{\partial_0W}(\tau_W^s,\partial_1W)
\arrow[r,"\simeq"]\arrow[d]&\BhAut^{\cong}_{\partial_0W}(\tau_W^s,\partial_1W)^\tau\arrow[r,"\simeq"]\arrow[d]&\BlockBhAut^{\cong}_{\partial_0W}(\tau_W^s,\partial_1W)^\tau\arrow[d]\\
\BhAut^{\cong}_{\partial_0W}(W,\partial_1W)
\arrow[r,equal]&\BhAut^{\cong}_{\partial_0W}(W,\partial_1W)\arrow[r,"\simeq"]&\BlockBhAut^{\cong}_{\partial_0W}(W,\partial_1W)
\end{tikzcd}
\end{equation}
whose right horizontal maps are induced by inclusion and are equivalences; the upper by \cref{lem:blockisnonblock} and the lower by the discussion in \cref{sect:hAutPrelim}. The left upper horizontal map is the equivalence from \eqref{equ:euivtangential} and the leftmost vertical arrow fits into a commutative square induced by inclusion 
\begin{equation}\label{equ:comparisondiagram2}
\begin{tikzcd}
\BhAut^{\cong}_{\partial_0W}(\tau_W^s,\partial_1W)\arrow[r]\arrow[d]&\BhAut_{\partial_0W}(\tau_W^s,\partial_1W)\arrow[d]\\
\BhAut^{\cong}_{\partial_0W}(W,\partial_1W)\arrow[r]&\BhAut_{\partial_0W}(W,\partial_1W)_{\tau^s_{W}},
\end{tikzcd}
\end{equation}
where $\hAut_{\partial_0W}(W,\partial_1W)_{\tau_W^s}$ are the components in the image of the map 
\begin{equation}\label{equ:forget-bundle-maps}\hAut_{\partial_0W}(\tau_W^s,\partial_1W)\lra\hAut_{\partial_0W}(W,\partial_1W)\end{equation} that forgets  bundle data. Before taking geometric realisation, the map \eqref{equ:forget-bundle-maps} is easily seen to be a Kan fibration, so the homotopy fibre of the right vertical map in \eqref{equ:comparisondiagram2} is equivalent to the classifying space of the gauge group of $\tau_W^s$ relative to $\partial_0W$, which is in turn canonically equivalent to the space $\Maps_{\partial_0W}(W,\BO)_{\tau^s_W}$ of maps homotopic to a choice of classifying map for $\tau_W^s$ relative to $\partial_0W$. The horizontal arrows in \eqref{equ:comparisondiagram2} are $1$-coconnected by construction, so the same holds for the induced map on vertical homotopy fibres. Combining this with the chain of equivalences \eqref{equ:comparisondiagram1}, we see that the left vertical map in \eqref{equ:added-map-between-fibration} agrees, up to canonical equivalence and postcomposition with a $1$-coconnected map, with a map 
\[\BlockhAut^{\cong}_{\partial_0W}(W,\partial_1W)/\BlockDiff_{\partial_0W}(W)\lra \Maps_{\partial_0W}(W,\BO)_{\tau_W^s},\] so it suffices to show that this map between nilpotent spaces is $p$-locally $(2p-3-k)$-connected and that its homotopy fibre has finite homotopy groups. On homotopy groups, this map has the following description: a class in $\pi_k(\BlockhAut^{\cong}_{\partial_0W}(W,\partial_1W),\BlockDiff_{\partial_0W}(W);\id)$ is represented by a homotopy equivalence of triads $\varphi\colon W\times D^k\ra W\times D^k$ that is the identity on $D^k\times \partial_0W$ and restricts to a diffeomorphism on $\partial D^k\times W$ that is the identity on $*\times M$ for a base point $*\in \partial D^k$. The pullback $\varphi^*\tau_{W\times D^k}^s$ is a stable vector bundle over $W\times D^k$ that agrees with $\tau_{W\times D^k}^s$ over $D^k\times \partial_0W\cup \{*\}\times W$ and comes with a canonical identification $\varphi^*\tau_{W\times D^k}^s|_{\partial D^k\times M}\cong\tau_{W\times D^k}^s|_{\partial D^k\times M}$ given by the derivative of $\varphi|_{\partial D^k\times M}$, so it defines a class in the $k$th homotopy group of $\Maps_{\partial_0W}(W,\BO)_{\tau_W^s}$ based at a choice of classifying map for $\tau_W^s$. Comparing this description with those of the equivalence \[\BlockhAut^{\cong}_{\partial_0W}(W,\partial_1W)/\BlockDiff_{\partial_0W}(W)\simeq\tilde{\bfS}(W)_{\id}\] and the map $\tilde{\bfS}(W)_{\id}\ra\Maps_*(W/\partial_0W,\BO)_0$ on homotopy groups explained in Sections~\ref{sect:structurespace} and~\ref{sect:normalinvariants} respectively, one sees that the diagram of nilpotent spaces
\[
\begin{tikzcd}[row sep=0.4cm,  ar symbol/.style = {draw=none,"\textstyle#1" description,sloped},
  equivalent/.style = {ar symbol={\simeq}}]
\BlockhAut^{\cong}_{\partial_0W}(W,\partial_1W)/\BlockDiff_{\partial_0W}(W)\arrow[r]\arrow[d,equivalent]&\Maps_{\partial_0W}(W,\BO)_{\tau^s_W}\arrow[dd,"\iota \circ((-)\cdot\nu^s_W)","\simeq"']\\[-0.3cm]
\tilde{\bfS}(W)_{\id}\arrow[d,"\simeq",swap]&\\
\Maps_*(W/\partial_0W,\G/\oO)_0\arrow[r]&\Maps_*(W/\partial_0W,\BO)_{0},
\end{tikzcd}
\]
commutes upon taking homotopy groups. Here the upper horizontal map is the one we just discussed, the right vertical equivalence is induced by multiplication with the stable normal bundle $\nu^s_W$ using the infinite loop space structure induced from that of $\BO$ followed by the involution on $\Maps_{\partial_0W}(W,\BO)_{0}$ induced by the canonical involution on $\BO$, the left vertical equivalences are induced by the surgery fibration (see Sections~\ref{sect:structurespace} and~\ref{sect:normalinvariants}), and the bottom horizontal is given by postcomposition with the canonical map $\G/\oO\ra \BO$.  Note that, being an infinite loop space, the mapping space $\Maps_*(W/\partial_0W,\BO)_{0}$ is simple, so its homotopy groups at different base points are canonically identified. The bottom arrow is a map of infinite loop spaces and its homotopy fibre equivalent to a collection of components of the infinite loop space $\Maps_*(W/\partial_0W,\G)$, so to finish the proof, we are left to show that the homotopy groups of this infinite loop space (including in degree $0$) are finite and vanish $p$-locally in degrees $*<(2p-3-k)$. This follows from an application of obstruction theory, since $W/\partial_0W$ has no cohomology above degree $k$ by the assumption on the relative handle dimension of $\partial_0W\subset W$ and the homotopy groups of $\G$ are $\bfZ/2$ in degree $0$ and agree the stable homotopy groups of spheres in positive degrees, which are finite and $p$-torsion free in degrees $*<2p-3$ by a result of Serre \cite[p.\,498, Prop.\,5]{Serre}.
\end{proof}

The stable tangential structure we shall be primarily interested in is the one encoding stable framings, which we denote by $\sfr\colon \EO\ra\BO$. In this case, the $p$-local approximation of the space of block diffeomorphisms with tangential structures provided by \cref{thm:rationalmodel} can be further simplified in terms of the union of components \[\BlockhAut^{\cong}_{\partial_0W}(W,\partial_1W)_\ell\subset\BlockhAut^{\cong}_{\partial_0W}(W,\partial_1W)\] given by the image of the canonical map that forgets tangential structures
\[\BlockBhAut^{\sfr,\cong}_{\partial_0W}(\tau_W^s,\partial_1W;\ell_0)^\tau_\ell\lra \BlockBhAut^{\cong}_{\partial_0W}(W,\partial_1W)\] on fundamental groups based at a fixed stable framing $\ell$ of $\tau_{W}$ (see \eqref{equ:components-hit-by-derivative} and \eqref{equ:components-aut-hit-by-diff} for the notation). Loosely speaking, these are the components of homotopy equivalences of triads that are homotopic to a diffeomorphism preserving the component of the stable framing $\ell$.

\begin{cor}\label{cor:stableframed}Let $d\ge6$ and $W$ be a $d$-dimensional manifold triad satisfying the $\pi$-$\pi$-condition. For a stable framing $\ell$ of $W$, the homotopy fibre of the natural map
\[\BlockBDiff^{\sfr}_{\partial_0W}(W;\ell_0)_\ell\lra \BlockBhAut^{\cong}_{\partial_0W}(W,\partial_1W)_\ell\] is nilpotent and has finite homotopy groups. Moreover, this fibre is $p$-locally $(2p-4-k)$-connected for primes $p$, where $k$ is the relative handle dimension of the inclusion $\partial_0W\subset W$.
\end{cor}

\begin{proof}Once we show that the map \[\BlockBhAut^{\sfr,\cong}_{\partial_0W}(\tau_W^s,\partial_1W;\ell_0)^\tau_\ell\lra \BlockBhAut^{\cong}_{\partial_0W}(W,\partial_1W)_\ell\] is an equivalence, the statement is a consequence of \cref{thm:rationalmodel}. By construction, the homotopy fibre of this map is connected. Using \eqref{equ:comparisondiagram1} and \eqref{equ:comparisondiagram2}, one sees that it is thus sufficient to show that the map 
\[\BhAut^{\sfr}_{\partial_0W}(\tau_W^s,\partial_1W;\ell_0)_\ell\lra \BhAut_{\partial_0W}(W,\partial_1W)\] is an equivalence. 
Taking vertical homotopy fibres in the map of fibre sequences
\[
\begin{tikzcd}[column sep=0.4cm]
\Bun_{\partial_0W}(\tau_W^s,\sfr^*\gamma;\ell_0)_\ell\rar \dar&\BhAut^{\sfr}_{\partial_0W}(\tau_W^s,\partial_1W;\ell_0)_\ell\rar\dar& \BhAut_{\partial_0W}(\tau_W^s,\partial_1W)\arrow[d]\\
*\rar&\BhAut_{\partial_0W}(W,\partial_1W)\arrow[r,equal]&\BhAut_{\partial_0W}(W,\partial_1W),
\end{tikzcd}
\]
where $\Bun_{\partial_0W}(\tau_W^s,\sfr^*\gamma;\ell_0)_\ell\subset \Bun_{\partial_0W}(\tau_W^s,\sfr^*\gamma;\ell_0)$ are the components of the $\ell$-orbit of the $\pi_0\hAut_{\partial_0W}(\tau_W^s,\partial_1W)$-action, we see that it suffices to show that the induced action of the loop space of the homotopy fibre of the right vertical map on $\Bun_{\partial_0W}(\tau_W^s,\sfr^*\gamma;\ell_0)_\ell$ is a torsor in the homotopical sense, i.e.\,that the map given by acting on $\ell$ is an equivalence. Since $\hAut_{\partial_0W}(\tau_W^s,\partial_1W)_\bullet\ra \hAut_{\partial_0W}(W,\partial_1W)_\bullet$ is a Kan fibration, this loop space is canonically equivalent to the space of bundle self-maps of $\tau_W^s$ that cover the identity and agree with the identity on $\tau_W^s|_{\partial_0W}$. Moreover, by construction of the top fibration, the induced action on $\Bun_{\partial_0W}(\tau_W^s,\sfr^*\gamma;\ell_0)_\ell$ is given by precomposition. Suitably modeled, this action is simply transitive, so the assertion follows.
\end{proof}

\section{High-dimensional handlebodies and their mapping classes}\label{section:mcg}
This section serves to compute variants of the mapping class group of a high-dimensional handlebody up to extensions in terms of automorphisms of the integral homology.

\subsection{Automorphisms of handlebodies}
\label{sect:autVg}
The proof of \cref{thm:mainthm} relies on considering a more general family of manifold than discs, the boundary connected sums
\[V_g\coloneq\natural^g (D^{n+1}\times S^n),\] and their boundaries as well as the manifolds obtained by cutting out a fixed embedded disc $D^{2n}\subset \partial V_g$, denoted by
\[W_g\coloneq\partial V_g\cong\sharp^g(S^n\times S^n)\quad\text{and}\quad W_{g,1}\coloneq W_g\backslash \interior{D^{2n}}.\] This includes the case $g=0$, that is $V_0=D^{2n+1}$, $W_0=S^{2n}$, and $W_{0,1}=D^{2n}$. Using the notation introduced in Sections \ref{section:diffeomorphisms} and \ref{sect:hAutPrelim}, the restriction of diffeomorphisms and relative homotopy automorphisms of $V_g$ to its boundary induces a commutative diagram 
\begin{equation}
\begin{tikzcd}\label{comparisonmaps}
\Diff_\partial(V_g)\arrow[r]\arrow[d]&\Diff_{D^{2n}}(V_g)\arrow[r]\arrow[d]&\Diff_{\partial}(W_{g,1})\arrow[d]\\
\hAut_\partial(V_g)\arrow[r]&\hAut_{D^{2n}}(V_g,W_{g,1})\arrow[r]&\hAut_{\partial}(W_{g,1})
\end{tikzcd}
\end{equation}
where the right horizontal maps are fibrations and the left maps the inclusions of the fibres over the identity. These fibrations need not be surjective; we denote their images by \[\Diff^{\ext}_{\partial}(W_{g,1})\subset \Diff_\partial(W_{g,1})\quad\text{ and }\quad\hAut^{\ext}_{\partial}(W_{g,1})\subset \hAut_{\partial}(W_{g,1}).\] Furthermore, in agreement with \eqref{equ:components-aut-hit-by-diff} we write \[\hAut^\cong_\partial(V_g),\quad \hAut^\cong_{D^{2n}}(V_g,W_{g,1}),\quad\text{and}\quad \hAut^\cong_{\partial}(W_{g,1})\] for the components hit by the vertical maps. Block variants of all of the above automorphism spaces are defined in the same way.

\subsection{The mapping class group}
\label{sect:mcgsubsection}As a first step in our analysis of the mapping class group $\pi_0\Diff_{D^{2n}}(V_g)$, we observe that we may equally study orientation-preserving diffeomorphisms of $V_g$ that do not necessarily fix the embedded disc in the boundary $D^{2n}\subset \partial V_{g}$, or diffeomorphisms that preserve a disc $D^{2n+1}\subset \interior{V_g}$ in the interior set- or pointwise.

\begin{lem}\label{noboundary}For $n\ge2$ and discs $D^{2n}\subset \partial V_g$ and $D^{2n+1}\subset \interior{V_g}$, the canonical compositions
\begin{align*}\begin{gathered}
\pi_0\Diff_{D^{2n+1}}(V_g)\lra \pi_0\Diffor(V_g,D^{2n+1})\lra \pi_0\Diffor(V_g)\quad\text{and}\\\pi_0\Diff_{D^{2n}}(V_g) \lra\pi_0\Diffor(V_g,D^{2n}) \lra \pi_0\Diffor(V_g)\end{gathered}
\end{align*}
consist of isomorphisms. 
\end{lem}
\begin{proof}

Up to isotopy, preserving a disc in the interior setwise is equivalent to preserving a point, so the group $\pi_0 \Diffor(V_g,D^{2n+1})$ agrees with the group of path components of the subgroup $\Diffor(V_g,*)\subset \Diffor(V_g)$ of diffeomorphisms that fix the centre $*\in D^{2n+1}\subset \interior{V_g}$. As $V_g$ is $(n-1)$-connected, the long exact sequence of the fibration $\Diffor(V_g)\ra \interior{V_g}$ given by evaluating diffeomorphisms at $*$ implies that $\pi_0 \Diffor(V_g,*)$ agrees with $\pi_0 \Diffor(V_g)$, so the second map in the first composition in the statement is an isomorphism. Taking derivatives at $*$ yields a homotopy fibre sequence of the form
\begin{equation}\label{restrictionfibration}\Diff_{D^{2n+1}}(V_g)\lra \Diffor(V_g,*)\xlra{d}\SO(2n+1)\end{equation}
and hence an exact sequence
\[\pi_1\SO(2n+1)\xlra{t}\pi_0\Diff_{D^{2n+1}}(V_g)\lra\pi_0 \Diffor(V_g,*)\lra 0.\] 
On the subgroup $\SO(n)\subset\SO(2n+1)$, the derivative map $d$ has a section since $V_g$ admits a smooth $\SO(n)$-action with fixed point $*\in V_g$ whose tangential representation agrees with the restriction of the standard representation to the subgroup $\SO(n)$: take the $g$-fold equivariant boundary connected sums of $D^{n+1}\times S^{n}$ with $\SO(n)$ acting by rotating $S^n$ along an axis. As $\SO(n)\subset\SO(2n+1)$ is $(n-1)$-connected, this section ensures that the long exact sequence on homotopy groups of \eqref{restrictionfibration} splits in degrees $*\le n-1$ into short exact sequences and shows in particular that $t$ is trivial for $n\ge2$, so $\pi_0\Diff_{D^{2n+1}}(V_g)\cong \pi_0 \Diffor(V_g)$ holds as claimed. Replacing \eqref{restrictionfibration} by the fibration sequence $\Diff_{D^{2n}}(V_g)\ra \Diffor(V_g,*)\ra\SO(2n)$ induced by taking the derivative at the centre $*\in D^{2n}\subset \partial V_g$  of the disc in the boundary, the proof of the claim regarding the maps in the second composition of the statement proceeds analogous to the first part of the proof. 
\end{proof}

\subsubsection{The homology action}\label{sec:homology-action}To obtain further information on the mapping class group $\pi_0\Diff_{D^{2n}}(V_g)$, we consider its action on the $n$th integral homology of $V_g$ and its boundary $\partial V_g=W_{g}$, which we abbreviate by \[H_{W_{g,1}}\coloneq \oH_n(W_{g,1};\bfZ)\quad \text{and}\quad H_{V_{g}}\coloneq \oH_n(V_{g};\bfZ).\] This action preserves further algebraic structures, such as the intersection pairing
\[\lambda \colon H_{W_{g,1}}\otimes H_{W_{g,1}}\lra\bfZ,\] which equips $H_{W_{g,1}}$ with a nondegenerate $(-1)^n$-symmetric form by Poincaré duality. In addition, any automorphism of $H_{W_{g,1}}$ induced by an orientation-preserving diffeomorphism of $W_{g,1}$ has to preserve the function
\[\alpha \colon H_{W_{g,1}}\lra \pi_{n}\BSO(n)\] given by representing a homology class by an embedded $n$-sphere and taking its normal bundle. Wall \cite[Thm 2]{Wall1} has shown that $\lambda$ and $\alpha$ satisfy the relations
\begin{enumerate}
\item $\lambda(x,x)=\partial_n \alpha(x)$ 
\item\label{item:ii-alpha}  $\alpha(x+y)=\alpha(x)+\alpha(y)+\lambda(x,y)\cdot \tau_{S^n}$ 
\end{enumerate}
as long as $n\ge3$, where \[\partial_n\colon \pi_n\BSO(n)\lra\pi_{n-1}S^{n-1}\cong\bfZ\] is induced by the fibration $S^{n-1}\ra \BSO(n-1)\ra\BSO(n)$ and $\tau_{S^n}\in \pi_n\BSO(n)$ is the class representing the tangent bundle of the $n$-sphere. In sum, we arrive at a morphism \[\pi_0\Diff(W_{g,1})\lra G_g\coloneq \Aut(H_{W_{g,1}},\lambda,\alpha)\] to the subgroup $G_g\subset \GL(H_{W_g})$ of automorphisms preserving $\lambda$ and $\alpha$, which is surjective by \cite[Lem.\,10]{Wall2}. However, we are interested in the mapping class group $\pi_0\Diff_{D^{2n}}(V_g)$ and not every automorphism in $G_g$ is realised by a diffeomorphism of $W_{g,1}$ that extends to one of $V_g$; it would at least have to preserve the Lagrangian subspace \[K_g\coloneq \ker\big(H_{W_{g,1}}\ra H_{V_{g}}\big),\] so there is a canonical map $\pi_0\Diff_{D^{2n}}(V_g)\ra G_g^{\ext}$ to the subgroup \[G_g^{\ext}\coloneq\{\Phi\in G_g\mid \Phi(K_g)\subset K_g\}\] of automorphisms that preserve this Lagrangian, given by acting on the homology of the boundary. Using   the canonical isomorphism $H_{W_{g,1}}/K_g\cong H_{V_g}$, the subgroup $G^{\ext}_g$ maps further to $\GL(H_{V_g})$. The resulting composition
\begin{equation}\label{equ:homologyaction}\pi_0\Diff_{D^{2n}}(V_g)\lra G_g^{\ext}\lra \GL(H_{V_g})\end{equation} agrees with the action on the homology of $V_g$ and one may ask for the (co)kernel of the three maps involved. Extending work of Wall \cite{Wall2,Wall3}, we express the answer in \cref{thm:mcg} below in terms of an exact sequence of $\GL(H_{V_g})$-modules
\begin{equation}\label{equ:sesGLmodules}0\lra H_{V_g}^{\vee}\otimes S\pi_n\SO(n)\lra N_g\lra M_g\lra 0,\end{equation}
where
\begin{enumerate}[label=(\arabic*)]
\item $S\pi_n\SO(n)\subset \pi_n\SO(n+1)$ is the image of the stabilisation $\pi_n\SO(n)\ra\pi_n\SO(n+1)$,
\item\label{item:ii} $M_g\subset (H_{V_g}\otimes H_{V_g})^{\vee}$ is the submodule of bilinear forms $\mu\in (H_{V_g}\otimes H_{V_g})^{\vee}$ that are 
\begin{itemize}\item $(-1)^{n+1}$-symmetric and satisfy
\item $\mu(x,x)\in \im(\partial_{n+1}\colon \pi_{n}\SO(n+1)\ra\bfZ)$ for $x\in H_{V_g}$, and 
\end{itemize}
\item $N_g\subset M_g\oplus (\pi_n\SO(n+1))^{H_{V_g}}$ is the submodule of pairs $(\mu,\beta)$ of a bilinear form $\mu\in M_g$ and a function $\beta\colon H_{V_g}\ra \pi_n\SO(n+1)$ that fulfil the conditions 
\begin{itemize}\item$\mu(x,x)=\partial_{n+1}\beta(x)$ for $x\in H_{V_g}$ and
\item $\beta(x+y)=\beta(x)+\beta(y)+\mu(x,y)\cdot \tau_{S^{n+1}}$ for $x,y\in H_{V_g}$,\end{itemize}
\end{enumerate}
all equipped with the evident $\GL(H_{V_g})$-action through $H_{V_g}$. Here we denoted the integral dual of a $G$-module $M$ by $M^\vee\coloneq\Hom(M,\bfZ)$. 

\begin{rem}\label{rem:matrixdescriptionMg}The image of the map $\partial_{n+1}\colon \pi_{n}\SO(n+1)\ra \pi_nS^n\cong\bfZ$ is generated by the order of the tangent bundle $\tau_{S^n}\in\pi_{n-1}\SO(n)$, so we have (see \cite[§1B)]{Levine})
\begin{equation}\label{equ:casedistinction}\im(\partial_{n+1})=
\begin{cases}
0&\mbox{for }n\text{ even}\\
\bfZ&\mbox{for }n=1,3,7\\
2\cdot \bfZ&\mbox{otherwise},
\end{cases}\end{equation}
which exhibits the condition $\mu(x,x)\in\im(\partial_{n+1})$ in \ref{item:ii} as vacuous unless $n\neq1,3,7$ is odd. Moreover, this shows that, after choosing a basis $H_{V_g}\cong \bfZ^{g}$, the module $M_g$ can be described equivalently in terms of $(-1)^{n+1}$-symmetric integral $(g\times g)$-matrices, with even diagonal entries if $n\neq1,3,7$ is odd.
\end{rem}

\begin{thm}\label{thm:mcg}Let $n\ge3$.
\begin{enumerate}
\item The action of $\pi_0\Diff_{D^{2n}}(V_g)$ on $H_{V_g}=\oH_n(V_g;\bfZ)$ gives rise to an extension 
\[0\lra N_g\lra \pi_0\Diff_{D^{2n}}(V_g)\lra \GL(H_{V_g})\lra 0.\]
\item The morphism $G_g^{\ext}\ra \GL(H_{V_g})$ fits into an extension of the form
\[0\lra M_g\lra G^{\ext}_g\lra \GL(H_{V_g})\lra 0.\]
\item The action of $\pi_0\Diff_{D^{2n}}(V_g)$ on $H_{W_{g,1}}=\oH_n(W_{g,1};\bfZ)$ induces an extension
\[0\lra H_{V_g}^\vee\otimes S\pi_n\SO(n)\lra \pi_0\Diff_{D^{2n}}(V_g)\lra G^{\ext}_g\lra 0.\]
\end{enumerate}
Moreover, the induced outer actions of these extensions is as specified above and the second extension admits a preferred splitting.
\end{thm}

\begin{rem}
For a complete description of $\pi_0\Diff_{D^{2n}}(V_g)$, one still needs to determine the extension problems of the first or third part of the theorem, which we do not pursue at this point. Similar extensions by Kreck \cite{Kreck} describing the closely related mapping class group $\pi_0\Diff_\partial(W_{g,1})$ for $n\ge3$ have been resolved in \cite{KrannichMCG} for $n$ odd.
\end{rem}

\begin{proof}[Proof of \cref{thm:mcg}]
We begin with three preparatory remarks.
\begin{enumerate}[label=(\arabic*)]
\item Results of Wall we shall use rely on are phrased in terms of pseudoisotopy instead of isotopy, but these notions agree in our situation by Cerf's work \cite{Cerf}.
\item Justified by \cref{noboundary}, we do not distinguish between seemingly different variants of $\pi_0\Diff_{D^{2n}}(V_g)$ fixing various discs point- or setwise.
\item We identify $K_g$ canonically with the dual ${H_{V_g}}^\vee$ as a $G^{\ext}_g$-module via the isomorphism induced by the form $\lambda$, and dually $H_{V_g}$ with $K_g^\vee$.
\end{enumerate}
Wall \cite[Lem.\,10]{Wall2} showed that the action $\pi_0\Diff_{D^{2n}}(V_g)\ra \GL(H_{V_g})$ is surjective and identified its kernel with those isotopy classes of diffeomorphisms $\varphi$ that are homotopic to the identity. Moreover, in \cite[p.\,298]{Wall3}, he defined a complete obstruction $(\mu_\varphi,\beta_\varphi)\in N_g$ for such a homotopically trivial diffeomorphism to be isotopic to the identity. By \cite[Lem.\,12--13]{Wall2}, these obstructions are additive and exhaust $N_g$, so the resulting function
\begin{equation}\label{equ:Ngkernel}\ker\big(\pi_0\Diff(V_g)\ra\GL(H_{V_g})\big)\lra N_g\end{equation} is an isomorphism of groups, which establishes the first of the three claims. To demonstrate the second, note that, as automorphisms in $G_g^{\ext}$ preserve the form $\lambda$, the composition
\[G^{\ext}_g\lra \GL(H_{V_g})\xlra{((-)^{-1})^{\vee}}\GL(K_g)\] agrees with the restriction to $K_g$. In particular, this shows that the kernel of the first map in this composition acts trivially on $K_g$, so there is a canonical monomorphism
\[\ker(G^{\ext}_g\ra \GL(H_{V_g}))\xlrahook{\Psi}\Hom(H_{V_g},K_g)\cong (H_{V_g}\otimes H_{V_g})^{\vee}\] defined by sending $\phi\in G_g^{\ext}$ to the linear map $\Psi(\phi)\colon H_{V_g}\ra K_g$ induced by the difference $(\phi-\id)\colon H_{W_{g,1}}\ra K_g$. This leaves us with identifying the image of $\Psi$ with $M_g$ for which it is helpful to note that a morphism $f\in \Hom(H_{V_g},K_g)$ lies in the subspace \[M_g\subset (H_{V_g}\otimes H_{V_g})^\vee\cong\Hom(H_{V_g},K_g)\] if and only if $f^{\vee}=(-1)^{n+1}f$ and $\lambda(f(x),\tilde{x})\in \im(\partial_{n+1})$ holds for all $x\in H_{V_g}$. Here $\tilde{x}\in H_{W_{g,1}}$ is a choice of preimage of $x$ under the projection $H_{W_{g,1}}\ra H_{V_g}$, but the value $\lambda(f(x),\tilde{x})$ is independent of this choice $\tilde{x}$ since $K_g$ is Lagrangian. For elements of the form $\Psi(\phi)$, the first property $\Psi(\phi)^{\vee}=(-1)^{n+1}\Psi(\phi)$ follows from the fact that $\phi$ preserves the form $\lambda$. To see the second, we note that, since $K_g$ is Lagrangian, the function $\alpha$ is additive on $K_g$, so it vanishes on it since the images of the second $S^n$-factors of the connected sum $W_{g,1}\cong \sharp^g(S^n\times S^n)\backslash D^{n}$ in $V_g=\natural^g (D^{n+1}\times S^n)$ induce a basis of $K_g$ and have trivial normal bundle. Using the fact that $\phi\in G_g^{\ext}$ preserves $\alpha$ and property \ref{item:ii-alpha} of $\alpha$, we compute \[\alpha(\tilde{x})=\alpha\big(\Psi(\phi)(x)+\tilde{x}\big)=\alpha\big(\Psi(\phi)(x)\big)+\alpha(\tilde{x})+\lambda\big(\Psi(\phi)(x),\tilde{x}\big)\cdot \tau_{S^n}=\alpha(\tilde{x})+\lambda\big(\Psi(\phi)(x),\tilde{x}\big)\cdot \tau_{S^n}\] and conclude that $\lambda\big(\Psi(\phi)(x),\tilde{x}\big)\cdot \tau_{S^n}$ vanishes, so $\lambda\big(\Psi(\phi)(x),\tilde{x}\big)\in \im(\partial_{n+1})$ holds as claimed. This proves that the image of $\Psi$ is contained in $M_g$, and to show that it agrees with it, we consider the commutative diagram
\[
\begin{tikzcd}[row sep=0.5cm]
0\rar&N_g\dar\rar&\pi_0\Diff(V_g)\dar\rar&\GL(H_{V_{g}})\arrow[d,equal]\rar&0\\
0\rar&{\ker(G_g^{\ext}\ra\GL(H_{V_g}))}\rar\arrow[d,hook,"\Psi",swap]&G_g^{\ext}\rar&\GL(H_{V_{g}})\rar&0\\
&M_g&&&&
\end{tikzcd}
\] whose vertical arrow $N_g\ra \ker(G_g^{\ext}\ra\GL(H_{V_g}))$ is induced by the isomorphism \eqref{equ:Ngkernel}. By \cite[Lem.\,24]{Wall3}, the vertical composition in the diagram agrees with the projection $N_g\ra M_g$ in \eqref{equ:sesGLmodules}, so it is surjective. Consequently, $\Psi$ is surjective as well and hence an isomorphism, which proves the second claim of the statement. To show the third, we observe that, given that $\Psi$ is an isomorphism and $N_g\ra M_g$ is surjective, the diagram implies that $\pi_0\Diff(V_g)\ra G_g^{\ext}$ is surjective and moreover that its kernel agrees with the kernel of $N_g\ra M_g$, which is $H_{V_g}^\vee\otimes S\pi_n\SO(n)$ as claimed.

We now identify the actions as asserted. For the second extension, one can argue as follows: an automorphism $\varphi\in \GL(H_{V_g})$ acts on $\ker(G_g^{\ext}\ra\GL(H_{V_g}))$ by conjugating with a choice of lift $\widetilde{\varphi}\in G_g^{\ext}$, so the isomorphism $\Psi$ is equivariant simply because of the identity $(\widetilde{\varphi}(\phi-\id)\widetilde{\varphi}^{-1})=(\widetilde{\varphi}\phi\widetilde{\varphi}^{-1}-\id)$ in $\Hom(H_{V_g},K_g)$ for all $\phi\in\Hom(H_{V_g},K_g)$. For the first extension, we use that the left vertical composition in the diagram above agrees with the projection $N_g\ra M_g$, so it suffices to show that the composition
 \begin{equation}\label{equ:testmiddlespheres}\ker(\pi_0\Diff_{D^{2n}}(V_g)\ra\GL(H_{V_g}))\xlra{\cong} N_g\lra (\pi_n\SO(n+1))^{H_{V_g}}\end{equation} of \eqref{equ:Ngkernel} with the projection is equivariant. From Wall's definition \cite[p.\,267]{Wall2} of the invariant $\beta_\varphi$ of an element $\varphi$ in the kernel, we see that this composition can be described as follows: representing a homology class $[e]\in H_{V_g}$ by an embedded sphere $e\colon S^n\rightarrow V_g$, we can alter $\varphi$ by an isotopy such that it preserves $e$ pointwise. In this case, the derivative of $\varphi$ restricts an automorphism of the normal bundle $\nu(e)\cong\varepsilon^{n+1}$ which induces an element $\beta_\varphi([e])\in\pi_n\SO(n+1)$. This uses a trivialisation of $\nu(e)$, but the resulting element $\beta_\varphi([e])$ is independent of this choice. From this description, the claimed equivariance is straight-forward to check, and the identification of the action of the last sequence follows from that of the first two by chasing through the diagram obtained by extending the diagram above by taking vertical kernels.
 
To see that the second extension splits, note that the second $S^n$-factors in the connected sum decomposition of $W_{g,1}=\sharp^gS^n\times S^n\backslash\interior{D^{2n}}\subset \natural^g D^{n+1}\times S^n=V_g$ induce a splitting of the canonical map $H_{W_{g,1}}\ra H_{V_g}$ and thus an isomorphism of the form $K_g\oplus H_{V_g}\cong H_{W_{g,1}}$. Using this splitting of the homology, we can define a morphism $\GL(H_{V_g})\ra G_{g}^{\ext}$ by assigning $\phi\in \GL(H_{V_g})$ the automorphism
\[H_{W_{g,1}}\cong K_g\oplus H_{V_g}\xra{(\phi^{-1})^{\vee}\oplus \phi}K_g\oplus H_{V_g}\cong H_{W_{g,1}},\] which clearly preserves $K_g$. Moreover, its induced automorphism of $H_{W_{g,1}}$ agrees with $\phi$ by construction, so we obtain a splitting as desired. 
\end{proof}

\begin{rem}With respect to the basis $H_{W_{g,1}}\cong\bfZ^{2g}$ suggested by the connected sum decomposition $W_{g,1}\cong\sharp^g(S^n\times S^n)\backslash\interior{D^{2n}}$, the subgroup $G_g^{\ext}\subset \GL(H_{W_{g,1}})\cong \GL_{2g}(\bfZ)$ agrees with the group of block matrices of the form \[\left(\begin{matrix} A&M\\ 0 &(A^{-1})^T\end{matrix}\right)\] with $M\in M_g$, using the matrix description of $M_g$ explained in \cref{rem:matrixdescriptionMg}. From this point of view, the splitting $\GL_g(\bfZ)\ra G_g^{\ext}$ described in the proof of \cref{thm:mcg} is the obvious one that sends a matrix $A\in\GL_g(\bfZ)$ to the block diagonal matrix with $M=0$.
\end{rem}

\subsection{Stable framings}
\label{sect:stableframedmcg}Choosing the canonical map $\sfr\colon \EO\ra\BO$ as a stable tangential structure in the sense of \cref{sec:stable-tangential-structures}, the space of $\sfr$-structures 
\[\Bun_{D^{2n}}(\tau_{V_g}^s,\sfr^*\gamma;\ell_0)\]
as defined in that section is the space of stable framings of $V_g$ relative to a fixed stable framing $\ell_0\colon \tau_{V_g}^s|_{D^{2n}}\ra \sfr^*\gamma$. We denote the stabiliser of the canonical action of $\pi_0\Diff_{D^{2n}}(V_g)$ on the set of components $\pi_0\Bun_{D^{2n}}(\tau_{V_g}^s,\sfr^*\gamma;\ell_0)$ and its image in $G_g^{\ext}$ by \[\pi_0\Diff_{D^{2n}}(V_g)_\ell\subset\pi_0\Diff_{D^{2n}}(V_g)\quad\text{and}\quad G_{g,\ell}^{\ext}\subset G_g^{\ext}.\]
Note that $\pi_0\Diff_{D^{2n}}(V_g)_\ell$ agrees with the image of the canonical  map $\BDiff^{\sfr}_{D^{2n}}(V_g)\ra\BDiff_{D^{2n}}(V_g)$ on fundamental groups based at the point induced by the stable framing $\ell$, or equivalently, with the kernel of the crossed homomorphism
\begin{equation}\label{equ:crossedhom}\pi_0\Diff_{D^{2n}}(V_g)\lra H_{V_g}^\vee \otimes \pi_n\SO\end{equation} given by acting on $\ell$. This uses the identification \[H_{V_g}^\vee \otimes \pi_n\SO\cong\Hom(H_{V_g},\pi_n\SO)\cong\pi_0\Maps_{D^{2n}}(V_g,\SO)\cong\pi_0\Bun_{D^{2n}}(\tau_{V_g}^s,\theta^*\gamma;\ell_0)\] whose first two isomorphisms are the evident ones and whose third is induced by the choice of stable framing $\ell$, using that $\pi_0\Bun_{D^{2n}}(\tau_{V_g}^s,\sfr^*\gamma;\ell_0)$ is a torsor over the group $\pi_0\Maps_{D^{2n}}(V_g,\SO)$ equipped with the pointwise multiplication.

\begin{rem}\label{rem:unstableframings}As the stabilisation map $\SO(2n+1)\ra \SO$ is $2n$-connected, the induced map $\Maps_{D^{2n}}(V_g,\SO(2n+1))\ra \Maps_{D^{2n}}(V_g,\SO)$ is $n$-connected, which implies that there is no difference between equivalence classes of stable and unstable framings, so the discussion of this subsection applies equally well to unstable framings instead of stable ones.
\end{rem}

To relate the subgroups $\pi_0\Diff_{D^{2n}}(V_g)_\ell\subset \pi_0\Diff_{D^{2n}}(V_g)$ and $G_{g,\ell}^{\ext}\subset G_{g}^{\ext}$ to the sequences established in \cref{thm:mcg}, we define the $\GL(H_{V_g})$-submodule
\begin{enumerate}[label=(\arabic*)]
\item $N_g^{\sfr}\subset N_g$ as the intersection of $N_g$ with $M_g\oplus\langle\tau_{S^{n+1}} \rangle^{H_{V_g}}$, where $\langle\tau_{S^{n+1}} \rangle$ is the subgroup of $\pi_n\SO(n+1)$ generated by the tangent bundle $\tau_{S^{n+1}}$ and 
\item $M_g^{\sfr}\subset M_g$ as the collection of $(-1)^{n+1}$-symmetric bilinear forms $\mu\in (H_{V_g}\otimes H_{V_g})^{\vee}$ that are even, i.e.\,$\mu(x,x)\in2\cdot\bfZ$ for all $x\in H_{V_g}$, which is automatic if $n$ is even.
\end{enumerate}
Standard arguments involving the long exact sequences in homotopy groups of the usual fibration $\SO(d)\ra\SO(d+1)\ra S^d$ (cf.\,\cite[§1B)]{Levine}) show that the sequence \eqref{equ:sesGLmodules} restricts to an exact sequence of $\GL(H_{V_g})$-modules of the form
\[0\lra H_{V_g}^\vee\otimes \ker\big(S\pi_n\SO(n)\ra\pi_n\SO\big)\lra N_g^{\sfr}\lra M_g^{\sfr}\lra 0.\]

\begin{prop}\label{prop:mcgframed}Let $n\ge3$.
\begin{enumerate}
\item The action of $\pi_0\Diff_{D^{2n}}(V_g)$ on the set $\pi_0\Bun_{D^{2n}}(\tau_{V_g}^s,\sfr^*\gamma;\ell_0)$ of equivalence classes of stable framings is transitive. 
\item\label{prop:extensions-stably-framed-mcg} For any stable framing $\ell\colon \tau_{V_g}^s\ra \sfr^*\gamma$, the sequences of \cref{thm:mcg} restrict to exact sequences of the form
\begin{equation*}
\begin{gathered}
0\lra N_g^{\sfr}\lra \pi_0\Diff_{D^{2n}}(V_g)_\ell\lra \GL(H_{V_g})\lra 0,\\
0\lra M_g^{\sfr}\lra G_{g,\ell}^{\ext}\lra \GL(H_{V_g})\lra 0,\text{ and}\\
0\lra H_{V_g}^\vee\otimes \ker\big(S\pi_n\SO(n)\ra\pi_n\SO\big)\lra \pi_0\Diff_{D^{2n}}(V_g)_\ell\lra G_{g,\ell}^{\ext}\lra 0.
\end{gathered}
\end{equation*}
\end{enumerate}
\end{prop}

\begin{rem}\label{rem:arithmeticdifferenceframed}The calculation of $\im(\partial_{n+1})$ in \cref{rem:matrixdescriptionMg} shows that the inclusion $M_g^{\sfr}\subset M_g$ is an equality for $n\neq1,3,7$, so the same holds for $G_{g,\ell}^{\ext}\subset G^{\ext}_g$ as a result of \cref{prop:mcgframed}.\end{rem}

\begin{proof}[Proof of \cref{prop:mcgframed}]
The first sequence of \cref{thm:mcg} fits into a diagram
\[
\begin{tikzcd}[row sep=0.5cm]
0\rar&N_g\arrow[dr, bend right=20]\rar&\pi_0\Diff_{D^{2n}}(V_g)\rar\dar{\tau}&\GL(H_{V_g})\rar&0\\
&&H_{V_g}^\vee\otimes \pi_n\SO&&
\end{tikzcd}
\]
where $\tau$ is the crossed homomorphism \eqref{equ:crossedhom}, which is up to isomorphism given by the action of $\pi_0\Diff_{D^{2n}}(V_g)$ on $\ell$, so the first part of the statement is equivalent to the surjectivity of $\tau$. The diagonal arrow is the morphism which assigns an element $(\mu,\beta)\in N_g$ the composition of $\beta\colon H_{V_g}\ra \pi_n\SO(n+1) $ with the stabilisation map $\pi_n\SO(n+1)\ra \pi_n\SO$. As $\tau_{S^{n+1}}\in \pi_n\SO(n+1)$ is stably trivial, it follows from the second defining property of $N_g$ that this composition is additive, so indeed defines an element of ${H_{V_g}}^\vee\otimes \pi_n\SO\cong\Hom(H_{V_g},\pi_n\SO)$. As a next step, observe that the diagonal map is surjective, since the projection
\[N_g\subset M_g\oplus (\pi_n\SO(n+1))^{H_{V_g}}\ra(\pi_n\SO(n+1))^{H_{V_g}}\] has a section over the subspace of linear maps ${H_{V_g}}^\vee\otimes \pi_n\SO(n+1)\subset (\pi_n\SO(n+1))^{H_{V_g}}$ by setting the $M_g$-coordinate to zero and because the stabilisation map $\pi_n\SO(n+1)\ra \pi_n\SO$ is surjective. This reduces the first claim of the statement to the commutativity of the triangle, which follows from the geometric description of the composition \eqref{equ:testmiddlespheres} we gave in the proof of \cref{thm:mcg} in a straight-forward manner.

The kernel of the vertical map agrees with the stabiliser $\pi_0\Diff_{D^{2n}}(V_g)_\ell$, so the surjectivity of the diagonal arrow in the diagram also shows that this stabiliser surjects onto $\GL(H_{V_g})$ and that the kernel of the restriction $\ker(\tau)\ra\GL(H_{V_g})$ agrees with the kernel of the diagonal arrow. But this kernel is exactly the submodule $N_g^{\sfr}\subset N_g$, because the kernel of the stabilisation map $\pi_n\SO(n+1)\ra \pi_n\SO$ agrees with the kernel of $\pi_n\SO(n+1)\ra\pi_n\SO(n+2)$, i.e.\,the subgroup generated by $\tau_{S^{n+1}}$. This establishes the first sequence of the second claim. For the second, note that the surjectivity of the morphism \[\pi_0\Diff_{D^{2n}}(V_g)_\ell\lra \GL(H_{V_g})\]  implies that also the morphism \[G_{g,\ell}^{\ext}\lra \GL(H_{V_g})\] is surjective. Comparing the first two sequences of \cref{thm:mcg}, we see that its kernel agrees with the image of $N_g^{\sfr}$ in $M_g$ under the projection $N_g\ra M_g$, i.e.\,with those bilinear forms $\mu\in M_g$ that satisfy $\mu(x,x)\in \im(\partial_{n+1}\langle \tau_{S^{n+1}}\rangle)$. The image of $\tau_{S^{n+1}}\in \pi_n\SO(n+1)$ under $\partial_{n+1}$ is the Euler characteristic of $S^{n+1}$ (see \cite[§1B)]{Levine}), so the kernel in question agrees with $M_g^{\sfr}$ as claimed. To establish the last sequence, one first observes that \[\pi_0\Diff_{D^{2n}}(V_g)_\ell\lra G_{g,\ell}^{\ext}\] is surjective by definition of the target, and then compares the first two sequences in the claim to see that its kernel agrees with the kernel of $N_g^{\sfr}\ra M_g^{\sfr}$, which agrees with $H_{V_g}^\vee\otimes \ker(S\pi_n\SO(n)\ra\pi_n\SO)$ as already noted in the discussion prior to this proof.
\end{proof}

\begin{lem}\label{lem:negativeidentity}The negative of the identity $-\id\in \GL(H_{W_g})$ is contained in the subgroup $G_{g,\ell}^{\ext}\subset\GL(H_{W_g})$ for all stable framings $\ell\colon\tau_{V_g}^s\ra \sfr^*\gamma$.
\end{lem}
\begin{proof}
We first restrict to the case $g=1$ and a particular choice of stable framing, namely that induced by the standard embedding $D^{n+1}\times S^n\subset \bfR^{n+1}\times \bfR^{n+1}$. The diffeomorphism 
\begin{equation}\label{equ:explicitdiff}\mapnoname{D^{n+1}\times S^n}{D^{n+1}\times S^n}{((x_1,\ldots,x_{n+1}),(y_1,\ldots,y_{n+1}))}{((-x_1,x_2,\ldots,x_{n+1}),(-y_1,y_2,\ldots,y_{n+1}))}\end{equation} is orientation preserving, maps to $-\id\in \GL(H_{W_1})$, and has constant derivative, so it preserves the stable framing as required. It does not preserve a disc in the boundary, but by \cref{noboundary}, we can rectify this by an isotopy. To extend this argument to higher genera, we take the boundary connected sum of two copies of this diffeomorphism using the fixed discs in the boundary to obtain a diffeomorphism of $V_2$ whose image in $\GL(H_{W_2})$ is $-\id$, which we can again isotope so it preserves a disc as required. Continuing like this yields a sequence of isotopy classes in $\pi_0\Diff_{D^{2n}}(V_g)$ for all $g\ge0$ that satisfy the requirements of the claim. This establishes the statement for one specific stable framing for each $g$, but implies the general case, the reason being that $-\id\in \GL(H_{W_g})$ is central and all subgroups $G_{g,\ell}^{\ext}\subset G_g^{\ext}$ are conjugate, because the different stabilisers $\pi_0\Diff_{D^{2n}}(V_g)_\ell$ are, as the action is transitive by \cref{prop:mcgframed}. This concludes the proof.
\end{proof}

\subsection{The homotopy mapping class group}
Not only the smooth mapping class groups $\pi_0\Diff_{D^{2n}}(V_g)$ and $\pi_0\Diff_\partial(W_{g,1})$ act on the homology $H_{W_{g,1}}$, also their homotopical cousins $\pi_0\hAut_{D^{2n}}(V_g,W_{g,1})$ and $\pi_0\hAut_\partial(W_{g,1})$ do. Restricting these actions to the images \[\pi_0\hAut^{\cong}_{D^{2n}}(V_g,W_{g,1})\subset \pi_0\hAut_{D^{2n}}(V_g,W_{g,1})\quad\text{ and }\quad\pi_0\hAut^\cong_\partial(W_{g,1})\subset \pi_0\hAut_\partial(W_{g,1})\] of the canonical maps  \[\pi_0\Diff_{D^{2n}}(V_g,W_{g,1})\ra \pi_0\hAut_{D^{2n}}(V_g,W_{g,1})\quad\text{and}\quad \pi_0\Diff_\partial(W_{g,1})\ra\pi_0\hAut_\partial(W_{g,1}),\] they land in the above described subgroups $G_g^{\ext}$ and $G_g$ of $\GL(H_{W_{g,1}})$, respectively.

\begin{lem}\label{lem:finitekernel}Let $n\ge3$. The morphisms induced by the action on $H_{W_{g,1}}=\oH_n(W_{g,1};\bfZ)$
\[\pi_0\hAut_{\partial}^{\cong}(W_{g,1})\lra G_g\quad\text{ and }\quad\pi_0\hAut_{D^{2n}}^{\cong}(V_g,W_{g,1})\lra G_g^{\ext}\] are surjective. Moreover, their kernels are finite and $p$-torsion free as long as $n<2p-4$.
\end{lem}

\begin{proof}In the course of this proof, we shall make frequent use of the fact that the homotopy groups $\pi_{n+k}(\vee^hS^n)$ with $n\ge3$ are $p$-torsion free for $k<2p-3$ as a result of the Hilton--Milnor theorem (see \eqref{equ:Hilton-Milnor} below) and the case $h=1$ due to Serre \cite[p.\,498, Prop.\,5]{Serre}. 

To begin with the actual proof, note that the two morphisms of the statement are certainly surjective, since this holds already for $\pi_0\Diff_{\partial}(W_{g,1})\ra G_g$ and $\pi_0\Diff(V_g)\ra G_g^{\ext}$ by \cref{thm:mcg} and the discussion it precedes. Evidently, the kernel of the first morphism of the statement is contained in that of $\pi_0\hAut_{\partial}(W_{g,1})\ra \GL(H_{W_{g,1}})$, which enjoys the claimed finiteness and torsion property by an application of the fibre sequence induced by restriction along the boundary inclusion \[\hAut_{\partial}(W_{g,1})\lra \hAut_{*}(W_{g,1})\lra \Maps_{*}(\partial W_{g,1},W_{g,1}),\] using the isomorphism $\pi_0\hAut_{*}(W_{g,1})\cong\GL(H_{W_{g,1}})$ induced by the homology action and the fact that the group $\pi_1\Maps_{*}(\partial W_{g,1},W_{g,1})\cong \pi_{2n}(\vee^{2g}S^n)$ is finite and without $p$-torsion  for $n<2p-3$. To deduce the claim for the kernel of second morphism from this, we consult the fibre sequence
 $\hAut_{\partial}(V_{g})\ra \hAut_{D^{2n}}(V_{g},W_{g,1})\ra \hAut_{\partial}(W_{g,1})$ to realise that it suffices to show that $\pi_0\hAut_{\partial}(V_{g})$ is finite and $p$-torsion free for $n<2p-4$. Using the fibre sequence
$\hAut_{\partial}(V_{g})\ra \hAut_{*}(V_{g})\ra \Maps(W_g,V_g)$
 and the observation that its fibre inclusion is trivial on path components since $\pi_0\hAut_{\partial}(V_{g})$ acts trivially on homology and $\pi_0\hAut_{*}(V_{g})\cong \GL(H_{V_g})$, we see that $\pi_0\hAut_{\partial}(V_{g})$ is a quotient of the fundamental group $\pi_1\Maps_*(W_g,V_g)$ based at the inclusion $\iota$. Finally, note that the fibre sequence
$\Maps_*(W_g,V_g)\ra \Maps_*(W_{g,1},V_g)\ra \Maps_*(\partial W_{g,1},V_g)$ induces an exact sequence \[\pi_{2n+1}(\vee^gS^n)^{\oplus g}\lra \pi_1(\Maps_*(W_g,V_g);\iota)\lra \pi_{n+1}(\vee^gS^n)^{\oplus 2g}\] whose outer groups are finite and $p$-torsion free for $n<2p-4$, so the claim follows.
\end{proof}

\subsection{Stabilisation}\label{sect:stabilisation}
To relate the automorphism spaces of the handlebody $V_g=\natural^gD^{n+1}\times S^n$ relative to the disc $D^{2n}\subset \partial V_g$ fixed in \cref{sect:autVg} to the ones of $V_{g+1}$, it is convenient to modify $V_g$ by introducing codimension $2$ corners at the boundary of the disc $D^{2n}\subset V_g$ in the boundary so that there is smooth boundary preserving embedding $c\colon (-1,0]\times D^{2n}\ra V_g$ whose restriction to ${\{0\}\times D^{2n}}$ agrees with the chosen disc. Abusing common terminology, we call such an embedding a \emph{collar}. Fixing another disc $D^{2n-1}\subset\partial D^{2n}$, we consider \[H\coloneq ([0,1]\times D^{2n})\natural (D^{n+1}\times S^n),\] where the boundary connected sum is performed away from the union \[D\coloneq \{0,1\}\times D^{2n}\cup [0,1]\times D^{2n-1}\subset [0,1]\times D^{2n},\] and think of $V_{g+1}$ as being obtained by gluing $V_g$ to $H$ along the two collared discs $D^{2n}\subset V_g$ and $\{0\}\times D^{2n}\subset H$, where we declare $\{1\}\times D^{2n}\subset H\subset V_{g+1}$ to the new distinguished disc in the boundary, which comes with a preferred collar. Given a tangential structure $\theta\colon B\ra\BO(d)$, a choice of bundle map $\ell_0\colon \varepsilon\oplus  \tau_{D^{2n}}\ra \theta^*\gamma_{2n+1}$ induces canonical $\theta$-structures on $\tau_{V_g}|_{D^{2n}}$ and $\tau_{V_{g+1}}|_{D^{2n}}$ by using the fixed collars, and also one on $\tau_{H}|_{D}$ by making use of the canonical trivialisation of $\tau_{[0,1]}$. With respect to these $\theta$-structures, which we generically denote by $\ell_0$, there is an evident gluing map for tangential structures
\[\Bun_{D^{2n}}(\tau_{V_g},\theta^*\gamma_{2n+1};\ell_0)\times \Bun_{D}(\tau_{H},\theta^*\gamma_{2n+1};\ell_0)\lra \Bun_{D^{2n}}(\tau_{V_{g+1}},\theta_{2n+1}^*\gamma;\ell_0)\] that is equivariant with respect to the gluing morphism \begin{equation}\label{equ:gluingmap1}\Diff_{D^{2n}}(V_g)\times \Diff_{D}(H)\lra \Diff_{D^{2n}}(V_{g+1})\end{equation} for diffeomorphisms. Taking homotopy orbits, this induces a map of the form
\[\BDiff_{D^{2n}}^\theta(V_g;\ell_0)\times \BDiff_{D}^\theta(H;\ell_0)\lra \BDiff_{D^{2n}}^\theta(V_{g+1};\ell_0)\] and hence a homotopy class of \emph{stabilisation maps} \begin{equation}\label{equ:gluingmap2}\BDiff_{D^{2n}}^\theta(V_g;\ell_0)\ra \BDiff_{D^{2n}}^\theta(V_{g+1};\ell_0)\end{equation} that in general depends on the choice of a component of $\BDiff_{D}^\theta(H;\ell_0)$, but not in any of the cases we shall be interested in, because of the following lemma.

\begin{lem}If $B$ is $n$-connected, then $\BDiff_{D}^\theta(H;\ell_0)$ is nonempty and connected.
\end{lem}
\begin{proof}
Up to smoothing corners, the pairs $(H,D)$ and $(V_1,D^{2n})$ are diffeomorphic, so there we have an equivalence $\BDiff_{D}^\theta(H;\ell_0)\simeq \BDiff_{D^{2n}}^\theta(V_1;\ell_0)$, which shows that the claim is equivalent to the transitivity of the action of $\pi_0\Diff_{D^{2n}}(V_1)$ on $\pi_0\Bun_{D^{2n}}(\tau_{V_1},\theta^*\gamma_{2n+1};\ell_0)$, using that the latter set is nonempty as $V_1$ is parallelizable. As $B$ is $n$-connected, it is a consequence of obstruction theory that every $\theta$-structure on $V_g$ is induced by a framing, so it suffices to consider the case $\theta\colon \EO(2n+1)\ra \BO(2n+1)$, which we have already settled as the first part of \cref{prop:mcgframed} (see also \cref{rem:unstableframings}).
\end{proof}

A similar discussion results in analogous stabilisation maps of the form \[\BlockBDiff^\Xi_{D^{2n}}(V_g;\ell_0)\lra \BlockBDiff^\Xi_{D^{2n}}(V_{g+1};\ell_0)\] for stable tangential structures $\Xi\colon B\ra \BO$ and bundle maps $\ell_0\colon \tau_{V_g}^s\ra \Xi^*\gamma$, which are compatible with the non-block variants defined above and are again unique up to homotopy if $B$ is $n$-connected. Moreover, using the splittings of the inclusions $H_{W_{g,1}}\ra H_{W_{g+1,1}}$ and $H_{V_g}\ra H_{V_{g+1}}$ suggested by the decomposition $V_{g+1}=V_g\cup H$, there are stabilisation maps for the respective linear groups on $H_{W_{g,1}}$ and $H_{V_g}$ given by extending automorphisms by the identity and these are related to the stabilisation maps described above via the action on the homology of $V_g$ and $W_{g,1}$ (see \cref{sect:mcgsubsection}), so there is a commutative diagram of compatible stabilisation maps that has the form
\[
\begin{tikzcd}[row sep=0.4cm]
\BDiff^{\Xi_{2n+1}}_{D^{2n}}(V_g;\ell_0)\rar\dar&\BlockBDiff^\Xi_{D^{2n}}(V_g;\ell_0)\rar\dar&\BG_g^{\ext}\rar\dar& \BGL(H_{V_g})\dar\\
\BDiff^{\Xi_{2n+1}}_{D^{2n}}(V_{g+1};\ell_0)\rar&\BlockBDiff^\Xi_{D^{2n}}(V_{g+1};\ell_0)\rar&\BG_{g+1}^{\ext}\rar&\BGL(H_{V_{g+1}}).
\end{tikzcd}
\]

\section{Relative homotopy automorphisms of handlebodies}\label{sect:hAut}
\cref{thm:rationalmodel} illustrates that the space of block diffeomorphisms $\BlockDiff_{D^{2n}}(V_g)$ is closely related to the space $\BlockhAut_{D^{2n}}(V_g,W_{g,1})$ of relative block homotopy automorphisms or, equivalently, to its non-block variant $\hAut_{D^{2n}}(V_g,W_{g,1})$ (see \cref{sect:hAutPrelim}). To access the homology of the classifying space of this space of homotopy automorphisms, one might try to study the Serre spectral sequence of the fibration sequence induced by taking components \[\BhAut^{\id}_{D^{2n}}(V_g,W_{g,1})\lra \BhAut_{D^{2n}}(V_g,W_{g,1})\lra \mathrm{B}\pi_0\hAut_{D^{2n}}(V_g,W_{g,1})\] for which one ought to know at least the homology of the fibre as a module over the group $\pi_0\hAut_{D^{2n}}(V_g,W_{g,1})$. This is what this section aims to compute---$p$-locally and in a range of degrees---by first calculating the $p$-local homotopy groups in a range using some tools from rational homotopy theory combined with an ad-hoc extension to the $p$-local setting tailored to our situation, and then pass from homotopy to homology groups.

\subsection{Conventions on gradings}\label{sec:gradings}
Essentially all objects in this section carry a $\bfZ$-grading, and we shall keep track of it throughout. For instance, we consider the (reduced) homology of a space $X$ always with its natural grading, even if it is supported in a single degree. 

We denote the $k$-fold suspension of a graded $R$-module $A$ over some commutative ring $R$ by $s^kA$, the graded $R$-module whose degree $k$ piece consists of $R$-module morphisms raising the degree by $k$ by $\Hom(A,B)$ for graded modules $A$ and $B$, the graded $R$-dual of $A$ by $A^\vee\coloneq \Hom(A,R[0])$ where $R[0]$ is the base ring concentrated in degree zero, and the subspace of elements of strictly positive degrees by $A^+\subset A$. For an ungraded $R$-module $M$ we write $M[k]$ for the graded $R$-module which is trivial in all degrees but $k$ where it agrees with $M$. The graded tensor product $A\otimes B$ is defined in the usual way. Note that $(s^kA)\otimes B=s^k(A\otimes B)=A\otimes(s^kB)$. The degreewise rationalisation or $p$-localisation of a graded $\bfZ$-module $A$ is denoted by $A_\bfQ$ or $A_{(p)}$ respectively, and we view it as a graded $\bfQ$- respectively $\bfZ_{(p)}$-module.

\subsection{Lie algebras and their derivations}\label{sec:lie-algs}
We consider differential graded (short dg) Lie algebras over a commutative ring $R$. However, most of the dg Lie algebras which we shall encounter actually have trivial differential. Examples include the free graded Lie algebra $\bfL(V)$ on a graded $R$-module $A$ or the onefold shift of the homotopy groups $\pi_{*+1}X$ of a based space $X$ with its canonical Lie algebra structure over $\bfZ$ given by the Whitehead bracket (except for a $2$-torsion subtlety that will not play a role for us). Given a dg Lie algebra $L$, we write $[L,L]\subset L$ for the graded subspace generated by brackets. An important principle in this section is that the homotopy type of mapping spaces is closely related to certain chain complexes of \emph{$f$-derivations} by which we mean the following: for a morphism $f\colon (L,d_L)\ra (L',d')$ of dg Lie algebras, an \emph{$f$-derivation of degree $k$} is a linear map $\theta\colon L\ra L'$ that raises the degree by $k$ and satisfies \[\theta([x,y])=[\theta(x),f(y)]+(-1)^{k|x|}[f(x),\theta(y].\] These derivations form the degree $k$ piece of the chain complex $\Der^f(L,L')$ of \emph{$f$-derivations} over $R$ whose differential is defined as $d(\theta)=d_{L'}\theta-(-1)^{|\theta|}\theta d_L$, so it vanishes if both $L$ and $L'$ have trivial differential. Given a cycle $\omega\in L$, we denote the subcomplex of $f$-derivations that vanish on $\omega$ by $\Der^f_{\omega}(L,L')\subset \Der^f(L,L')$. In the case $L=L'$ and $f=\id$, we abbreviate the complex of $\id$-derivations $\Der^{\id}(L,L)$ by $\Der(L)$.

\subsection{Rational homotopy theory, Quillen style}
Recall from \cite{QuillenRational} Quillen's functor $\lambda$, which assigns a simply connected based space $X$ a dg Lie algebra $\lambda(X)$ over the rationals, one of whose many properties is that it captures the rationalised homotopy Lie algebra of $X$ via a natural isomorphism $\oH_*(\lambda(X))\cong\pi_{*+1}(X)_\bfQ$ of graded Lie algebras, where $\oH_*(\lambda(X))=\ker(d_{\lambda(X)})/\im(d_{\lambda(X)})$ is the homology Lie algebra of $\lambda(X)$. A \emph{Lie model} of $X$ is a rational dg Lie algebra $L_\bfQ^X$ quasi-isomorphic to $\lambda(X)$. Such a model is called \emph{free} if the underlying graded Lie algebra of $L_\bfQ^X$ is isomorphic to a free graded Lie algebra $\bfL(V)$ on a graded $\bfQ$-vector space $V$ and \emph{minimal} if it is free and has decomposable differential, i.e.\,$d(L_\bfQ^X)\subset[L_\bfQ^X,L_\bfQ^X]$. Any simply connected based space has a minimal Lie model $L_\bfQ^X$, unique up to (non-canonical) isomorphism, and a based map between such spaces $f\colon X\ra Y$ gives rise to a map $f\colon L_\bfQ^X \ra L_\bfQ^Y$ between their minimal models.

\subsection{Derivations and mapping spaces}
As mentioned earlier, the homotopy theory of mapping spaces is tightly connected to derivations of dg Lie algebras. In the rational setting, this is made precise for instance by a result of Lupton--Smith \cite[Thm 3.1]{LuptonSmith}. The version of their result we shall need is marginally stronger than stated in \cite{LuptonSmith}, but follows from the given proof in a straight-forward way (see also \cite[Thm\,3.6]{BerglundMadsen}).

\begin{thm}[Lupton--Smith]\label{thm:LuptonSmith}Let $f\colon X\ra Y$ be a map between simply connected finite based CW-complexes, with minimal Lie model $f\colon L_\bfQ^X\ra L_\bfQ^Y$. There is an isomorphism
\[\pi_*(\Maps_*(X,Y);f)_\bfQ\xlra{\cong}\oH_*(\Der^{f}(L_\bfQ^X,L_\bfQ^Y))\] for $*\ge2$, which is natural in both $X$ and $Y$. For $X$ a co-H-space, this also holds for $*=1$.
\end{thm}

\subsection{A $p$-local generalisation}\label{sect:plocal}
From the point of view of Quillen's approach to rational homotopy theory, the spaces we shall be applying \cref{thm:LuptonSmith} to are of the simplest nature  possible: they are homotopy equivalent to boquets of equidimensional spheres. The minimal model of such a space $X\simeq \vee^gS^n$ for $n\ge 2$ agrees with the free graded Lie algebra \begin{equation}\label{equ:rationalLiefree}L_\bfQ^{X}\coloneq\bfL(s^{-1}H_\bfQ^X)\cong\pi_{*+1}X_\bfQ\quad\text{on}\quad H_\bfQ^X\coloneq \widetilde{\oH}_*(X;\bfQ),\end{equation} equipped with the trivial differential. Given a map between spaces of this kind, the induced map on minimals models is simply given by the induced map on rational homotopy groups. It is a consequence of the Hilton--Milnor theorem that, $p$-locally in small degrees with respect to $p$, the homotopy Lie algebra $\pi_{*+1}X$ is free even before rationalisation. To make this precise, we abbreviate the integral and $p$-local analogue of \eqref{equ:rationalLiefree} by \[L^X\coloneq \bfL(s^{-1}H^{X})\text{\ \ and\ \ }L_{(p)}^X\coloneq \bfL(s^{-1}H_{(p)}^{X})\text{\ \ \ where\ \ \ } H^X\coloneq \widetilde{\oH}_*(X;\bfZ)\text{\ \ and\ \ }H^X_{(p)}\coloneq \widetilde{\oH}_*(X;\bfZ_{(p)}).\] The inverse of the Hurewicz map $\oH_n(X;\bfZ_{(p)})\cong\pi_nX_{(p)}$  induces a map $L_{(p)}^X\ra \pi_{*+1}X_{(p)}$ of graded Lie algebras over $\bfZ_{(p)}$, which turns out to be an isomorphism in a range of degrees.

\begin{lem}\label{lem:torsioninwedge}\label{lem:tamewedges}For an odd prime $p$ and a based space $X$ that is homotopy equivalent to $ \vee^gS^n$ with $n\ge2$, the morphism
\[L_{(p)}^X\lra \pi_{*+1}X_{(p)}\] is an isomorphism on torsion free quotients. Moreover, the right hand side is torsion free in degrees $*<2p-4+n$, so the map is an isomorphism in this range.
\end{lem}

\begin{proof}As a preparation to the proof, note that by specialising the Hilton--Milnor theorem to $\vee^gS^n$, we have an isomorphism
\begin{equation}\label{equ:Hilton-Milnor}\textstyle{\pi_{i+1}(\vee^gS^n)\cong \bigoplus_{\omega\in L_g}\pi_{i+1}(S^{l(\omega)(n-1)+1})}\end{equation}
where $L_g$ denotes a Hall basis for the free \emph{ungraded} Lie algebra in $g$ ordered generators and $l(\omega)$ is the word-length of $\omega$. Here the map $\pi_{i+1}(S^{l(\omega)(n-1)+1})\ra \pi_{i+1}(\vee^gS^n)$ corresponding to $\omega\in L_g$ is given by mapping a class $x\in\pi_{i+1}(S^{l(\omega)(n-1)+1})$ to the composition $(\iota_\omega \circ x)$, where $\iota_\omega \in \pi_{l(\omega)(n-1)+1}(\vee^gS^{n})$ is the class obtained by taking Whitehead products of the canonical classes $\iota_i \in \pi_{n}(\vee^gS^n)$ for $1\le i\le g$ represented by the inclusions of the summands as guided by the Lie word $\omega\in L_g$. A proof can be extracted from  \cite{Whitehead}: combine XI.6.6 and the subsequent discussion with VII.2.6 and X.7.10.

To prove the asserted claim, we use that source and domain of the morphism in the statement are both degreewise finitely generated and that the rationalisation of this morphism agrees with \eqref{equ:rationalLiefree}, so to prove the first part of the claim, it suffices to show that all classes in $\pi_{*+1}X_{(p)}$ of infinite order are in the image. From \eqref{equ:Hilton-Milnor}, we see every class in $\pi_{k}X_{(p)}$ is a composition $(y \circ x)$ of some $x\in \pi_kS^m_{(p)}$ with $m\ge n$ and a class $y\in \pi_mX_{(p)}$ in the image of the map in question. The group $\pi_kS^m_{(p)}$ is finite unless $k=m$, where it is generated by the identity, or $m=2l$ and $k=4l-1$, where it is is generated by $[\id_{S^{2l}},\id_{S^{2l}}]$, since this element has Hopf invariant $2$ and we assumed $p$ to be odd. As $y\circ [\id_{S^{2l}},\id_{S^{2l}}]=[y,y]$ in $\pi_{4l-1}X_{(p)}$ and the image of $L_{(p)}^X\ra \pi_{*+1}X_{(p)}$ is closed under taking brackets, this implies the first part of the claim. The second part follows from Serre's result \cite[p.\,498, Prop.\,5]{Serre} that $\pi_kS^m$ is $p$-torsion free for $k-m<2p-3$ together with another application of \eqref{equ:Hilton-Milnor}.\end{proof}

As a result of \cref{lem:torsioninwedge}, every map $f\colon X\ra Y$ between bouquets of equidimensional spheres induces a morphism $f_*\colon L^X_{(p)}\ra L^Y_{(p)}$ by taking torsion free quotients of $p$-local homotopy groups, so the following extension of \cref{thm:LuptonSmith} might not come as a surprise.

\begin{prop}\label{prop:plocalLS}For an odd prime $p$ and a map $f\colon X\ra Y$ between based spaces $X\simeq \vee^g S^n$ and $Y\simeq\vee^h S^m$ with $n,m\ge 2$ , the map of \cref{thm:LuptonSmith} fits into a commutative square
\[
\begin{tikzcd}
\pi_*(\Maps_*(X,Y);f)_{(p)}\arrow[d,"(-)\otimes\bfQ"]\arrow[r]&\Der^{f}(L_{{(p)}}^{X},L_{{(p)}}^{Y})^+\arrow[d,"(-)\otimes\bfQ"]\\
\pi_*(\Maps_*(X,Y);f)_\bfQ\arrow[r,"\cong"]&\Der^{f}(L_{\bfQ}^{X},L_{\bfQ}^{Y})^+
\end{tikzcd}\quad\text{for }*>0,
\] 
which is natural in $X$ and $Y$ and whose upper arrow is an isomorphism for $*<2p-3-(n-m)$.
\end{prop}

\begin{rem}Dwyer's \emph{tame homotopy theory} \cite{DwyerTame} provides a $p$-local generalisation of Quillen's rational homotopy theory for primes $p$ that are just large enough with respect to the degree to prevent stable $k$-invariants from appearing. It is not unlikely that \cref{thm:LuptonSmith} could be generalised to this setting, but our layman extension \cref{prop:plocalLS} for bouquets of spheres suffices for the applications we have in mind.
\end{rem}

\begin{proof}[Proof of \cref{prop:plocalLS}]
We begin with a twofold simplification of the statement. Firstly, the claimed naturality is automatic, since the vertical maps are evidently natural, the bottom map is natural by \cref{thm:LuptonSmith}, and the right vertical map is injective, so it suffices to construct a top arrow with the desired properties for $X=\vee^gS^n$. Secondly, there is a commutative diagram
\[
\begin{tikzcd}
\Der^{f}(L_{{(p)}}^{X},L_{{(p)}}^{Y})\arrow[r,"\cong"]\arrow[d,"(-)\otimes\bfQ"]&\Hom(s^{-1}H^X_{(p)},L_{(p)}^Y)\arrow[d,"(-)\otimes\bfQ"]\\
\Der^{f}(L_{{\bfQ}}^{X},L_{{\bfQ}}^{Y})\arrow[r,"\cong"]&\Hom(s^{-1}H^X_\bfQ,L_{\bfQ}^Y),
\end{tikzcd}
\]
induced by restricting derivations to generators, which shows that it is enough to produce a dashed arrow making the diagram
\begin{equation}\label{equ:plocalmappingspace}
\begin{tikzcd}
\pi_*(\Maps_*(X,Y);f)_{(p)}\arrow[d,"(-)\otimes\bfQ"]\arrow[rr,dashed]&&\Hom(s^{-1}H^X_{(p)},L_{(p)}^Y)^+\arrow[d,"(-)\otimes\bfQ"]\\
\pi_*(\Maps_*(X,Y);f)_\bfQ\arrow[r,"\cong"]&\Der^{f}(L_{{\bfQ}}^{X},L_{{\bfQ}}^{Y})^+\arrow[r,"\cong"]&\Hom(s^{-1}H^X_\bfQ,L_{\bfQ}^Y)^+,
\end{tikzcd}
\end{equation}
commute. To do so, we consider the composition
\begin{equation}\label{equ:inverseiso}\pi_*(\Maps_*(X,Y);f)\xrightarrow[\cong]{(-f)_*}\pi_*(\Maps_*(X,Y);*)\cong\Hom(s^{-1}H^X,\pi_{*+1}Y)^+\end{equation} whose first isomorphism is given by acting with the inverse of $f$, using the loop space structure on $\Maps_*(X,Y)$, and whose second isomorphism is induced by mapping a class in $ \pi_k(\Maps_*(X,Y);*)$ represented by a pointed map $g\colon S^k\ra \Maps_*(X,Y)$  to the composition
\[\oH_n(X)\cong \oH_{n+k}(S^k\wedge X)\cong\pi_{n+k}(S^k\wedge X)\xlra{g_*}\pi_{n+k}(Y),\] involving the suspension isomorphism, the inverse of the Hurewicz map, and the adjoint of $g$. Postcomposing \eqref{equ:inverseiso} with the map given by $p$-localising and taking torsion free quotients results by \cref{lem:torsioninwedge} in a dashed map with the claimed connectivity property, so we are left to show that this choice does make the diagram \eqref{equ:plocalmappingspace} commute, i.e.\,that the rationalisation of \eqref{equ:inverseiso} agrees with the bottom map of \eqref{equ:plocalmappingspace}. The adjoint of a map $h\colon S^k\ra \Maps_*(X,Y)$ representing a class in $\pi_*(\Maps_*(X,Y);f)$ forms the top arrow of the diagram
\begin{equation}\label{equ:change-basepoints}
\begin{tikzcd}[row sep=0.45cm]
S^k\times X\arrow[r,"h"]\arrow[d]&Y\\
S^k_+\wedge X\arrow[d,"\simeq",swap]\arrow[ur]&\\
X\vee S^k\wedge X,\arrow[uur,"f\vee(-f)_*(h)",swap]
\end{tikzcd}
\end{equation}
whose middle diagonal arrow is induced by $h$ via the canonical homeomorphism $(S^k\times X)/(S^k\vee*)\cong S^k_+\wedge X$ and whose vertical equivalence is given as the composition
\[S^k_+\wedge X\xra{\id_{S^k_+}\wedge \nabla} S^k_+\wedge X\vee S^k_+ \wedge X\xra{c}X\vee S^k \wedge X\] using the co-H-space structure $\nabla$ of $X$ and the evident collapse map $c$. The map $(-f)_*(h)$ is the adjoint of a representative of the image of $h$ under the first map in \eqref{equ:inverseiso}, so \eqref{equ:change-basepoints} commutes up to changing $h$ within its class in $\pi_k(\Maps_*(X,Y);f)$. We thus obtain a rational model for the top arrow in \eqref{equ:change-basepoints} as the composition
\begin{equation}\label{equ:lie-model-of-map}\left(\bfL\left(s^{-1}H^X_{\bfQ}\oplus s^{-1+k}H_\bfQ^X\oplus s^{-1}H_\bfQ^{S^k}\right),d \right)\lra L_\bfQ^{S^k\vee S^k\wedge X}\xra{\pi_{*+1}(f\vee (-f)_*(g))\otimes\bfQ} L_{\bfQ}^Y,\end{equation} where the source is the Lie model of $S^k\times X$ described in \cite[Cor.\,2.2]{LuptonSmith}, i.e.\,its differential $d$ is trivial on $s^{-1}H^X_{\bfQ}\oplus  s^{-1}H_\bfQ^{S^k}$ and is on $s^{-1+k}H_\bfQ^X$ given as the composition
\[s^{-1+k}H_\bfQ^X\cong s^{-1}H^X_{\bfQ}\xra{(-1)^{k-1}[z,-]}\bfL\left(s^{-1}H^X_{\bfQ}\oplus s^{-1+k}H_\bfQ^X\oplus s^{-1}H_\bfQ^{S^k}\right)\]
where the first isomorphism is the canonical identification as ungraded vector spaces induced by the identity and $z\in s^{-1}H_\bfQ^{S^k}$ denotes the standard generator. The first map in the composition \eqref{equ:lie-model-of-map} takes the quotient by the dg Lie ideal generated by the subspace $s^{-1}H_\bfQ^{S^k}$ and the second map is defined as indicated. Using this particular choice of rational model in the definition of the isomorphism of \cref{thm:LuptonSmith} in \cite[p.\,176--177]{LuptonSmith}, the image of the class defined by $h$ under the bottom horizontal composition in \eqref{equ:plocalmappingspace} is precisely its image under \eqref{equ:inverseiso} after rationalisation, so the claim follows.
\end{proof} 

\subsection{$p$-local homotopy groups of $\BhAut^{\id}_{D^{2n}}(V_g,W_{g,1})$}
The theory set up in the previous paragraphs will allow us to compute the $p$-local homotopy groups of the classifying space $\BhAut^{\id}_{D^{2n}}(V_g,W_{g,1})$ of the identity component of the topological monoid of relative homotopy automorphisms as defined in \cref{sect:autVg} as a module over the group \[\pi_1\BhAut_{D^{2n}}(V_g,W_{g,1})\cong \pi_0\hAut_{D^{2n}}(V_g,W_{g,1}).\] 
More generally, we will compute the $p$-local homotopy groups of the spaces participating in the fibration sequence (see \cref{sect:autVg} for the notation)
\begin{equation}\label{equ:fundamental-homotopy-fibration}
\BhAut_{\partial}(V_g)\lra\BhAut_{D^{2n}}(V_g,W_{g,1})\lra\BhAut^{\ext}_{\partial}(W_{g,1})
\end{equation}
induced by restriction, together with the induced action of $\pi_0\hAut_{D^{2n}}(V_g,W_{g,1})$. To state the answer (and give the proof), we adopt the notation of in the previous subsection for the three manifolds $W_{g,1}$, $V_g$, and $\partial W_{g,1}$ involved, which are homotopy equivalent to bouquets of spheres. We do however omit the $g$-superscripts to increase readability, so write
\[H_{(p)}^W=\widetilde{\oH}_*(W_{g,1};\bfZ_{(p)}),\quad L_{(p)}^W=\bfL(s^{-1}H_{(p)}^W),\quad H_{(p)}^V=\widetilde{\oH}_*(V_{g};\bfZ_{(p)}),\quad\text{and}\quad L_{(p)}^V=\bfL(s^{-1}H_{(p)}^V)\]
and omit the $(p)$-subscripts to denote the integral analogues.
Moreover, we generically write $\iota$ for any combination of the inclusions \[\partial W_{g,1}\subset W_{g,1}\subset W_g\subset V_g.\] Finally, we let $\omega\in L_{W_{g,1}}$ be the class that represents the inclusion $S^{2n-1}=\partial W_{g,1}\subset W_{g,1}$ of the boundary of $W_{g,1}=\sharp^g(S^n\times S^n)\backslash\interior{D^{2n}}$, i.e.\,the attaching map 
\begin{equation}\label{equ:attaching-map}\textstyle{\omega=\sum_{i=1}^{g}[e_i,f_i]\in \pi_{2n-1}W_{g,1}},\end{equation} where $e_i,f_i\in\pi_nW_{g,1}$ correspond to the first respectively second $S^n$-summand in the $i$th summand of \[W_{g,1}\cong \sharp^{g}S^n\times S^n\backslash \interior{D^{2n}}\simeq\vee^g (S^n\vee S^n).\] 
Note that the inclusion $\partial W_{g,1}\subset V_g$ is trivial, since it factors over $W_g=W_{g,1}\cup_{\partial W_{g,1}}D^{2n}$.

\begin{rem}Note that $H^{W}$ and $H^{V}$ stand for the \emph{graded} $\bfZ$-module given by the reduced homology of $W_{g,1}$ and $V_g$, which shall not be confused with the \emph{ungraded} middle dimensional integral homology groups of these spaces that featured in \cref{sec:homology-action} as $H_{W_{g,1}}$ and $H_{V_g}$.
\end{rem}

\begin{thm}\label{thm:rationalmain}Let $n\ge2$ and $p$ an odd prime.
\begin{enumerate}
\item\label{label:rationamainI}  The inclusion $\pi_{0}\Maps_{\partial}(V_{g},V_g)\subset \pi_{0}\hAut_{\partial}(V_{g})$ is an equality. This group is abelian.
\item\label{label:rationamainII} In degrees $0<*<2p-3-n$, the boundary map of the fibration \eqref{equ:fundamental-homotopy-fibration} fits into a commutative diagram of graded $\bfZ_{(p)}$-modules with exact rows
\[
\begin{tikzcd}[column sep=0.5cm]
0\rar&\pi_{*+1}\BhAut^{\ext}_{\partial}(W_{g,1})_{(p)}\rar\dar{\partial}&[-0.2cm]s^{-(2n-1)}H^{W}_{(p)}\otimes L_{(p)}^{W}\dar{\iota_*\otimes\iota_*}\rar{\left[-,-\right]}& s^{-(2n-2)}[L_{(p)}^{W}, L_{(p)}^{W}]\rar\dar{\iota_*}&0\\
0\rar&\pi_{*}\BhAut_{\partial}(V_{g})_{(p)}\rar&s^{-(2n-1)}H^{V}_{(p)}\otimes L_{(p)}^{V}\rar{\left[-,-\right]}& s^{-(2n-2)}[L_{(p)}^{V}, L_{(p)}^{V}]\rar&0.
\end{tikzcd}
\]
 which is $\pi_0\hAut_{D^{2n}}(V_g,W_{g,1})$-equivariant with respect to the action on the leftmost column induced by \eqref{equ:fundamental-homotopy-fibration} and by the  action through $H^{W}$ and $H^{V}$ on the other columns.
 \item\label{label:rationamainIII} Rationally, the conclusion of \ref{label:rationamainII} holds in all positive degrees.\end{enumerate}
\end{thm}

A splitting of the canonical projection $\iota_*\colon H_{W_{g,1}}\ra H_{V_g}$ induces compatible splittings of the rightmost two columns of the diagram in \cref{thm:rationalmain} \ref{label:rationamainII}, so the boundary map $\partial$ of the fibration \eqref{equ:fundamental-homotopy-fibration} is $p$-locally split surjective in a range and we conclude the following.

\begin{cor}\label{cor:rationalmain-cor}Let $n\ge 2$ and $p$ an odd prime. In degrees $0<*<2p-4-n$, the graded $\bfZ_{(p)}[\pi_0\hAut_{D^{2n}}(V_g,W_{g,1})]$-module $\pi_{*+1}\BhAut_{D^{2n}}(V_g,W_{g,1})_{(p)}$ is isomorphic to the common kernel of the maps
\[s^{-(2n-1)}H^{W}_{(p)}\otimes L_{(p)}^{W}\xlra{\left[-,-\right]} s^{-(2n-2)}[L_{(p)}^{W}, L_{(p)}^{W}] \quad\text{and}\quad s^{-(2n-1)}H^{W}_{(p)}\otimes L_{(p)}^{W}\xlra{\iota_*\otimes\iota_*}s^{-(2n-1)}H^{V}_{(p)}\otimes L_{(p)}^{V}.\]
Rationally, this holds in all positive degrees.
\end{cor}

In particular, \cref{thm:rationalmain} and \cref{cor:rationalmain-cor} imply that the $\pi_0\hAut_{D^{2n}}(V_g,W_{g,1})$-action on the $p$-local higher homotopy groups of the spaces participating in \eqref{equ:fundamental-homotopy-fibration} factors in a range of degrees through the morphism (recall $K_g=\ker(\oH_n(W_{g,1})\ra \oH_n(V_g))$ from \cref{sec:homology-action})
 \[\pi_0\hAut_{D^{2n}}(V_g,W_{g,1})\lra \{\phi\in \GL(H_{W_{g,1}})\mid \phi(K_g)\subset K_g\}\] induced by the action on the homology of $W_{g,1}$. During the proof of \cref{thm:rationalmain} and the preceding \cref{rationallemma1}, we frequently use \cref{prop:plocalLS} to implicitly identify $p$-local homotopy groups of path components of pointed mapping spaces between bouquets of equidimensional spheres with derivations of free graded Lie algebras in a range of degrees. We denote by $\Maps_*^f(X,Y)$ for a map of based spaces $f\colon X\ra Y$ the corresponding component of the mapping space, pointed by $f$. Reminding the reader of our notation for spaces of derivations in \cref{sec:lie-algs}, we begin with the following lemma, whose first part was rationally obtained by Berglund and Madsen \cite[Prop.\,5.6]{BerglundMadsen}.

\begin{lem}\label{rationallemma1}Let $n\ge2$ and $p$ an odd prime. 
\begin{enumerate}
\item\label{item:rationallemma1-item1} In degrees $*<2p-3-n$, the morphism induced by relaxing the boundary condition
 \[\pi_*\Maps_\partial^{\id}(W_{g,1},W_{g,1})_{(p)}\lra\pi_*\Maps_*^{\id}(W_{g,1},W_{g,1})_{(p)}\cong \Der(L_{(p)}^{W})^+\]
is injective and has image $\Der_{\omega}(L_{(p)}^{W})^+\subset \Der(L_{(p)}^{W})^+$.
\item\label{item:rationallemma1-item2} In the range $*<2p-3-n$, the morphisms induced by the inclusion $W_{g,1}\subset W_g$
\[\pi_*\Maps_*^{\iota}(W_{g},V_g)_{(p)}\lra\pi_*\Maps_*^{\iota}(W_{g,1},V_g)_{(p)}\cong \Der^{\iota}(L_{(p)}^{W},L_{(p)}^{V})^+\]
is injective and has image $\Der^{\iota}_{\omega}(L_{(p)}^{W},L_{(p)}^{V})^+\subset \Der^{\iota}(L_{(p)}^{W},L_{(p)}^{V})^+$.
\end{enumerate}
\end{lem}

\begin{proof}
Restriction along the inclusion $\partial W_{g,1}\subset W_{g,1}$ yields a fibration \begin{equation}\label{equ:restriction-fibration-wg}\Maps_*(W_{g,1},W_{g,1})\lra \Maps_*(\partial W_{g,1},W_{g,1})\end{equation} whose fibre at $\iota$ is $\Maps_\partial(W_{g,1},W_{g,1})$. The induced maps on homotopy groups fits in the range $0<*<2p-2-n$ into a diagram of the form
\begin{equation}\label{equ:derivationdiagram}
\begin{tikzcd}[row sep=0.5cm]
\pi_*\Maps_*^{\id}(W_{g,1},W_{g,1})_{(p)}\arrow[r]\arrow[d,"\cong"]&\pi_*\Maps_*^{\iota}(\partial W_{g,1},W_{g,1})_{(p)}\arrow[d,"\cong"]\\
\Der(L_{(p)}^{W})^+\arrow[r]\arrow[d,"\cong"]&\Der^{\iota}(L_{(p)}^{\partial W_{g,1}},L_{(p)}^{W})^+\arrow[d,"\cong"]\\
\big(s^{-(2n-1)}L_{(p)}^{W}\otimes H_{(p)}^{W}\big)^+\arrow[r,"{\left[-,-\right]}"]&\big(s^{-(2n-2)}{[L_{(p)}^{W},L_{(p)}^{W}]}\big)^+.
\end{tikzcd}
\end{equation} whose top square is provided by \cref{prop:plocalLS}, so commutes. The bottom square is given as follows: the bottom right vertical map is the evaluation at the fundamental class \[[\partial W_{g,1}]\in s^{-1}H_{(p)}^{\partial W_{g,1}}\cong\bfQ[2n-2],\] which factors as a composition of isomorphisms 
\[\Der^{\iota}(L_{(p)}^{\partial W_{g,1}},L_{(p)}^{W})^+\xlra{\cong }\Hom(s^{-1}H_{(p)}^{\partial W_{g,1}},L_{(p)}^{W})^+\xlra{\cong}\big(s^{-(2n-2)}L_{(p)}^{W}\big)^+=\big(s^{-(2n-2)}{[L_{(p)}^{W},L_{(p)}^{W}]}\big)^+\]
where the first isomorphism restricts to generators, the second isomorphism evaluates at the fundamental class, and the final equality holds for degree reasons. The latter is because elements of degree $>0$ in $s^{-(2n-2)}L_{(p)}^{W}$ correspond to elements of degree $>2n-2$ in $L_{(p)}^{W}$, so are sums of brackets since this Lie algebra is generated in degree $n-1$ as $H_W$ is supported in degree $n$. The bottom left vertical map is the restriction to positive degrees of the map
\begin{equation}
\label{equ:derivation-identification-w-w}\Der(L_{(p)}^{W})\xlra{\cong }\Hom(s^{-1}H_{(p)}^{W},L_{(p)}^{W})\cong L_{(p)}^{W}\otimes (s^{-1}H_{(p)}^{W})^{\vee}\cong s^{-(2n-1)}L_{(p)}^{W}\otimes H_{(p)}^{W},\end{equation}
 where the first isomorphism is given by restricting to generators, the second is the canonical one, and the third is induced by the intersection form on $s^{-1}H_{W_{g,1}}$  (see \cref{ex:appendix}). By construction, the composition 
 \begin{equation}\label{equ:ev-omega}\Der(L_{(p)}^{W})^+\lra \big(s^{-(2n-2)}[L_{(p)}^{W},L_{(p)}^{W}]\big)^+\end{equation}
 coincides with the evaluation at the class $\omega$ that represents the inclusion $\partial W_{g,1}\subset W_{g,1}$, so it follows from \cref{lem:appcommutative} that this square commutes up to a sign, since $\omega=\sum_{i=1}^{g}[e_i,f_i]\in L_{W_{g,1}}$ agrees up to a sign with the element \eqref{equ:canonical-element-app} from the appendix as $e_i^{\#}= f_i$ and $f_i^{\#}= e_i$ holds in the notation of the appendix up to a fixed sign depending on $n$ (which does not play a role in the argument). As the bottom horizontal map is surjective as a consequence of the graded Jacobi identity, the middle horizontal arrow is surjective as well, and hence so is the top one. A consultation of the long exact sequence in homotopy groups induced by the fibration \eqref{equ:restriction-fibration-wg} thus proves \ref{item:rationallemma1-item1}, since the kernel of the middle horizontal arrow of the diagram is $\Der_{\omega}(L_{(p)}^{W})$ as \eqref{equ:ev-omega} is given by the evaluation at $\omega$. This finishes the proof of \ref{item:rationallemma1-item1}. The proof of \ref{item:rationallemma1-item2} is completely analogous, based on the fibration sequence obtained by applying $\Maps_*(-,V_g)$ to the cofibration sequence $\partial W_{g,1}\ra W_{g,1}\ra W_g$ instead of \eqref{equ:restriction-fibration-wg}. The bottom square of the diagram corresponding to \eqref{equ:derivationdiagram} is now given by
 \begin{equation}\label{equ:derivationdiagram2}
 \begin{gathered}[b]
\begin{tikzcd}[row sep=0.5cm,baseline={(current bounding box.center)}]
\Der^\iota(L_{(p)}^{W},L_{(p)}^{V})^+\arrow[r]\arrow[dr,"\ev_\omega"]\arrow[d,"\cong"]&\Der^{\iota}(L_{(p)}^{\partial W_{g,1}},L_{(p)}^{V})^+\arrow[d,"\cong"]\\
\big(s^{-(2n-1)}L_{(p)}^{V}\otimes H_{(p)}^{W}\big)^+\arrow[r,"{\left[-,\iota_*(-)\right]}"]&\big(s^{-(2n-2)}{[L_{(p)}^{V},L_{(p)}^{V}]}\big)^+
\end{tikzcd}
\end{gathered}
\end{equation}
whose vertical arrows are given by the composition
 \[\Der^{\iota}(L_{(p)}^{\partial W_{g,1}},L_{(p)}^{V})^+\xlra{\cong }\Hom(s^{-1}H_{(p)}^{\partial W_{g,1}},L_{(p)}^{V})^+\xlra{\cong}\big(s^{-(2n-2)}L_{(p)}^{V}\big)^+=\big(s^{-(2n-2)}{[L_{(p)}^{V},L_{(p)}^{V}]}\big)^+\]
 and the restriction to elements of positive degrees of the composition
 \begin{equation}
\label{equ:derivation-identification-w-v}\Der^\iota(L_{(p)}^{W},L_{(p)}^{V})\xra{\cong}\Hom(s^{-1}H_{(p)}^{W},L_{(p)}^{V})\cong L_{(p)}^{V}\otimes (s^{-1}H_{(p)}^{W})^{\vee}\cong s^{-(2n-1)}L_{(p)}^{V}\otimes H_{(p)}^{W},\end{equation}
both completely analogous to the two compositions explained below diagram \eqref{equ:derivationdiagram}.
\end{proof}

\begin{proof}[Proof of \cref{thm:rationalmain}]
We consider the map of horizontal fibration sequences
\begin{equation}\label{equ:fibrationcomparison}
\begin{tikzcd}
\Maps_\partial(V_g,V_g)\rar{\subset}\arrow[d,equal]&\Maps_{D^{2n}}((V_g,W_{g,1}),(V_g,W_{g,1}))\rar\dar{\subset}&\Maps_\partial(W_{g,1},W_{g,1})\dar\\
\Maps_\partial(V_g,V_g)\rar{\subset}&\Maps_*(V_g,V_g)\rar{(-)\circ \iota }&\Maps_*(W_g,V_g),
\end{tikzcd}
\end{equation} where the top right horizontal arrow is induced by restriction and the rightmost vertical arrow is given by extending a selfmap of $W_{g,1}$ relative to the boundary over the complement of $W_{g,1}\subset W_{g}$ by the identity, followed by postcomposition with the inclusion $W_{g}\subset V_g$. The induced morphism $\pi_*\Maps^{\id}_*(V_g,V_g)\ra \pi_*\Maps^{\iota}_*(W_g,V_g)$ is injective because its composition with the morphism $\pi_*\Maps^{\iota}_*(W_g,V_g)\ra \pi_*\Maps^{\iota}_*(W_{g,1},V_g)$ induced by restriction along the inclusion $W_{g,1}\subset W_g$ is a retract since the composition \[\vee^g(S^n\vee S^n)\simeq W_{g,1}\subset W_g\subset V_g\simeq \vee^gS^n\] is a homotopy retraction. From the long exact sequence in homotopy groups of the bottom fibration, we see that the monoid $\pi_0\Maps_{\partial}(V_g,V_g)$ receives a surjection from $\pi_1\Maps_*^{\iota}(W_g,V_g)$, and this is a monoid homomorphism as it agrees with the map induced on $\pi_0(-)$ by the homotopy fibre inclusion of the fibre sequence of $A_\infty$-spaces
\vspace{-0.1cm}
\[\big(\Omega\Maps_*^{\iota}(W_g,V_g) \simeq \hofib_{\id}(\inc)\big)\lra \Maps_{\partial}(V_g,V_g)\xlra{\inc}\Maps_{*}(V_g,V_g).\]
The group $\pi_1\Maps_*^{\iota}(W_g,V_g)$ is abelian since we have
\[\pi_1\Maps_*^{\iota}(W_g,V_g)\cong[S^1\wedge W_g,V_g]_*\cong[S^1\wedge (\vee^{2g} S^n\vee S^{2n}),V_g]_*\cong\pi_{n+1}(V_g)^{\oplus 2g}\oplus \pi_{2n+1}(V_g),\]
using $S^1\wedge W_g\simeq S^1\wedge (\vee^{2g} S^n\vee S^{2n})$ due to the fact that the attaching map \eqref{equ:attaching-map} of the top-dimensional cell in the usual CW-decomposition of $W_g$ is a sum of Whitehead-brackets and thus nullhomotopic after suspension. Being surjected upon by an abelian group, the monoid
$\pi_0\Maps_{\partial}(V_g,V_g)$ is itself an abelian group and hence agrees with $\pi_0\hAut_{\partial}(V_g)$, as claimed in \ref{label:rationamainI}. To prove \ref{label:rationamainII}, we combine the injectivity we just observed with \cref{rationallemma1} \ref{item:rationallemma1-item2} and the long exact sequence of the bottom row in \eqref{equ:fibrationcomparison} to obtain a short exact sequence
\begin{equation}\label{equ:homotopy-groups-vg-rel-full-bdy}0\lra\Der(L_{(p)}^{V})^+\xlra{(-)\circ \iota_*}\Der^{\iota}_{\omega}(L_{(p)}^{W},L_{(p)}^{V})^+\lra\pi_{*-1}\Maps_\partial^{\id}(V_g,V_g)_{(p)}\lra0\end{equation} in the range $*<2p-3-n$. Combining this with \cref{rationallemma1} \ref{item:rationallemma1-item1}, we see that the boundary map in the long exact sequence in homotopy groups of the upper fibration  of \eqref{equ:fibrationcomparison} fits in degrees $*<2p-3-n$ into a  commutative diagram
\begin{equation}\label{equ:derivationcomposition}
\begin{tikzcd}
\pi_*\Maps^{\id}_\partial(W_{g,1},W_{g,1})_{(p)}\arrow[d,"\cong"]\arrow[rr,"\partial"]&&\pi_{*-1}\Maps_\partial(V_g,V_g)_{(p)}\arrow[d,"\cong"]\\\Der_{\omega}(L_{(p)}^{W})^+\rar{\iota_*\circ(-)}& \Der^{\iota}_{\omega}(L_{(p)}^{W},L_{(p)}^{V})^+\rar{\pi}& \frac{\Der^{\iota}_{\omega}(L_{(p)}^{W},L_{(p)}^{V})^+}{\Der(L_{(p)}^{V})^+}
\end{tikzcd}
\end{equation}where $\pi$ is the quotient map and the right vertical map is induced by \eqref{equ:homotopy-groups-vg-rel-full-bdy}. To finish the proof of \ref{label:rationamainII}, we thus need to show that the bottom composition of \eqref{equ:derivationcomposition} fits as the left vertical arrow in a diagram as in \ref{label:rationamainII}. To see this, we first combine the bottom square of \eqref{equ:derivationdiagram} with the compatible square \eqref{equ:derivationdiagram2} to obtain a commutative diagram with exact rows
\begin{equation}
\begin{tikzcd}[column sep=0.6cm]\label{equ:derivation-comparison}
0\rar&[-0.4cm]\Der_{\omega}(L_{(p)}^{W})^+\arrow["(-)\circ\iota_*",d,swap]\rar&[-0.4cm] \big(s^{-(2n-1)}L_{(p)}^{W}\otimes H_{(p)}^{W}\big)^+\arrow[r,"{\left[-,-\right]}"]\dar{\iota_*\otimes\id }&\big(s^{-(2n-2)}{[L_{(p)}^{W},L_{(p)}^{W}]}\big)^+\dar{\iota_*}\rar&[-0.4cm]0\\
0\rar&\Der^{\iota}_{\omega}(L_{(p)}^{W},L_{(p)}^{V})^+\rar&\big(s^{-(2n-1)}L_{(p)}^{V}\otimes H_{(p)}^{W}\big)^+\arrow[r,"{\left[-,\iota_*(-)\right]}"]&\big(s^{-(2n-2)}{[L_{(p)}^{V},L_{(p)}^{V}]}\big)^+\rar&0.
\end{tikzcd}
\end{equation}
Next, writing \[K\coloneq \ker(\iota_*\colon H^W\ra H^V),\] we note that there is a chain of natural isomorphisms
\begin{equation}\label{equ:iso-k-intersection}\Der(L_{(p)}^{V})\cong\Hom(H_{(p)}^{V},L_{(p)}^{V})\cong L_{(p)}^{V}\otimes (s^{-1}H_{(p)}^{V})^\vee\cong s^{-(2n-1)}L_{(p)}^{V}\otimes K_{(p)}\end{equation}
defined analogously to (and compatible with) \eqref{equ:derivation-identification-w-v}, using that the isomorphism $(H^{W})^\vee\cong H^{W}$ induced by the intersection form (neglecting grading shifts) sends $(H^{V})^\vee\subset (H^{W})^\vee$ to $K\subset H^{W}$. Except for the equivariance claim,  \ref{label:rationamainII} now follows by combining the chain of isomorphisms \eqref{equ:iso-k-intersection} with the diagrams \eqref{equ:derivationcomposition}--\eqref{equ:derivation-comparison} and the canonical chain of isomorphisms 
\[
\big(s^{-(2n-1)}L_{(p)}^{V}\otimes H_{(p)}^{W}\big)/\big( s^{-(2n-1)}L_{(p)}^{V}\otimes K_{(p)} \big)\cong s^{-(2n-1)}L_{(p)}^{V}\otimes \big(H_{(p)}^{W}/ K_{(p)} \big)\cong s^{-(2n-1)}L_{(p)}^{V}\otimes H_{(p)}^{V}.
\]
This uses that the inclusion $\hAut^{\ext}_\partial(W_{g,1})\subset \Maps_\partial(W_{g,1},W_{g,1})$ is $0$-coconnected and that we have $\hAut_\partial(V_g)=\Maps_\partial(V_g,V_g)$ by \ref{label:rationamainI}. To see the  equivariance, note that all vertical maps in the diagram of \ref{label:rationamainII} are equivariant by construction. Since they are also surjective (see the discussion after the statement), it suffices to show that the top row is equivariant. This is clear for the second map in the top row, so we are left with showing equivariance of the first map
\begin{equation}\label{equ:equivriant-map-left}
\pi_{*+1}\BhAut_\partial(W_{g,1})_{(p)}\lra s^{2n-1}H_{(p)}^{W}\otimes L_{(p)}^{W}.\end{equation}
With respect to the canonical isomorphisms in positive degrees \[\pi_{*+1}\BhAut_\partial(W_{g,1})_{(p)}\cong\pi_{*}\hAut_\partial(W_{g,1})_{(p)}\cong \pi_*\Maps_\partial(W_{g,1},W_{g,1})_{(p)},\] the action on the domain of \eqref{equ:equivriant-map-left} is induced by conjugation. Going through the proof, we see that \eqref{equ:equivriant-map-left} arises as a composition of the form
\[\pi_{*}\Maps^{\id}_\partial(W_{g,1},W_{g,1})_{(p)}\ra\pi_{*}\Maps^{\id}_*(W_{g,1},W_{g,1})_{(p)}\ra \Der(L_{(p)}^{W})\cong s^{2n-1}H_{(p)}^{W}\otimes L_{(p)}^{W}.\] The first map relaxes the boundary condition, which is equivariant. The second map is given by the isomorphism in \cref{prop:plocalLS}, and its equivariance follows from the naturality part of that proposition. The third map is provided by the chain of isomorphisms \eqref{equ:derivation-identification-w-w}, which is equivariant by the naturality of the intersection form. This finishes the proof of \ref{label:rationamainII}. To see \ref{label:rationamainIII}, note that all restrictions on the degree in the proof of  \ref{label:rationamainII} originated from the assumption on the degree in \cref{prop:plocalLS}. This proposition holds rationally without that assumption, so the proof of \ref{label:rationamainII} also applies to \ref{label:rationamainIII}.
 \end{proof}
 
In a range of degrees, the particular shape of the $p$-local homotopy groups of the space $\BhAut^{\id}_{D^{2n}}(V_g,W_{g,1})$ ensured by \cref{cor:rationalmain-cor} allows us to pass from homotopy to homology groups, which is what we are actually interested in. In the following statement, we consider the module $\oH_n(W_{g,1};\bfZ_{(p)})$ as \emph{ungraded}.

\begin{cor}\label{cor:homologyhAut}For $n\ge2$, there is an injective map of graded $\pi_0\hAut_{D^{2n}}(V_g,W_{g,1})$-modules
\[\widetilde{\oH}_{*}(\BhAut^{\id}_{D^{2n}}(V_g,W_{g,1}); \bfZ_{(p)})\longhookrightarrow\big(\oH_n(W_{g,1};\bfZ_{(p)})^{\otimes 3}\big)[n]\] in degrees $*< \min(2n-1,2p-3-n)$ for primes $p$.
\end{cor}

\begin{proof}We may assume $p> 3$, since otherwise the claim has no content. As a result of \cref{cor:rationalmain-cor}, there is an injective map of graded $\pi_0\hAut_{D^{2n}}(V_g,W_{g,1})$-modules
\[\pi_{*+1}\BhAut^{\id}_{D^{2n}}(V_g,W_{g,1})_{(p)}\longhookrightarrow s^{-(2n-1)}H^{W}_{(p)}\otimes L^{W}_{(p)}=s^{-(2n-1)}H^{W}_{(p)}\otimes \bfL(s^{-1}H^{W}_{(p)})\]
 in degrees $0<*< 2p-4-n$. Using that we have \[[s^{-1}H^{W}_{(p)},s^{-1}H^{W}_{(p)}]\subset (s^{-1}H^{W}_{(p)})^{\otimes 2}\] by antisymmetrisation and that $H^{W}_{(p)}=\widetilde{\oH}_*(W_{g,1};\bfZ_{(p)})$ is concentrated in degree $n$, we get
 \[s^{-(2n-1)}H^{W}_{(p)}\otimes \bfL(s^{-1}H^{W}_{(p)})\subset s^{-(2n-1)}H^{W}_{(p)}\otimes (s^{-1}H^{W}_{(p)})^{\otimes 2}=\big(\oH_n(W_{g,1};\bfZ_{(p)})^{\otimes 3}\big)[n-1],\]
 in degrees $*<2n-2$, so the claim holds for homotopy instead of homology groups. This leaves us with showing that the $p$-local Hurewicz homomorphism
 \[
 \pi_{*}\BhAut^{\id}_{D^{2n}}(V_g,W_{g,1})_{(p)}\lra \widetilde{\oH}_*(\BhAut^{\id}_{D^{2n}}(V_g,W_{g,1})_{(p)};\bfZ_{(p)})
 \]
 is an isomorphism in degree $*< m\coloneq \min(2n-1,2p-3-n)$. Since submodules of free $\bfZ_{(p)}$-modules are free, it follows from the first part of the proof that $n$-truncation induces a $p$-locally $m$-connected map of the form
\[\BhAut^{\id}_{D^{2n}}(V_g,W_{g,1})\lra K(A,n)\] 
where $A$ is a free $\bfZ_{(p)}$-module, so it suffices to show that $K(A,n)$ has trivial $\bfZ_{(p)}$-homology in the range $n<*<m$. Since $A$ is free, it is enough to show that $\oH_*(K(\bfZ_{(p)},n);\bfZ_{(p)})\cong\oH_*(K(\bfZ,n);\bfZ_{(p)})$ vanishes in this range, which is certainly true rationally, so we may instead prove that $\oH^*(K(\bfZ,n);\bfF_{p})$ vanishes for $n+1<*<m+1$. As the natural map \[\bfHF_p^*(\bfHZ)={\lim}_n\oH^{*+n}(K(\bfZ,n);\bfF_p)\lra \oH^{*+n}(K(\bfZ,n);\bfF_p)\] is an isomorphism in degrees $*<n$, this follows from showing that the spectrum cohomology $\bfHF_p^*(\bfHZ)$ vanishes in degrees $0<*<\min(n,2p-2-2n)$. But $\bfHF_p^*(\bfHZ)$ is a quotient of the mod $p$ Steenrod algebra $\bfHF_p^*(\bfHF_p)$ by an ideal containing the Bockstein, so $\bfHF_p^*(\bfHZ)$ vanishes in degrees $0<*<2p-2$ and we conclude the assertion. 
\end{proof} 

\section{The proof of \cref{thm:mainthm}}
\label{sect:proofmainthm}
This final section is devoted to the proof of the following refinement of \cref{thm:mainthm}.

\begin{thm}\label{thm:mainthmtechnical}For $n>3$, there is a nilpotent space $X$ and a zig-zag 
\[\BC(D^{2n})\xlra{\phi} X\xlla{\psi} \Omega_0^{\infty+1}\oK(\bfZ)\] such that $p$-locally, $\phi$ is $\min(2n-4,2p-4-n)$-connected and $\psi$ is $\min(2n,2p-4)$-connected.\end{thm}

\begin{rem}As remarked in the introduction, our proof is independent of Waldhausen's work on pseudoisotopy theory. If one is willing to invest Dwyer--Weiss--Williams' index theorem \cite{DWW} (which relies in parts of Waldhausen's work), then it takes little effort to also explicitly identify $X$ and $\psi$ in terms of well-known infinite loop spaces, see \cref{sec:reformulation}. 
\end{rem}

As a first step towards proving \cref{thm:mainthmtechnical}, we replace $\BC(D^{2n})$ by an equivalent space that is more convenient to compare to the various automorphism spaces of high-dimensional handlebodies $V_g=\natural ^gD^{n+1}\times S^n$ we studied in the previous sections.

\begin{lem}\label{lem:differentmodel}For $d\ge5$, there exists a homotopy equivalence \[\BC(D^{d})\simeq \BlockDiff_{D^{d}}(D^{d+1})/\Diff_{D^{d}}(D^{d+1})=\hofib(\BDiff_{D^{d}}(D^{d+1})\ra\BlockBDiff_{D^{d}}(D^{d+1}))\]
\end{lem}
\begin{proof}
A choice of identification $D^{d}\times [0,1]\cong D^{d+1}$ by smoothing corners induces an equivalence $\BC(D^{d})\simeq \BDiff_{D^{d}}(D^{d+1})$, so the claim is equivalent to showing that the space of block diffeomorphisms $\BlockDiff_{D^{d}}(D^{d+1})$ is contractible. This follows from Cerf's result that every concordance of a manifold of dimension at least $5$ is isotopic to the identity \cite{Cerf}, together with the chain of isomorphisms $\pi_k(\BlockDiff_{D^{d}}(D^{d+1});\id)\cong \pi_0\Diff_{D^{d+k}}(D^{d+k+1})\cong \pi_0\oC(D^{d+k})$ whose first isomorphism is most easily seen from the combinatorial description of the homotopy groups of the Kan complex $\BlockDiff_{D^{d}}(D^{d+1})_\bullet$ (see \cref{section:diffeomorphisms}).
\end{proof}

The alternative point of view on $\BC(D^{2n})$ as the homotopy fibre \[\BlockDiff_{D^{2n}}(D^{2n+1})/\Diff_{D^{2n}}(D^{2n+1})=\BlockDiff_{D^{2n}}(V_0)/\Diff_{D^{2n}}(V_0)\] is advantageous as it makes a stabilisation map of the form \begin{equation}\label{equ:stabilisationBlockdiffDiff}\BC(D^{2n})\simeq\BlockDiff_{D^{2n}}(V_0)/\Diff_{D^{2n}}(V_0)\lra \BlockDiff_{D^{2n}}(V_g)/\Diff_{D^{2n}}(V_g),\end{equation} apparent, which is induced by iterating the stabilisation maps for $\BDiff_{D^{2n}}(V_g)$ and its block analogue explained in \cref{sect:stabilisation}. It is a consequence of Morlet's lemma of disjunction that this map is highly connected:

\begin{lem}\label{lem:Morletlemma}The stabilisation map \eqref{equ:stabilisationBlockdiffDiff} is $(2n-4)$-connected.
\end{lem}
\begin{proof}
Taking vertical homotopy fibres in the diagram (see \cref{sect:autVg} for the notation)
\[
\begin{tikzcd}[column sep=-0.3cm,row sep=0.1cm]
&\BDiff_\partial(V_g)\arrow[rr]\arrow[dd]&&\BDiff_{D^{2n}}(V_g)\arrow[rr]\arrow[dd]&&\BDiff^{\ext}_{\partial}(W_{g,1})\arrow[dd]\\
\BDiff_\partial(D^{2n+1})\arrow[rr,crossing over]\arrow[ur]&&\BDiff_{D^{2n}}(D^{2n+1})\arrow[rr,crossing over]\arrow[ur]&&\BDiff^{\ext}_{\partial}(D^{2n})\arrow[ur]&\\
&\BlockBDiff_\partial(V_g)\arrow[rr]&&\BlockBDiff_{D^{2n}}(V_g)\arrow[rr]&&\BlockBDiff^{\ext}_{\partial}(W_{g,1})\\
\BlockBDiff_\partial(D^{2n+1})\arrow[rr]\arrow[ur]\arrow[from=uu, crossing over]&&\BlockBDiff_{D^{2n}}(D^{2n+1})\arrow[rr]\arrow[ur]\arrow[from=uu, crossing over]&&\BlockBDiff^{\ext}_{\partial}(D^{2n})\arrow[ur]\arrow[from=uu, crossing over]&
\end{tikzcd}
\] of fibre sequences whose diagonal arrows are given by the iterated stabilisation maps results in a map of fibre sequences of the form
\[
\begin{tikzcd}[column sep=-0.1cm,row sep=0.2cm]
&\frac{\BlockDiff_\partial(V_g)}{\Diff_\partial(V_g)}\arrow[rr]&&\frac{\BlockDiff_{D^{2n}}(V_g)}{\Diff_{D^{2n}}(V_g)}\arrow[rr]&&\frac{\BlockDiff^{\ext}_{\partial}(W_{g,1})}{\Diff^{\ext}_{\partial}(W_{g,1})}\arrow[r,phantom,"\simeq"]&\frac{\BlockDiff_{\partial}(W_{g,1})}{\Diff_{\partial}(W_{g,1})}\\
\frac{\BlockDiff_\partial(D^{2n+1})}{\Diff_\partial(D^{2n+1})}\arrow[rr,crossing over]\arrow[ur]&&\frac{\BlockDiff_{D^{2n}}(D^{2n+1})}{\Diff_{D^{2n}}(D^{2n+1})}\arrow[rr,crossing over]\arrow[ur]&&\frac{\BlockDiff^{\ext}_{\partial}(D^{2n})}{\Diff^{\ext}_{\partial}(D^{2n})}\arrow[ur]\arrow[r,phantom,"\simeq"]&\frac{\BlockDiff_{\partial}(D^{2n})}{\Diff_{\partial}(D^{2n})},&
\end{tikzcd}
\]
whose inner diagonal map is the map in question and whose rightmost equivalences follow from another application of Cerf's result mentioned in the previous proof. As the manifolds $V_g$ and $W_{g,1}$ are $(n-1)$-connected, the two outer diagonal maps are $(2n-4)$-connected by a form of Morlet's lemma of disjunction \cite[p.\,29, Cor.\,3.2]{BLR}, so the claim follows from the induced ladder of long exact sequences in homotopy groups.\end{proof}

Combining the previous two lemmas results in a $(2n-4)$-connected map
\begin{equation}\label{equ:Morletmap}\BC(D^{2n})\lra \BlockDiff_{D^{2n}}(V_\infty)/\Diff_{D^{2n}}(V_\infty)\coloneq \hocolim_g\BlockDiff_{D^{2n}}(V_g)/\Diff_{D^{2n}}(V_g)\end{equation}
to the homotopy colimit over the stabilisation maps, so in order to prove \cref{thm:mainthmtechnical}, it remains to establish a zig-zag with the claimed connectivity properties between this homotopy colimit and the zero component of the once looped algebraic $K$-theory space of the integers $\Omega_0^{\infty+1}\oK(\bfZ)$, which we model as the plus-construction\footnote{We take all plus-constructions with respect to the unique maximal perfect subgroup, i.e.\ the intersection of the terms in the transfinite derived series. It follows from maximality that this subgroup is automatically normal.}
\[\Omega_0^{\infty+1}\oK(\bfZ)\simeq \BGL_\infty(\bfZ)^+\]
of the homotopy colimit $\BGL_\infty(\bfZ)\coloneq \hocolim_g\BGL_g(\bfZ)$ over the stabilisation maps induced by the usual block inclusions $\GL_g(\bfZ)\subset\GL_{g+1}(\bfZ)$. The zig-zag we construct arises as part of a commutative zig-zag of horizontal homotopy fibre sequences of the form \begin{equation}\label{equ:bigdiagram}
\begin{tikzcd}[row sep=0.6cm]
\BlockDiff_{D^{2n}}(V_\infty)/\Diff_{D^{2n}}(V_\infty)\arrow[r]\arrow[d]&\BDiff^{\sfr}_{D^{2n}}(V_\infty;\ell_0)_\ell\arrow[d]\arrow[r]&\BlockBDiff^{\sfr}_{D^{2n}}(V_\infty;\ell_0)_\ell\arrow[d,"\circled{2}"]\\
X\arrow[r]&\BDiff^{\sfr}(V_\infty;\ell_0)_\ell^+\arrow[r]&\BGL_\infty(\bfZ)^+\\
\Omega_0\BGL_\infty(\bfZ)^+\arrow[r]\arrow[u,"\circled{3}",swap]&*\arrow[r]\arrow[u,"\circled{1}",swap]&\BGL_\infty(\bfZ)^+\langle 1\rangle\arrow[u],
\end{tikzcd}
\end{equation}  which we explain now.
Denoting the tangential structure encoding stable framings by $\sfr\colon \EO\ra \BO$, the upper right corner is defined as the homotopy colimit \[\BlockBDiff^{\sfr}_{D^{2n}}(V_\infty;\ell_0)_\ell\coloneq \hocolim_g\BlockBDiff^{\sfr}_{D^{2n}}(V_g;\ell_0)_\ell\] along the stabilisation maps explained in \cref{sect:stabilisation}, and the space $\BDiff^{\sfr}_{D^{2n}}(V_\infty;\ell_0)_\ell$ is defined analogously, using the unstable $(2n+1)$-dimensional tangential structure induced by $\sfr\colon \EO\ra \BO$, which we denote by the same symbol (see \cref{sec:stable-tangential-structures}). The upper right horizontal map is the homotopy colimit of the comparison map \[\BDiff^{\sfr}_{D^{2n}}(V_g;\ell_0)_\ell\lra \BlockBDiff^{\sfr}_{D^{2n}}(V_g;\ell_0)_\ell\] whose homotopy fibre at the base point is canonically equivalent to $\BlockDiff_{D^{2n}}(V_\infty)/\Diff_{D^{2n}}(V_\infty)$ in view of \cref{lem:forgetting-is-cartesian} and the fact that homotopy fibres commute with sequential homotopy colimits. Using a functorial model of the plus-construction (see e.g.\,\cite[VII.6.2]{BousfieldKan}), the upper right square of \eqref{equ:bigdiagram} is induced by the homotopy colimit of the composition
\begin{equation}\label{eq:homology-action-proof-main-thm}\BDiff^{\sfr}(V_g;\ell_0)_\ell\lra \BlockBDiff^{\sfr}(V_g;\ell_0)_\ell\lra\BGL(H_{V_g})=\BGL(\oH_n(V_g;\bfZ))\cong \BGL_g(\bfZ)\end{equation} along the stabilisation maps (see \cref{sect:stabilisation}), where the second map is induced by the action on the middle homology of $V_g$. The space $X$ is defined as the homotopy fibre of the map $\BDiff^{\sfr}(V_\infty;\ell_0)_\ell^+\ra\BGL_\infty(\bfZ)^+$, and it receives a map from the top left corner induced by the commutativity of the upper right square. The bottom row is induced by the inclusion of the basepoint in the universal cover $\BGL_\infty(\bfZ)^+\langle 1\rangle$ whose homotopy fibre agrees with the base point component of $\Omega \BGL_\infty(\bfZ)^+$. This explains the diagram \eqref{equ:bigdiagram}, aside from the map of fibre sequences from the bottom to the middle row, which is induced by the universal cover $\BGL_\infty(\bfZ)^+\langle 1\rangle\ra\BGL_\infty(\bfZ)^+$ and the basepoint inclusion of $\BDiff^{\sfr}(V_\infty;\ell_0)_\ell^+$. 

\medskip

In the two following subsections, we continue the preparations of the proof of \cref{thm:mainthmtechnical} by analysing the vertical maps \circled{$1$} and \circled{$2$}. 

\subsection{The stable homology of $\BDiff_{D^{2n}}(V_g)$ and the map \circled{$1$}}
\label{sect:BPSection}Botvinnik and Perlmutter \cite{BotvinnikPerlmutter} have computed the stable homology of $\BDiff^\theta_{D^{2n}}(V_g;\ell_0)_\ell$ in homotopy theoretical terms for all tangential structures $\theta\colon B\ra\BSO(2n+1)$ whose space $B$ is $n$-connected. For us, their main result \cite[Cor.\,6.8.1, Prop.\,6.14]{BotvinnikPerlmutter} is most conveniently expressed as an identification of the group completion of the disjoint union
\[\textstyle{\cM_{\theta}\coloneq \coprod_{g\ge0}\BDiff_{D^{2n}}^{\theta}(V_g;\ell_0)_\ell},\] which becomes a homotopy commutative topological monoid under boundary connected sum when choosing an appropriate point-set model (see \cite[Prop.\,6.11, Prop.\,6.14]{BotvinnikPerlmutter}). 
\begin{thm}[Botvinnik--Perlmutter]\label{thm:BP}For $n>3$, a tangential structure $\theta\colon B\ra\BSO(2n+1)$ for which $B$ is $n$-connected, there is a homotopy equivalence of the form \[\Omega\mathrm{B}\cM_{\theta}\simeq \Omega^\infty\Sigma^\infty_+B.\]
\end{thm}

The equivalence they produce factors as a composition
\begin{equation}\label{equ:BP-chain}\Omega\mathrm{B} \cM_{\theta}\xlra{\simeq}\Omega\mathrm{B}\cC^\partial_\theta\xlra{\simeq}\Omega^\infty\Sigma^\infty_+B\end{equation}
whose intermediate terms is the group-completion of a topological category $\cC_\theta^\partial$ of bordisms between $2n$-manifolds with $\theta$-structures, possibly with boundary. This category was studied by Genauer \cite{Genauer}, who established the final equivalence in \eqref{equ:BP-chain} as a parametrised form of the Pontryagin--Thom construction. Botvinnik and Perlmutter showed that $\cM_{\theta}$ can be seen as a submonoid of the endomorphism monoid $\cC_\theta^\partial(D^{2n},D^{2n})$ of the closed disc $D^{2n}$ with some $\theta$-structure and that the chain of inclusions $\cM_{\theta}\subset \cC_\theta^\partial(D^{2n},D^{2n})\subset \cC_\theta^\partial$ is an equivalence upon taking classifying spaces (see \cite[Thm 6.3, Prop.\ 6.14]{BotvinnikPerlmutter}).

\medskip

As $\cM_\theta$ is homotopy commutative, Randal-Williams' elucidation of the group completion theorem \cite[Cor.\,1.2]{RWperfection} moreover provides an equivalence of the form \begin{equation}\label{equ:BPequivalence}\BDiff_{D^{2n}}^{\theta}(V_\infty;\ell_0)_\ell^+\simeq \Omega_0\mathrm{B}\cM_\theta,\end{equation}
which leads to the following consequence of \cref{thm:BP} when specialised to the tangential tangential structure encoding stable framings.

\begin{cor}\label{cor:BPcorollary}The space $\BDiff_{D^{2n}}^{\sfr}(V_\infty;\ell_0)_\ell^+$ is nilpotent and $p$-locally $\min(2n,2p-4)$-connected as long as $n>3$.
\end{cor}
\begin{proof}Connected $H$-spaces are nilpotent, so \eqref{equ:BPequivalence} settles the nilpotency claim. Regarding the connectivity part of the statement, note that the unstable $(2n+1)$-dimensional tangential structure $(\sfr)_{2n+1}$ induced by the stable structure $\sfr\colon \EO\ra\BO$ (see  \cref{sec:stable-tangential-structures}) is equivalent to the inclusion $\oO/\oO(2n+1)\ra\BO(2n+1)$ of the homotopy fibre of the stabilisation $\BO(2n+1)\ra\BO$, so \cref{thm:BP} and the discussion preceding this corollary show that the space in question is equivalent to $\Omega^\infty_0\Sigma^\infty_+\oO/\oO(2n+1)$. As $\oO/\oO(2n+1)$ is $2n$-connected, the homotopy groups of $\Omega^\infty_0\Sigma^\infty_+\oO/\oO(2n+1)$ agree in positive degrees less than $2n+1$ with the stable homotopy groups of spheres which are free of $p$-torsion in degrees less than $2p-3$ by a result of Serre \cite[p.\,498, Prop.\,5]{Serre}, so the claim follows.
\end{proof}

Said differently \cref{cor:BPcorollary} shows that the base point inclusion $\circled{1}$ in \eqref{equ:bigdiagram} is $p$-locally $\min(2n,2p-4)$-connected for $n>3$.

\subsection{The homology action and the map \circled{$2$}}
As a consequence of \cref{prop:mcgframed} \ref{prop:extensions-stably-framed-mcg}, the map on fundamental group induced by \circled{$2$} is infinite abelian, so the map \circled{$2$} is as far from being highly connected as possible, even $p$-locally for any $p$. Nevertheless, it turns out that it does induce an isomorphism on $p$-local \emph{homology} groups in a range, which we shall prove by separately studying the effect on homology of the two maps in the factorisation 
\begin{equation}\label{equ:factorisation}\BlockBDiff^{\sfr}_{D^{2n}}(V_g;\ell_0)_{\ell}\lra \BG^{\ext}_{g,\ell}\lra \BGL(H_{V_g}),\end{equation} of the second map in \eqref{eq:homology-action-proof-main-thm}. Here $G_{g,\ell}^{\ext}\subset \GL(H_{W_{g,1}})$ is the subgroup considered in \cref{sect:stableframedmcg}.

\begin{lem}\label{lem:blockdiffvsautomorphismgroup}
For $n\ge3$ and any prime $p$, the induced map \[\oH_*\big(\BlockBDiff^{\sfr}_{D^{2n}}(V_g;\ell_0)_\ell;\bfZ_{(p)}\big)\lra \oH_*\big(\BG^{\ext}_{g,\ell};\bfZ_{(p)}\big)\] is an isomorphism for $*<\min(2n-1,2p-3-n)$ and a surjection in that degree.
\end{lem}

\begin{proof}The claim is vacuous for $p=2$, so we assume otherwise. Consider the factorisation of the map in question
\begin{equation}\label{equ:factorisation-haut-to-gg}
\BlockBDiff^{\sfr}_{D^{2n}}(V_g;\ell_0)_\ell\ra\BlockBhAut^{\cong}_{D^{2n}}(V_g,W_{g,1})_\ell\ra\mathrm{B}\pi_0\BlockhAut^{\cong}_{D^{2n}}(V_g,W_{g,1})_\ell\ra\BG^{\ext}_{g,\ell},
\end{equation}
where the first map is that of \cref{cor:stableframed} applied to the triad $(V_g;D^{2n},W_{g,1})$, the second map is induced by taking path components, and the third is given by acting on the homology of $W_{g,1}\subset \partial V_g$ (see \cref{sect:stableframedmcg}). It follows from the corollary just mentioned that the first map induces an isomorphism in homology with $\bfZ_{(p)}$-coefficients in degrees $*<2p-3-n$ and a surjection in that degree since $V_g$ is obtained from $D^{2n}$ by attaching $n$-handles and the triad $(V_g;D^{2n},W_{g,1})$ satisfies the $\pi$-$\pi$-condition as $W_{g,1}$ and $V_g$ are simply connected for $n\ge2$. To study the remaining maps, we may replace the space of block homotopy automorphisms by its equivalent non-block analogue $\BhAut^{\cong}_{D^{2n}}(V_g,W_{g,1})_\ell$ (see \cref{sect:hAutPrelim}), so the $E_2$-page of the $p$-local Serre spectral sequence of the second map has the form
\begin{equation}\label{equ:e2pageblockdiffss}E_{k,l}^2\cong \oH_k\Big(\pi_0\hAut^{\cong}_{D^{2n}}(V_g,W_{g,1})_\ell;\oH_l \big(\BhAut^{\id}_{D^{2n}}(V_g,W_{g,1});\bfZ_{(p)} \big) \Big).\end{equation} To compute the entries $E_{k,l}^2$, we employ the Serre spectral sequence of the extension
\begin{equation}\label{equ:kernel-p-torsion-free}0\lra L_g\lra \pi_0\hAut^{\cong}_{D^{2n}}(V_g,W_{g,1})_\ell\lra G^{\ext}_{g,\ell}\lra 0,\end{equation}with coefficients in $\oH_l(\BhAut^{\id}_{D^{2n}}(V_g,W_{g,1});\bfZ_{(p)})$, where $L_g$ is the kernel as indicated. The $E_2$-page of this spectral sequence has the form
\[E_{s,t}^2\cong \oH_s\Big(G^{\ext}_{g,\ell};\oH_t\big(L_g;\oH_l(\BhAut^{\id}_{D^{2n}}(V_g,W_{g,1});\bfZ_{(p)})\big)\Big).\]
The kernel $L_g$ of the extension \eqref{equ:kernel-p-torsion-free} is a subgroup of the kernel of the analogous extension without the $\ell$-subscripts, and as the latter is finite and $p$-torsion free for $n<2p-4$ by \cref{lem:finitekernel}, so is the former. The group $L_g$ thus has no nontrivial homology in positive degrees with coefficients in a $\bfZ_{(p)}[L_g]$-module if $n<2p-4$, so in this case the spectral sequence $E_{s,t}^2$ is concentrated in the bottom row and we conclude  that 
\begin{equation}\label{equ:newly-labeled-ss}E_{k,l}^2\cong \oH_k\Big(G^{\ext}_{g,\ell};\oH_l\big(\BhAut^{\id}_{D^{2n}}(V_g,W_{g,1});\bfZ_{(p)}\big)_{L_g}\Big)\quad\text{for }n<2p-4,\end{equation}
where $(-)_{L_g}$ stands for taking coinvariants. By \cref{lem:negativeidentity}, the group $G^{\ext}_{g,\ell}$ contains the negative of the identity, which is central and acts on the coefficients in \eqref{equ:newly-labeled-ss} by $-1$ as long as $0<l<m\coloneq \min(2n-1,2p-3-n)$ as a result of \cref{cor:homologyhAut}, since it acts by $-1$ on $\oH_n(W_{g,1};\bfZ)^{\otimes 3}$ as $(-1)^3=(-1)$. This allows us to apply the ``centre kills''-trick\footnote{Given an element $g\in G$ in a group and a $\bfZ[G]$-module $M$, the multiplication map $g(-)\colon M\ra M$ is equivariant with respect to the conjugation map $c_g\colon G\ra G$. By the same argument as for the trivial module, the pair $(c_g,g(-))$ induces the identity on $\oH_*(G;M)$, so if $g$ is central and thus $c_g=\id$, then $(\id,(1-g)(-))$ annihilates $\oH_*(G;M)$. In particular, if $g$ acts on $M$ by multiplication by $k\in \bfZ$, then  $\oH_*(G;M)$ is $(1-k)$-torsion.} to conclude that the $E_2$-page $E_{k,l}^2$ is $2$-torsion for $0<l<m$ (which implies $n<2p-4$) and is therefore trivial as $p$ was assumed to be odd. In this range, the spectral sequence \eqref{equ:e2pageblockdiffss}  is therefore concentrated at the bottom row, so the second map in \eqref{equ:factorisation-haut-to-gg} induces an isomorphism on $\bfZ_{(p)}$-homology in degrees $*<m$ and a surjection in that degree. This leaves us with arguing that the final map in \eqref{equ:factorisation-haut-to-gg} also has this property. But we already observed that the kernel in \eqref{equ:kernel-p-torsion-free} is $p$-locally acyclic for $n<2p-4$, so in this case the third map is in fact a $p$-local homology isomorphism. As the statement is vacuous unless $1<m$ and thus $n<2p-4$, this finishes the proof.
 \end{proof}

In contrast to the first map in \eqref{equ:factorisation}, the second map might not induce a $p$-local homology isomorphism in a range of degrees, but we will see below that its homotopy colimit $\BG^{\ext}_{\infty,\ell}\ra \BGL(H_{V_\infty})$ with respect to the stabilisation maps explained in \cref{sect:stabilisation} does.

\begin{lem}\label{lem:compareautomorphismgroups}For odd primes $p$, the induced map
\[\oH_*\big(\BG^{\ext}_{\infty,\ell};\bfZ_{(p)}\big)\lra\oH_*\big(\BGL(H_{V_\infty});\bfZ_{(p)}\big)\] is an isomorphism.
\end{lem}
\begin{proof}
By an application of the Serre spectral sequence of the extension
\[0\lra M_g^{\sfr}\lra G_{g,\ell}^{\ext}\lra \GL(H_{V_g})\lra 0\] established in \cref{prop:mcgframed} \ref{prop:extensions-stably-framed-mcg}, it suffices to show that the colimit induced by stabilisation
\[\colim_g\oH_*\big(\GL(H_{V_g});\widetilde{\oH}_*(M_g^{\sfr};\bfZ_{(p)})\big)\cong \colim_g\oH_*\big(\GL(H_{V_g});\widetilde{\oH}_*(M_g^{\sfr}\otimes \bfZ_{(p)};\bfZ_{(p)})\big)\] vanishes; the isomorphism can be viewed as being induced by the $p$-localisation of $K(M_g^{\sfr},1)$, see \cref{section:nilpotency}.  Recall from \cref{sect:stableframedmcg} that the $\GL(H_{V_g})$-module $M^{\sfr}_g\subset (H_{V_g}\otimes H_{V_g})^\vee$ consists of all even $(-1)^{n+1}$-symmetric bilinear forms $\mu\in(H_{V_g}\otimes H_{V_g})^\vee$, so as $p$ is assumed to be odd, its $p$-localisation $M_g^{\sfr}\otimes \bfZ_{(p)}$ is isomorphic to the dual of the symmetric square $\Sym^2(H_{V_g})^\vee\otimes\bfZ_{(p)}$ if $n$ is odd and the dual of the exterior square $\Lambda^2(H_{V_g})^\vee\otimes\bfZ_{(p)}$ if $n$ is even. In particular, by antisymmetrisation, the $\GL(H_{V_g})$-module $M_g^{\sfr}\otimes\bfZ_{(p)}$ is a direct summand of $(H_{V_g}\otimes H_{V_g})^{\vee}\otimes\bfZ_{(p)}$. Choosing a basis $H_{V_g}\cong\bfZ^g$ compatible with the stabilisation maps, this shows that it suffices to show that the stable $\GL_g(\bfZ)$-homology with coefficients in $\oH_k((\bfZ^{g}\otimes \bfZ^{g})^{\vee};\bfZ)\cong \Lambda^{k}(\bfZ^{g}\otimes \bfZ^{g})^{\vee}$ vanishes for $k>0$. Pulling back this module along the automorphism of $\GL_g(\bfZ)$ given by taking transpose inverse, we see that we may replace this module by its dual $\Lambda^{k}(\bfZ^{g}\otimes \bfZ^{g})$ whose $\GL_g(\bfZ)$-homology for large $g$ with respect to $k$ does indeed vanish by an application of a result due to Betley \cite[Thm\,3.1]{Betley}.
\end{proof}

\begin{cor}\label{cor:homologyequivalence}Let $n\ge3$ and $p$ a prime. The map on homology induced by the map \circled{$2$},
\[\oH_*(\BlockBDiff^{\sfr}_{D^{2n}}(V_\infty;\ell_0)_\ell;\bfZ_{(p)})\lra \oH_*(\BGL_\infty(\bfZ)^+;\bfZ_{(p)}),\]
is an isomorphism for $*<\min(2n-1,2p-3-n)$ and a surjection in that degree.
\end{cor}
\begin{proof}This is free of content if $p$ is even and follows for $p$ odd from a combination of Lemmas~\ref{lem:blockdiffvsautomorphismgroup} and \ref{lem:compareautomorphismgroups}, using that the canonical map $\BGL_\infty(\bfZ)\ra\BGL_\infty(\bfZ)^+$ is acyclic.
\end{proof}

\subsection{Proof of \cref{thm:mainthmtechnical}}
As in the previous proofs, we assume $p>2$;  the claim is vacuous otherwise. Our goal is to demonstrate that the precomposition of the left column in \eqref{equ:bigdiagram} with the $(2n-4)$-connected map \eqref{equ:Morletmap} provides a zig-zag as promised by \cref{thm:mainthmtechnical}. As a result of the discussion in \cref{sect:BPSection}, the space $X$ is the homotopy fibre of a map out of an infinite loop space which is moreover surjective on fundamental groups as a result of \cref{prop:mcgframed}, so it follows from \cref{lem:nilpotentfibre} that $X$ is nilpotent. That the map \circled{$3$} is $p$-locally $\min(2n,2p-4)$-connected is a consequence of \circled{$1$} having this property by \cref{cor:BPcorollary} and the $1$-connected cover $\BGL_\infty(\bfZ)^+\langle 1 \rangle\ra\BGL_\infty(\bfZ)^+$ being a $p$-local equivalence as $\pi_1\BGL_\infty(\bfZ)^+\cong\bfZ/2$ and $p$ is odd. This leaves us with showing that
\[\BC(D^{2n})\lra \BlockDiff_{D^{2n}}(V_\infty)/\BlockDiff_{D^{2n}}(V_\infty)\lra X\] is $p$-locally $\min(2n-4,2p-4-n)$-connected, which we can test on $p$-local homology groups by \cref{lem:simplicialaction} as source and target are nilpotent. The first map in this composition is $(2n-4)$-connected by \cref{lem:Morletlemma}, so we may focus on the second and show that it induces an isomorphism in $p$-local homology in the required range. Since plus-constructions do not affect homology groups and \circled{$2$} is an isomorphism in homology with $\bfZ_{(p)}$-coefficients in degrees less than $\min(2n-1,2p-3-n)$ and an epimorphism in that degree by \cref{cor:homologyequivalence}, the claim follows from an application of Zeeman's comparison theorem (see e.g.\,\cite[Thm.\,3.2]{HiltonRoitberg}) to the map of Serre spectral sequences induced by the first two rows of \eqref{equ:bigdiagram}, provided we ensure that the actions of the fundamental groups of the bases of both fibre sequences on the $\bfZ_{(p)}$-homology of the respective fibres are trivial in this range. For the second row, this follows from the $p$-local high connectivity of the map \circled{$3$} established in the first part of the proof, using that the canonical action of $\pi_1\BGL_\infty(\bfZ)^+$ on the homology of $\Omega \BGL_\infty(\bfZ)^+$ is trivial as $\BGL_\infty(\bfZ)^+\simeq \Omega_0 \mathrm{B}(\coprod_g\BGL_g(\bfZ))$ is a loop space (in fact, an infinite one). For the  first row, the triviality of the action is a consequence of \cref{lem:Morletlemma} together with the observation that any element in $\pi_0\Diff_{D^{2n}}(V_g)$ can be represented by a diffeomorphism that fixes $D^{2n+1}=V_0\subset V_g$ pointwise, so commutes with any diffeomorphism in the image of the iterated stabilisation map $\oC(D^{2n})\simeq \Diff_{D^{2n}}(D^{2n+1})\ra \Diff_{D^{2n}}(V_g)$. This finishes the proof of \cref{thm:mainthmtechnical}, which in particular implies \cref{thm:mainthm}.

\subsection{A reformulation}\label{sec:reformulation}
Although not necessary for the proof of \cref{thm:mainthmtechnical} or \cref{thm:mainthm},  we shall explain in \cref{sec:DWW} below how an instance of the Dwyer--Weiss--Williams index theorem \cite{DWW} shows that the plus constructed stable homology action
\begin{equation}\label{equ:homology-action}\BDiff^{\sfr}_{D^{2n}}(V_\infty)_\ell^+\lra \BGL_{\infty}(\bfZ)^+\end{equation} featuring in \eqref{equ:bigdiagram} agrees with respect to the equivalence $\BDiff^{\sfr}_{D^{2n}}(V_\infty)_\ell^+\simeq \Omega^\infty_0\Sigma_+^\infty\oO/\oO(2n+1)$ explained in \cref{sect:BPSection} with the map obtained by applying $\Omega^\infty_0(-)$ to the composition
\begin{equation}\label{equ:composition-unit-mult}
\Sigma_+^\infty\oO/\oO(2n+1)\xlra{\pr}\bfS{\xlra{\iota}}K(\bfZ)\xlra{(-1)^n}K(\bfZ),
\end{equation}
where $\pr$ is the projection and $\iota$ the unit. This identifies the homotopy fibre $X$ in \eqref{equ:bigdiagram} as \[X\simeq \Omega^\infty_+\hofib\big(\Sigma_+^\infty\oO/\oO(2n+1)\xra{\pr}\bfS\xra{\iota} K(\bfZ)\big),\] so we may postcompose the map $\phi$ from \cref{thm:mainthmtechnical} with the $(2n+1)$-connected map $\hofib(\iota\circ\pr)\ra\hofib(\iota)$ induced by the projection to arrive at a cleaner formulation of \cref{thm:mainthm}: there is a $p$-locally $\min(2n-4,2p-4-n)$-connected map 
\begin{equation}\label{equation:better-map}
\BC(D^{2n})\lra \Omega_0^\infty\hofib(\bfS\xra{\iota} K(\bfZ)).\end{equation}

\subsubsection{Relation to Waldhausen's work}It is reasonable to expect that the map \eqref{equation:better-map} agrees up to equivalences with the composition
\begin{equation}\label{equ:waldhausen-composition}
\BC(D^{2n})\lra \Omega_0^\infty\Wh^{\Diff}(*)=\Omega_0^\infty\hofib\big(\bfS\ra K(\bfS)\big)\lra \Omega_0^\infty\hofib\big(\bfS\ra K(\bfZ)\big)
\end{equation}
of the map known from Waldhausen's work \cite{WaldhausenManifold} with the map induced by linearisation. While it is not hard to show that the map \eqref{equation:better-map} does indeed factor over the second map in \eqref{equ:waldhausen-composition}, a convincing comparison between \eqref{equation:better-map} and \eqref{equ:waldhausen-composition} appears to be more laborious and we will not go into this matter at this point (however, see the final remark of the introduction).

\subsubsection{An application of the index theorem}\label{sec:DWW}
To justify the claimed identification of \eqref{equ:homology-action} with the map resulting from applying $\Omega^\infty_0(-)$ of \eqref{equ:composition-unit-mult}, one can argue as follows: firstly, it suffices to show that these maps agree when precomposed with an arbitrary map $B\ra \BDiff^{\sfr}_{D^{2n}}(V_\infty)_\ell^+$ that classifies a smooth $V_g$-bundle $\pi\colon E\ra B$ for some $g\ge0$ together with a trivial $D^{2n}$-subbundle and stably framed vertical tangent bundle $T_\pi E$. Tracing through the equivalences featuring in \cref{sect:BPSection}, one finds that the composition
\begin{equation}\label{equ:bundle-composition}
B\lra \BDiff^{\sfr}_{D^{2n}}(V_\infty)_\ell^+\simeq \Omega^\infty_0\Sigma^\infty_+(\oO/\oO(2n+1))
\end{equation}
 represents the class \[[T_\pi E]\circ  \BG(\pi)-\chi(V_g)\in [\Sigma_+^\infty B, \Sigma^\infty_+(\oO/\oO(2n+1)]\] where  $\chi(-)$ is the Euler characteristic, $\BG(\pi)\in[\Sigma_+^\infty B, \Sigma^\infty_+E]$ is the Becker--Gottlieb transfer of the bundle $\pi$ described in terms of the Pontryagin--Thom construction \cite{BeckerGottlieb1,BeckerGottlieb2}, and the class $[T_\pi E]\in [\Sigma_+^\infty E, \Sigma^\infty_+\oO/\oO(2n+1)]$ is induced by the vertical tangent bundle of $\pi$ together with its stable framing. It thus suffices to show that the class
\begin{equation}\label{equ:BG-transfer}
(-1)^n(\iota \circ \pr)^* (\BG(\pi)-\chi(V_g))\in [\Sigma_+^\infty B, K(\bfZ)]\end{equation}
agrees with the class represented by the composition
 \[B\lra \BDiff^{\sfr}_{D^{2n}}(V_\infty)_\ell^+\lra \BGL_{\infty}(\bfZ)^+\simeq \Omega^{\infty}_0K(\bfZ).\] The latter can be identified with the class
 \begin{equation}\label{equ:homology-transfer}[H_n(\pi)]-g\in [\Sigma^\infty_+B,K(\bfZ)]\end{equation} where $H_n(\pi)$ is the local system over $B$ of fibrewise middle-dimensional homology groups of $\pi$. Dwyer--Weiss--Williams' improved Riemann--Roch theorem \cite[p.\ 2]{DWW} gives
\[(\iota\circ \pr)^* \BG(\pi)=1+(-1)^n\cdot [H_n(\pi)]\in [\Sigma^\infty_+B,K(\bfZ)],\]
so using  $\chi(V_g)=1+(-1)^ng$ it follows that \eqref{equ:BG-transfer} and \eqref{equ:homology-transfer} indeed agree.

\appendix

\section{Stable tangential bundle maps are stable bundle maps}
\label{sect:appendix}
Fix a $d$-dimensional vector bundle $\xi$ over a space $X$, a stable vector bundle $\{\psi_k\ra B_k\}_{k\ge 0}$, subcomplexes $A,C\subset X$, and a bundle map $\ell_0\colon \xi|_A\oplus\varepsilon^k\ra \psi_{d+k}$ covering a map $\bar{\ell}_0\colon X\ra B_{d+k}$ for some $k\ge0$. In addition to the notation for various types of bundle maps in \cref{section:prelim}, we abbreviate 
\[\BlockMaps_A(X,B;\bar{\ell}_0)_\bullet\coloneq\colim_{m\ge k}\BlockMaps_A(X,B_m;\bar{\ell}_0)_\bullet,\] where $\BlockMaps_A(X,B_m;\bar{\ell}_0)_\bullet$ is the semi-simplicial set whose $p$-simplices are block maps $\Delta^p\times X\ra \Delta^p\times B_{m}$ which agree with $\id_{\Delta^p}\times\bar{\ell}_0$ on $\Delta^p\times A$. The colimit is taken over the maps induced by post-composition with the maps $B_m\ra B_{m+1}$ underlying the structure maps of $\psi$. The sub semi-simplicial set of maps $\Delta^p\times X\ra \Delta^p\times B_m$ over $\Delta^p$ is denoted by $\Maps_A(X,B;\bar{\ell}_0)_\bullet\subset \BlockMaps_A(X,B;\bar{\ell}_0)_\bullet$.

\begin{lem}\label{lem:kanfibration}If the base $X$ of $\xi$ is a finite CW-complex, then the semi-simplicial maps 
\[\BlockBun_A(\xi^s,\psi;\ell_0)^\tau_\bullet\lra \BlockMaps_A(X,B;\bar{\ell}_0)_\bullet\quad\text{and}\quad\BlockhAut_A(\xi^s;C)^\tau_\bullet\lra \BlockhAut_A(X;C)_\bullet\] induced by forgetting bundle maps are Kan fibrations.
\end{lem}
\begin{proof}
For $\psi=\xi^s$ and $\ell_0=\inc$ the inclusion, the second map agrees with the pullback of the first map along the inclusion $\BlockhAut_A(X;C)_\bullet\subset  \BlockMaps_A(X,X;\inc)_\bullet$, so it suffices to show that the first map is a Kan fibration. We shall do so in the case where $A$ is empty; the argument for general $A$ is similar. The upper horizontal map in a lifting problem
\[
\begin{tikzcd}
(\Lambda_j^p)_\bullet\dar\rar&\BlockBun(\xi^s,\psi)^\tau_\bullet\dar\\
(\Delta^p)_\bullet\arrow[ur,dashed]\rar&\BlockMaps(X,B)_\bullet,
\end{tikzcd}
\]
induces bundle maps \[\phi_i\colon\tau_{\Delta^p_i}\times \xi\oplus\varepsilon^k\lra \tau_{\Delta^p_i}\times \psi_{d+k}\] for $i\neq j$ that agree on their faces and cover the restriction of the map $\bar{\phi}\colon\Delta^p\times X\ra \Delta^p\times B_{d+k}$ induced by the bottom horizontal arrow to $\Delta^p_i\times X$. By replacing $\xi$ with its stabilisation $\xi\oplus\varepsilon^k$, we may assume $k=0$. There is an extension of $\phi_i$ to $\tau_{\Delta^p}|_{\Delta^p_i}$ given by the composition
\[\tau_{\Delta^p}|_{\Delta^p_i} \times \xi\xra{\cong}\varepsilon\oplus \tau_{\Delta^p_i} \times \xi\xra{\id_\varepsilon\oplus\phi_i } \varepsilon\oplus\tau_{\Delta^p_i}\times \psi_{d}\xra{\cong}\tau_{\Delta^p}|_{\Delta^p_i} \times \psi_{d}\] whose outer isomorphisms are induced by the differential of the diffeomorphism $c_{i,\epsilon}$ of \cref{sect:blockspaces}. The condition \eqref{equ:collarsquare} in the definition of tangential bundle maps is made precisely such that these extensions agree on their intersections, so they assemble to a bundle map 
\[\phi_{\Lambda^p_j}\colon \tau_{\Delta^p}|_{\Lambda^p_j}\times\xi\lra \tau_{\Delta^p}|_{\Lambda^p_j}\times\psi_{d}
\]
that covers the restriction of $\bar{\phi}$ to $\Lambda^p_j\times X$ and we are left to argue that this bundle map is the restriction of a $p$-simplex in $\Bun(\xi^s,\psi)^\tau_\bullet$ covering $\bar{\phi}$. As the inclusion $\Lambda^p_j\times X\subset \Delta^p\times X$ is a trivial cofibration, obstruction theory provides an extension of $\phi_{\Lambda^p_j}$ to a bundle map 
\[\phi_{\Delta^p}\colon \tau_{\Delta^p}\times\xi\lra \tau_{\Delta^p}\times\psi_{d}\] that covers $\bar{\phi}$, but this extension might violate condition \eqref{equ:collarsquare} on the $j$th face, i.e.\,the map \[\phi_j\colon\varepsilon\oplus \tau_{\Delta^p_j}\times\xi\lra  \varepsilon\oplus \tau_{\Delta^p_j}\times\psi_{d}\] obtained by restricting $\phi_{\Delta^p}$ to $\Delta^p_j\times X$ and using the isomorphism $\tau_{\Delta^p}|_{\Delta^p_j}\cong \varepsilon\oplus \tau_{\Delta^p_j}$ induced by the derivative of $c_{j,\epsilon}$ need not have the form $\id_\varepsilon\oplus \phi_j$ for a bundle map \[\phi_j\colon \tau_{\Delta^p_j}\times \xi\lra \tau_{\Delta^p_j}\times \psi_{d}.\] Nevertheless, its restriction to $\partial\Delta^p_j\times X$ does have this form and we will argue that after adding a trivial bundle of sufficiently large dimension, say $n$, we can alter $\phi_j$ by a homotopy of bundle maps relative to $\partial\Delta^p_i\times X$ that covers $\bar{\phi}|_{\Delta^p_i}$ so that it does have the required form. This would show the claim because we can use this homotopy to change $\phi_{\Delta^p}\oplus\id_{\varepsilon^n}$ to a $p$-simplex in $\Bun(\xi \oplus\varepsilon^n,\psi_{d}\oplus\varepsilon^n)^\tau_\bullet$ that covers $\bar{\phi}$ and extends $\phi_{\Lambda^p_j}\oplus\id_{\varepsilon^n}$ and thus provides a lift as required. To this end, we denote by $\Iso(\nu,\eta)\ra Y$ for vector bundles $\nu$ and $\eta$ over a space $Y$ the fibre bundle whose sections correspond to bundle morphisms $\nu\ra\eta$ over the identity, so the fibre over $y\in Y$ is the space of isomorphisms $\nu_y\ra \eta_y$ between the fibres of $\nu$ and $\eta$. Abbreviating $\nu\coloneq \tau_{\Delta_j^p}\times \xi$ and $\eta\coloneq \bar{\phi}^*(\tau_{\Delta_j^p}\times \psi_d)$, we consider the diagram
\[
\begin{tikzcd}[row sep=0.7cm]
\partial\Delta^p_j\times X\arrow[d,hook]\rar&\Iso(\nu\oplus \varepsilon^n, \eta\oplus \varepsilon^n )\arrow[r,"\id_{\varepsilon}\oplus-"]&\Iso(\varepsilon\oplus \nu\oplus \varepsilon^n, \varepsilon\oplus \eta\oplus \varepsilon^n)\arrow[d]\\
\Delta^p_j\times X\arrow[r,equal]\arrow[ur,dashed]\arrow[rru]&\Delta_j^p\times X\arrow[r,equal]\arrow[u,leftarrow,crossing over]&\Delta^p_j\times X
\end{tikzcd}
\]
in which the solid diagonal arrow is induced by $\phi_j\oplus\id_{\varepsilon^n}$ and the upper left horizontal one by its restriction to $\partial \Delta^p_j\times X$, which makes the subdiagram of solid arrows commute strictly. From this point of view, the task we set us is equivalent to constructing a dashed arrow for some $n$ that makes the upper leftmost triangle commutes strictly and the one formed by the two diagonal arrows up to homotopy relative $\partial \Delta^p_j\times X$. Using that $X$ is a finite CW complex and that the connectivity of the map on vertical homotopy fibres of the right square increases in $n$ as it agrees with the inclusion $\oO(p+d+n)\subset \oO(p+d+n+1)$ up to equivalence, this follows from an application of obstruction theory.
\end{proof}

\begin{nrem}We learnt the ``stabilisation trick'' of the previous proof from Appendix D of \cite{BerglundMadsen}, which contains results similar to those of this appendix.
\end{nrem}

\begin{cor}\label{cor:tangentialiskan}If the base $X$ of $\xi$ is a finite CW-complex, then the following semi-simplicial sets satisfy the Kan property.
{\setlength\multicolsep{2pt}\begin{multicols}{2}
\begin{enumerate}
\item $\BlockBun_A(\xi^s,\psi;\ell_0)^\tau_\bullet$ 
\item $\Bun_A(\xi^s,\psi;\ell_0)^\tau_\bullet$ 
\item $\BlockhAut_A(\xi^s;C)_\bullet^\tau$ 
\item $\hAut_A(\xi^s;C)_\bullet^\tau$
\end{enumerate}
\end{multicols}}
\end{cor}
\begin{proof}
It is straight-forward to see that $\BlockMaps_A(X,B;\bar{\ell}_0)_\bullet$ is Kan, so the first part follows from \cref{lem:kanfibration} and the fact that the domain of a Kan fibration over a Kan complex is Kan. The same reasoning applies to the second semi-simplicial set, using that $\Bun_A(\xi^s,\psi;\ell_0)^\tau_\bullet\ra \Maps_A(X,B;\bar{\ell}_0)_\bullet$ is a Kan fibration because it is the pullback of the first map of \cref{lem:kanfibration} along the inclusion $\Maps_A(X,B;\bar{\ell}_0)_\bullet\subset\BlockMaps_A(X,B;\bar{\ell}_0)_\bullet$. Also the semi-simplicial sets $\hAut_A(X;C)_\bullet$ and $\BlockhAut_A(X;C)_\bullet$ are easily seen to be Kan (see \cref{sect:hAutPrelim}), so the remaining claims can be proved in the same way.
\end{proof}

\begin{lem}\label{lem:blockisnonblock}If the base $X$ of $\xi$ is a finite CW-complex, then the inclusions
\[\Bun_A(\xi^s,\psi;\ell_0)_\bullet^\tau\subset \BlockBun_A(\xi^s,\psi;\ell_0)_\bullet^\tau\quad\text{and}\quad \hAut_A(\xi^s;C)_\bullet^\tau\subset \BlockhAut_A(\xi^s;C)_\bullet^\tau\] are equivalences.
\end{lem}
\begin{proof}
These inclusions are pullbacks of the inclusions \begin{equation}\label{equ:blockinclusionsapp}\Maps_A(X,B;\bar{\ell}_0)_\bullet\subset\BlockMaps_A(X,B;\bar{\ell}_0)_\bullet\quad\text{ and }\quad\hAut_A(X;C)_\bullet\subset\BlockhAut_A(X;C)_\bullet\end{equation} along the two Kan fibrations discussed in \cref{lem:kanfibration}, so the claim follows from showing that the inclusions \eqref{equ:blockinclusionsapp} are equivalences. As already mentioned in the previous proof, it is straight-forward to show that these semi-simplicial sets are Kan. Using the combinatorial description of their homotopy groups, the claim follows from the contractibility of $\Maps_{\partial \Delta^p}(\Delta^p,\Delta^p)$ (cf.\,\cref{sect:hAutPrelim}).
\end{proof}

\begin{lem}\label{lem:tangentialisequ}If the base $X$ of $\xi$ is a finite CW-complex, then the extension maps of \eqref{equ:euivtangential}
\[\Bun_A(\xi^s,\psi;\ell_0)_\bullet\ra \Bun_A(\xi^s,\psi;\ell_0)_\bullet^\tau\quad\text{and}\quad \hAut_A(\xi^s;C)_\bullet\ra \hAut_A(\xi^s;C)_\bullet^\tau\] are equivalences.
\end{lem}
\begin{proof}
The proof for the two maps are essentially identical; we restrict our attention to the first map. Its source and target have the Kan property (the source because it is the singular complex of a space, the target by \cref{cor:tangentialiskan}), so we may test the claim on semi-simplicial homotopy groups. The two semi-simplicial sets involved have the same $0$-simplices, so the map is clearly surjective on path components. To show that it is surjective on homotopy groups in positive degrees, we fix a semi-simplicial base point in $\Bun_A(\xi^s,\psi;\ell_0)_\bullet$ by choosing a $0$-simplex $\ell\colon\xi\oplus\varepsilon^l\ra\psi_{d+l}$ and taking products with the tangent bundles $\tau_{\Delta^p}$. Let \[\phi\colon\xi\oplus\varepsilon^k\times \tau_{\Delta^p} \lra\psi_{d+k}\times \tau_{\Delta^p}\] be a $p$-simplex that represents a class in $\pi_p(\Bun_A(\xi^s,\psi;\ell_0)^\tau_\bullet;\ell)$. Up to changing $\phi$ within its homotopy class, we may assume that the underlying map $\bar{\phi}\colon X\times\Delta^p\ra B_{d+k}\times \Delta^p$ satisfies the collaring condition from \cref{sect:blockspaces} for some $0<\epsilon<1/2$. By replacing $\xi$ with $\xi\oplus\epsilon^k$, we may assume $k=0$. Fixing a trivialisation $F\colon \tau_{\Delta^p}\cong \bfR^p\times\Delta^p$, our candidate for a preimage in $\pi_p(\Bun_A(\xi^s,\psi;\ell_0)_\bullet;\ell)$ is the class defined by the composition
\begin{equation}\label{equ:composition-appendix}\xi\oplus\varepsilon^p\times\Delta^p\xlra{\id_{\xi}\times F^{-1}}\xi\times\tau_{\Delta^p}\xlra{\phi}\psi_{d}\times \tau_{\Delta^p}\xlra{\id_{\psi_{d}}\times F} \psi_d\oplus\varepsilon^p\times\Delta^p\lra\psi_{d+p}\times\Delta^p,\end{equation} where the last arrow is induced by the structure map of $\psi$. To justify this, we will show the existences of a $(p+1)$-simplex in $\Bun_A(\xi\oplus\varepsilon^p\oplus\varepsilon^2,\psi_{d+p+2};\ell_0)^\tau_\bullet$ which on the $p$th face agrees with the $(p+2)$-fold stabilisation of $\phi$, on the $(p+1)$st face with the image $\phi^F\in  \Bun_A(\xi\oplus \epsilon^p\oplus\epsilon^2,\psi_{d+p+2};\ell_0)_p^\tau$ of the $2$-fold stabilisation of the composition \eqref{equ:composition-appendix} and on the remaining faces with the basepoint. To this end, we consider the linear map
\[\map{\Lambda}{\bfR^p\times \bfR\times \bfR\times \bfR^p}{\bfR^p\times \bfR\times \bfR\times \bfR^p}{(x,u,v,y)}{\begin{cases}(y,u,v,x)&\text{if }p\text{ is even}\\
(y,v,u,x)&\text{if }p\text{ is odd},
\end{cases}},\]
which has determinant $1$ (this is the reason we introduced the additional $(\bfR\times\bfR)$-coordinate), so there exists a path $\gamma\colon [0,1]\ra\GL(\bfR^p\times \bfR\times \bfR\times\bfR^p)$ from the identity to $\Lambda$, which we may choose to be constant in a neighborhood of $[0,\epsilon]\cup[1-\epsilon,1]$. In terms of this path, we define a homotopy $H_t$ of bundle automorphisms of $\varepsilon^p\oplus\tau_{\Delta^p}$ covering the identity by
\[\varepsilon^p\oplus\varepsilon^2\oplus\tau_{\Delta^p}\xra{\id_{\bfR^p\times\bfR^2}\times F}\bfR^p\times \bfR^2\times \bfR^p\times\Delta^p\xra{\gamma_t\times\id_{\Delta^p}}\bfR^p\times \bfR^2\times \bfR^p\times\Delta^p\xra{\id_{\bfR^p\times \bfR^2}\times F^{-1}}\varepsilon^p\oplus\tau_{\Delta^p}.\] Using $H$, we define a homotopy $\widetilde{H}$ of bundle maps as the composition \[\xi\oplus\varepsilon^p\oplus\varepsilon^2\times\tau_{\Delta^p}\xlra{\id_{\xi}\times H_t}\xi\oplus\varepsilon^p\oplus\varepsilon^2\times\tau_{\Delta^p}\xra{\phi}\psi_{d}\oplus\varepsilon_p\oplus\varepsilon^2\times\tau_{\Delta^p}\xlra{\id_{\psi_d}\times H_t^{-1}}\psi_{d}\oplus\varepsilon_p\oplus\varepsilon^2\times\tau_{\Delta^p}\ra\psi_{d+p+2}\times\tau_{\Delta^p},\] where the last map is induced by the structure map of $\psi$. Going through the definitions, one checks that this is a homotopy from the stabilisation of $\phi$ to $\phi^F$ and that its underlying homotopy of maps of spaces is constantly $\bar{\phi}$. Using the canonical trivialisation of $\tau_{[0,1]}$, this homotopy gives rise to a bundle map $\widetilde{H}\colon \xi\oplus\varepsilon^p\oplus\varepsilon^2\times\tau_{\Delta^p\times [0,1]}\ra\psi_{d+p+2}\times\tau_{\Delta^p\times [0,1]}$ which one checks to descend uniquely to a dashed arrow making the diagram 
\[
\begin{tikzcd}
\xi\oplus \varepsilon^p\oplus\varepsilon^2\times \tau_{\Delta^p\times [0,1]}\arrow[r,"\widetilde{H}"]\arrow[d,"\id_{\xi\oplus \varepsilon^p\oplus\varepsilon^2}\times\mathrm{d}c",swap]&\psi_{d+p+2}\times \tau_{\Delta^p\times[0,1]}\arrow[d,"\id_{\xi\oplus \varepsilon^p\oplus\varepsilon^2}\times\mathrm{d}c"]\\
\xi\oplus \varepsilon^p\oplus\varepsilon^2\times \tau_{\Delta^{p+1}}\arrow[r,dashed,"\bar{H}"]&\psi_{d+p+2}\times \tau_{\Delta^{p+1}}
\end{tikzcd}
\]
commute, where $c$ is the map
\[\map{c}{\Delta^p\times [0,1]}{\Delta^{p+1}}{((x_0,\ldots,x_p),s)}{(x_0,\ldots,x_{p-1},s\cdot x_p,(1-s)\cdot x_{p})}\] whose derivative is surjective (though not fibrewise). The resulting bundle map $\bar{H}$ has the correct behaviour on all faces, so almost provides a $(p+1)$-simplex in $\Bun_A(\xi\oplus\varepsilon^p\oplus\varepsilon^2,\psi_{d+p+2};\ell_0)^\tau_\bullet$ as wished. The only problem is that it does not satisfy the collaring condition \eqref{equ:collarsquare} for $i=p$, but this can be rectified as follows: choosing a block diffeomorphism $\alpha\colon \Delta^{p+1}\ra\Delta^{p+1}$ which agrees with the identity on $\Delta^{p+1}_{p+1, \delta}$ for some $\delta>0$ (see \cref{sect:blockspaces} for the notation) and makes the diagram
\[
\begin{tikzcd}[row sep=0cm]
&\Delta^{p+1}\arrow[dd,"\alpha"]\\
{[0,\epsilon)\times \Delta^p \backslash\Delta^p_{p,{\delta'}}}\arrow[ur,"c_{p}",bend left=10]\arrow[dr,"c",bend right=10,swap]&\\
&\Delta^{p+1}
\end{tikzcd}
\]
commute for some $\delta'>\epsilon$, the composition
\[\xi\oplus \varepsilon^p\oplus\varepsilon^2\times \tau_{\Delta^{p+1}}\xlra{\id_{\xi\oplus\varepsilon^p\oplus\varepsilon^2}\times \mathrm{d}\alpha}\xi\oplus \varepsilon^p\oplus\varepsilon^2\times \tau_{\Delta^{p+1}}\xlra{\bar{H}}\psi_{d+p+2}\times \tau_{\Delta^{p+1}}\xlra{\id_{\xi\oplus\varepsilon^p\oplus\varepsilon^2}\times \mathrm{d}\alpha^{-1}}\psi_{d+p+2}\times \tau_{\Delta^{p+1}}\] defines a $(p+1)$-simplex as required. This finishes the proof of surjectivity of the map on homotopy groups and injectivity follows from a relative version of the argument.
\end{proof}

\section{A lemma on symplectic derivations}This appendix serves to record a minor generalisation of Proposition 3.9 in \cite{BerglundMadsen}. We adopt the notation and conventions introduced in Sections~\ref{sec:gradings} and \ref{sec:lie-algs} and fix a principal ideal domain $R$ with $2\in R^\times$, an integer $k\in \bfZ$, and a graded $R$-module $V$ that is degreewise free, together with map $\lambda\colon V\otimes V\ra R[0]$ of degree $-k$ such that the adjoint map $s^{-k} V\ra  V^\vee$ given by mapping $v$ to $\lambda(v,-)$ is an isomorphism. We moreover assume that $\lambda$ is graded antisymmetric, i.e.\,that $\lambda(x,y)=-(-1)^{|x||y|}\lambda(y,x)$ holds for all $x,y\in V$. 

\begin{ex}\label{ex:appendix}Let $M$ be a compact $R$-oriented $d$-manifold whose $R$-homology is free. If $\partial M\simeq S^{d-1}$, then the desuspension $s^{-1}\widetilde{\oH}_*(M;R)$  of the reduced homology qualifies as an example of $V$ as above where $k=d-2$ and $\lambda$ is induced by the intersection form. Concretely, $\lambda$ is defined to vanish on elements $s^{-1}x\otimes s^{-1}y$ unless $|x|+|y|=d$ in which case it is given by $\lambda(s^{-1}x\otimes s^{-1}y)=(-1)^{|x|}\langle D(x)\cup D(y),[M,\partial M]\rangle$, where $D(-)$ denotes the Poincaré duality isomorphism.\end{ex}

Following \cite[Sec.\,3.5]{BerglundMadsen}, given a homogenous basis $\{a_i\}$ of $V$, we use the dual basis  $\{a_i^\#\}$ determined by $\lambda(a_i,a_j^\#)=\delta_{ij}$ to define an element $\omega\coloneq \sum_ia_i^\#\otimes a_i\in V\otimes V$. This element can be seen to be independent of the chosen basis (see p.\,91 loc.cit.\,for a proof for $R=\bfQ$, which generalises) and to be homogenous of degree $k$, since $|a_i^{\#}|=k-|a_i|$. We identify $[V,V]$ with a subspace of $V\otimes V$ via the antisymmetrisation map $[V,V]\ra V\otimes V$ mapping $[x,y]$ to $x\otimes y-(-1)^{|x||y|}y\otimes x$, which is injective since its composition with the bracket is multiplication by $2\in R^\times$. Under this identification, one verifies 
\begin{equation}\label{equ:canonical-element-app}\textstyle{\omega=\tfrac{1}{2}\sum_i[a_i^{\#},a_i]\in [V,V]}.\end{equation}
as in p.\,91 loc.cit. Given a dg Lie algebra $K$ over $R$ and a map $f\colon V\ra K$ of graded $R$-modules, we consider the diagram of graded $R$-modules
\begin{equation}\label{equ:diagram-appendix}
\begin{tikzcd}
\Der^f(\bfL(V),K)\arrow[r,"\ev_{\omega}"]\arrow["\cong",d,swap]&s^{-k}K\\
s^{-k}K\otimes V\arrow[ur,"{\left[-,f(-)\right]}",swap]&
\end{tikzcd}
\end{equation}
whose vertical isomorphism is given by the composition
\[\Der^f(\bfL(V),K)\xlra{\cong}\Hom(V,K)\xlra{\cong} K\otimes V^\vee\xlra{\cong}K\otimes s^{-k}V=s^{-k}K\otimes V,\]
where the first map is given by restriction to generators, the second map is the canonical isomorphism, and the third map is induced by the inverse of the adjoint map of $\lambda$.

\medskip

The following lemma is a straight-forward extension of Proposition 3.9 of \cite{BerglundMadsen}, which corresponds to the case $R=\bfQ$, $K=\bfL(V)$, and $f=\inc_{V\subset \bfL(V)}$.

\begin{lem}\label{lem:appcommutative}The diagram \eqref{equ:diagram-appendix} is commutative.
\end{lem}
\begin{proof}
Following the beginning of the proof of \cite[Prop.\,3.9]{BerglundMadsen}, one sees that a derivation $\theta\in \Der^f(\bfL(V),K)$ evaluates on the special element $\omega$ as $\theta(\omega)=\sum_i[\theta(a_i^{\#}),f(a_i)]$. We will use this to show the equivalent claim that the inverse of the vertical map in \eqref{equ:diagram-appendix} followed by $\ev_\omega$ agrees with $[-,f(-)]$. This inverse assigns an element $k\otimes v\in s^{-k}K\otimes V$ the unique $f$-derivation $\theta_{k,v}$ that extends the linear map $\lambda(v,-)\cdot k\colon V\ra K$. Using that for $v\in V$, we have $v=\sum_i\lambda(v,a_i^{\#})\cdot a_i$ by the definition of the dual basis, we compute
\[\textstyle{\theta_{k,v}(\omega)=\sum_i[\theta_{k,v}(a_i^{\#}),f(a_i)]=\sum_i[\lambda(v,a_i^{\#})\cdot k,f(a_i)]=[k,f(\sum_i\lambda(v,a_i^{\#})\cdot a_i)]=[k,v]},\]
which implies the claim.
\end{proof}

\vspace{-0.4cm}

\bibliographystyle{amsalpha}
\bibliography{literature}
\vspace{0.2cm}

\end{document}